\documentclass[11pt]{amsart}

\usepackage{  
  amsmath,
  amsfonts,
  amssymb,
  mvsymb,   
  graphicx, 
  times,
  color,
  mathptmx, 
  hyphenat}

\input xy
\xyoption{all}

\newcommand{\cp}[1]{\bC\mbox{P}^{#1} }

\newcommand{\brak}[1]{\langle #1\rangle}
\newcommand{\n}{\noindent}
\newcommand{\foam}{\mathbf{Foam}_{N}}
\newcommand{\PF}{\mathbf{Pre-foam}}
\newcommand{\sln}{\mathfrak{sl}(N)}
\newcommand{\V}{\mathbf{Vect}_{\bZ}}
\DeclareMathOperator{\sk}{s_\gamma}

\newcommand{\figins}[3] 
{\raisebox{#1pt}{\includegraphics[height=#2 in]{#3.eps}}}
\newcommand{\figwins}[3] 
{\raisebox{#1pt}{\includegraphics[width=#2 in]{#3.eps}}}

\newtheorem{thm}{Theorem}[section]
\newtheorem{lem}[thm]{Lemma}
\newtheorem{cor}[thm]{Corollary}
\newtheorem{prop}[thm]{Proposition}

\theoremstyle{definition}
\newtheorem{defn}[thm]{Definition}

\newcommand{\bZ}{\mathbb{Z}}
\newcommand{\bQ}{\mathbb{Q}}

\newcommand{\bC}{\mathbb{C}}

\newcommand{\cF}{\mathcal{F}}

\newcommand{\ra}{\rightarrow}
\newcommand{\lra}{\longrightarrow}

\newcommand{\qbin}[2]{\left[{{#1}\atop {#2}}\right]}

\def\G{\mathcal{G}}

\def\deg{\mathop{\rm deg}}
\def\rank{\mathop{\rm rank}}

\def\sgn{\mathop{\rm sgn}}
\def\Tr{\mathop{\rm Tr}}

\def\End{\mathop{\rm End}}
\def\Ext{\mathop{\rm Ext}}
\def\Id{\mathop{\rm Id}}
\def\qdim{\mathop{\rm qdim}}

\title{$\sln$-link homology ($N\geq 4$) using foams and the Kapustin-li formula}
\author{Marco Mackaay}
\address{Departamento de Matem\'{a}tica\\ Universidade do Algarve\\ 
Campus de Gambelas\\ 8005-139 Faro\\ Portugal and CAMGSD\\Instituto Superior T\'{e}cnico\\ Avenida Rovisco Pais\\ 
1049-001 Lisboa\\ Portugal}
\email{mmackaay@ualg.pt}
\author{Marko Sto\v si\'c }
\address{Instituto de Sistemas e Rob\'{o}tica and CAMGSD\\Instituto Superior T\'{e}cnico\\ Avenida Rovisco Pais\\ 
1049-001 Lisboa\\ Portugal}
\email{mstosic@math.ist.utl.pt}
\author{Pedro Vaz}
\address{Departamento de Matem\'{a}tica\\ Universidade do Algarve\\ 
Campus de Gambelas\\ 8005-139 Faro\\ Portugal and  
CAMGSD\\Instituto Superior T\'{e}cnico\\ Avenida Rovisco Pais\\ 
1049-001 Lisboa\\ Portugal}
\email{pfortevaz@ualg.pt}

\date{}

\begin{document}

\begin{abstract}
We use foams to give a topological construction of a rational link homology categorifying 
the $\mathfrak{sl}(N)$ link invariant, for $N\geq 4$. To evaluate 
closed foams we use the Kapustin-Li formula adapted to foams by 
Khovanov and Rozansky~\cite{KR-LG}. 
We show that for any link our homology is isomorphic to the 
Khovanov-Rozansky~\cite{KR} homology.
\end{abstract}

\maketitle

\section{Introduction}
\label{sec:intro}

In~\cite{MOY} Murakami, Ohtsuki and Yamada (MOY) developed a graphical calculus for the 
$\mathfrak{sl}(N)$ link polynomial. 
In~\cite{khovanovsl3} Khovanov categorified the $\mathfrak{sl}(3)$ polynomial 
using singular cobordisms between webs called foams. Mackaay and Vaz~\cite{mackaay-vaz} 
generalized Khovanov's results to obtain the universal $\mathfrak{sl}(3)$ integral link 
homology, following an approach similar to the one adopted by Bar-Natan~\cite{bar-natancob} 
for the original $\mathfrak{sl}(2)$ integral Khovanov homology. 
In~\cite{KR} Khovanov and Rozansky (KR) 
defined a rational link homology which categorifies the $\mathfrak{sl}(N)$ 
link polynomial using the theory of 
matrix factorizations. 

In this paper we use foams, as in~\cite{bar-natancob, khovanovsl3, mackaay-vaz}, for an 
almost completely combinatorial topological construction of 
a rational link homology categorifying the 
$\mathfrak{sl}(N)$ link 
polynomial. Our theory is functorial under link cobordisms. 
Khovanov had to modify 
considerably his original setting for the construction of $\mathfrak{sl}(2)$ 
link homology in order to produce his $\mathfrak{sl}(3)$ link homology. It required the 
introduction of singular cobordisms with a particular type of singularity. The jump from 
$\mathfrak{sl}(3)$ to $\mathfrak{sl}(N)$, for $N>3$, requires the introduction of a new type of 
singularity. The latter is needed for proving invariance under the third Reidemeister move. 
Furthermore the combinatorics involved in establishing certain identities gets much harder 
for arbitrary $N$. The theory of symmetric polynomials, in particular Schur polynomials, is 
used to handle that problem. 

Our aim was to find a 
combinatorial topological definition of Khovanov-Rozansky 
link homology. Such a definition is desirable for several 
reasons, the main one being that it might help to find a 
good way to compute the Khovanov-Rozansky link homology. 
Unfortunately the construction that we 
present in this paper is not completely combinatorial. 
The introduction of the 
new singularities makes it much harder to evaluate closed 
foams and we do not know how to do it combinatorially. 
Instead we use the Kapustin-Li formula~\cite{KL}, adapted by 
Khovanov and Rozansky~\cite{KR-LG}\footnote{We 
thank M Khovanov for suggesting that we try to use the 
Kapustin-Li formula.}. A positive side-effect is that it 
allows us to show that for any link 
our homology is isomorphic to Khovanov and Rozansky's. 

Although we have not completely achieved our final goal, we believe 
that we have made good progress towards it. In 
Propositions~\ref{prop:principal rels1} and 
\ref{prop:principal rels2} we derive a small 
set of relations on foams which we show to be 
sufficient to 
guarantee that our link homology is homotopy invariant  
under the Reidemeister moves. By deriving these relations 
from the Kapustin-Li formula we prove that these relations 
are consistent. However, in order to get a purely 
combinatorial construction we would have to show that 
they are also sufficient for the evaluation of closed foams, 
or, equivalently, that they generate the kernel of the 
Kapustin-Li formula.  
We conjecture that this holds true, but so far our attempts 
to prove it have failed. It would be very interesting to 
have a proof of this conjecture, not just because it 
would show that our method is completely combinatorial, 
but also because our theory could then be used to prove that 
other constructions, using different combinatorics, 
representation theory or symplectic/complex geometry, give 
functorial link homologies equivalent to 
Khovanov and Rozansky's. So far we can only conclude that any other way of 
evaluating closed foams which satifies the same relations as in 
Propositions~\ref{prop:principal rels1} 
and \ref{prop:principal rels2} gives rise to a functorial link homology 
which categorifies the $\mathfrak{sl}(N)$ link polynomial. 
We conjecture that such a link homology is equivalent to the one 
presented in this paper and therefore to Khovanov and Rozansky's.  

In section~\ref{sec:slN} we recall some basic facts about 
the 
$\mathfrak{sl}(N)$-link polynomials. In section~\ref{sec:partial flags} we recall some 
basic facts about Schur polynomials and the cohomology of partial flag 
varieties. In section~\ref{sec:pre-foam} we define pre-foams and their 
grading. In section~\ref{sec:KL} we explain the Kapustin-Li formula for 
evaluating closed pre-foams and compute the spheres and the theta-foams. 
In section~\ref{sec:foamN} we derive a set of basic relations in the 
category $\foam$, which is the quotient of the category of pre-foams by 
the kernel of the Kapustin-Li evaluation. In section~\ref{sec:invariance} 
we show that our link homology complex is homotopy invariant under the 
Reidemeister moves. In section~\ref{sec:functoriality} we show that our link 
homology complex extends to a link homology functor. In 
section~\ref{sec:taut-functor} we show that our link homology, obtained from 
our link homology complex using the tautological functor, 
categorifies the $\mathfrak{sl}(N)$-link polynomial and that it is isomorphic to the 
Khovanov-Rozansky link homology.

\section{Graphical calculus for the $\mathfrak{sl}(N)$ polynomial}
\label{sec:slN}

In this section we recall some facts about the graphical calculus for $\mathfrak{sl}(N)$.
The $\mathfrak{sl}(N)$ link polynomial is defined by the skein relation
$$
q^{N}
P_N(\undercrossing)
-q^{-N}
P_N(\overcrossing)
=(q-q^{-1})
P_N(\orsmoothing),
$$
and its value for the unknot, which we take to be equal to $[N]=(q^N-q^{-N})/(q-q^{-1})$. 
Let $D$ be a diagram of a link $L\in S^3$ with $n_+$ positive crossings and $n_-$ negative crossings. 
\begin{figure}[h!]
$$\begin{array}{cclcccl}
\text{positive: } \ 
\figins{-9}{0.3}{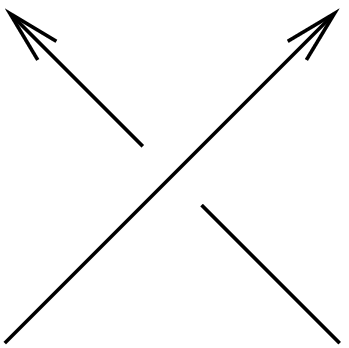} &=&   
\figins{-9}{0.3}{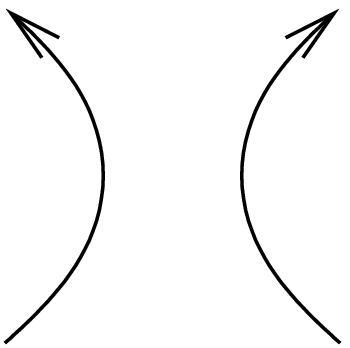}  -   q  
\figins{-9}{0.3}{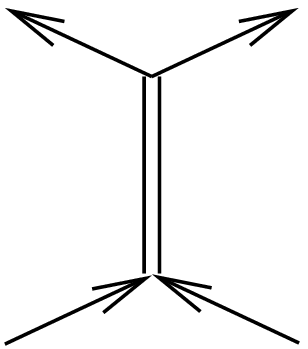} 
&\quad&
\text{negative: } \ 
\figins{-9}{0.3}{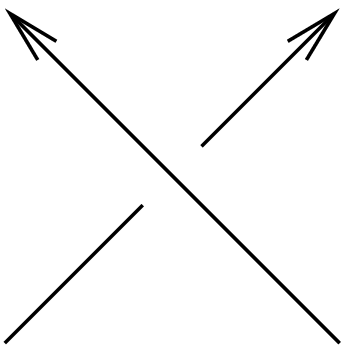} &=& \, 
\figins{-9}{0.3}{dbedge} -  q \
\figins{-9}{0.3}{orsmoothing} \vspace{1ex} 
\\
&& \hspace{1.7ex} 0 \hspace{6.4ex}  1 &&&& \hspace{1.7ex} 0 \hspace{7ex} 1
\end{array}$$
\caption{Positive and negative crossings and their 0 and 1-flattening}
\label{fig:flatten}
\end{figure}
Following an approach based on MOY's state sum model~\cite{MOY} 
we can compute $P_N(D)$ by flattening each crossing of $D$ in two possible ways, 
as shown in Figure~\ref{fig:flatten}, 
where we also show our convention for positive and negative crossings. 
Each complete flattening of $D$ is an example of a {\em web}: a trivalent graph with three  
types of edges: {\em simple}, {\em double} and {\em marked} edges. 
Only the simple edges are equipped with an orientation. 
Near each vertex at most one edge can be distinguished with a $(*)$, 
as in Figure~\ref{fig:vertices}. Note that a complete flattening of $D$ never has marked 
edges, but we will need the latter for webs that show up in the proof of invariance under 
the third Reidemeister move.  
\begin{figure}[h!]
$$\xymatrix@R=1.2mm{
\figins{0}{0.32}{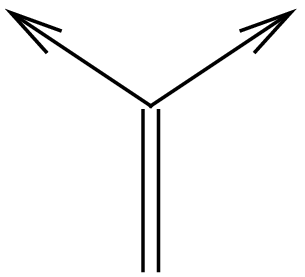} &
\figins{0}{0.32}{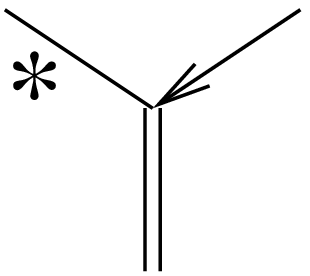} &
\figins{0}{0.32}{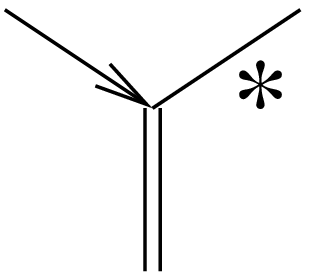}  \\
\figins{0}{0.32}{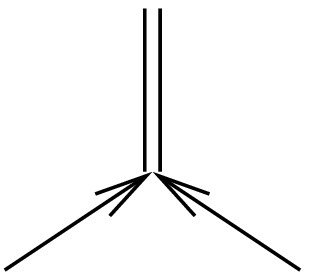}  &
\figins{0}{0.32}{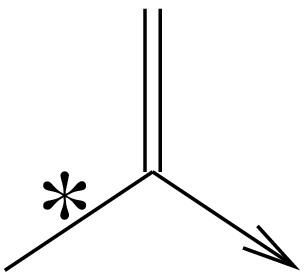} &
\figins{0}{0.32}{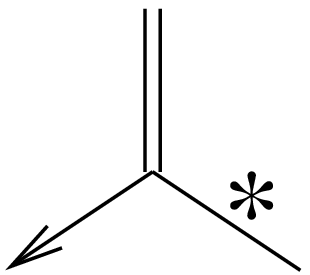}
}$$
\caption{Vertices}
\label{fig:vertices}
\end{figure}

Simple edges correspond to edges labelled 1, double edges to edges labelled 2 and 
marked simple edges to edges labelled 3 in~\cite{MOY}, 
where edges carry labels from 1 to $N-1$ and label $j$ is associated to the $j$-th exterior power of the 
fundamental representation of $\mathfrak{sl}(N)$.

We call a planar trivalent graph generated by the vertices and edges defined above a \emph{web}. 
Webs can contain closed plane loops (simple, double or marked). The \emph{MOY web moves} 
in Figure~\ref{fig:moy} provide a recursive way of assigning to each web $\Gamma$ that only contains 
simple and double edges a polynomial in 
$\bZ[q,q^{-1}]$ with positive coefficients, which we call $P_N(\Gamma)$. There are more general 
web moves, which allow for the evaluation of arbitrary webs, but we do not need them here. Note that 
a complete flattening of a link diagram only contains simple and double edges. 

\begin{figure}[h]
$$\bigcirc=[N],\quad
\figins{-3.1}{0.16}{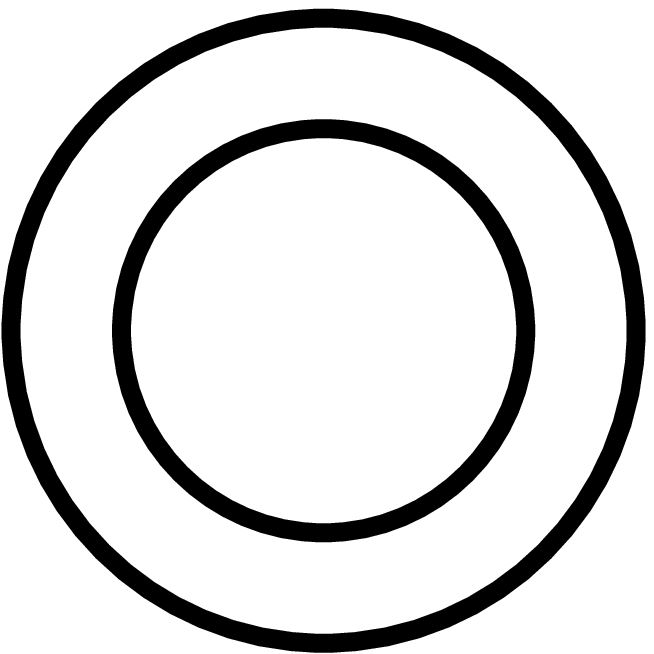}=\qbin N 2$$

$$
\figins{-8}{0.3}{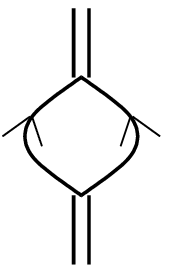} = [2]\ 
\figins{-8}{0.3}{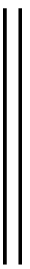}\ , \qquad
\figins{-8}{0.3}{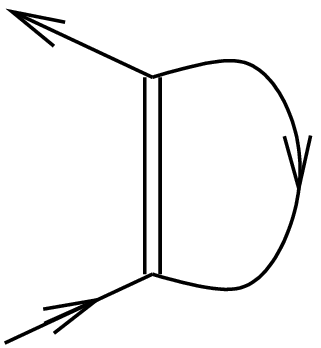} =[N-1]\ 
\figins{-8}{0.3}{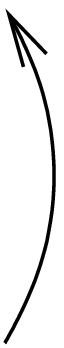}
$$

$$
\figins{-8}{0.3}{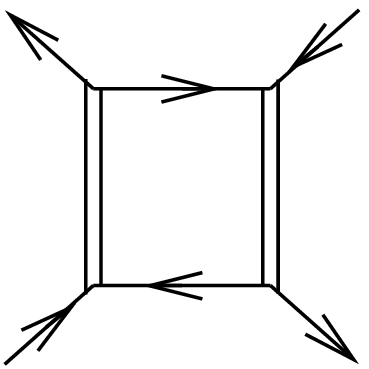} =
\figins{-8}{0.3}{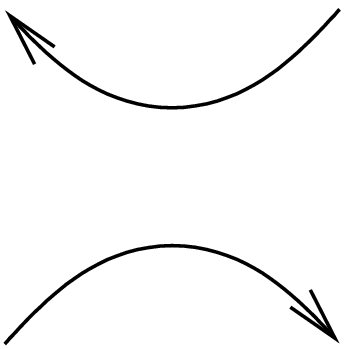} + [N-2]\ 
\figins{-8}{0.3}{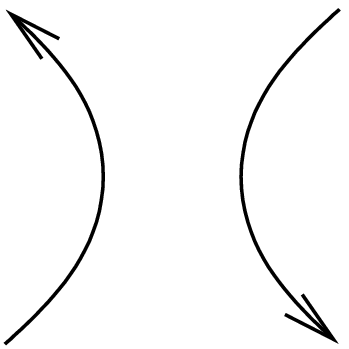}
$$

$$
\figins{-17}{0.55}{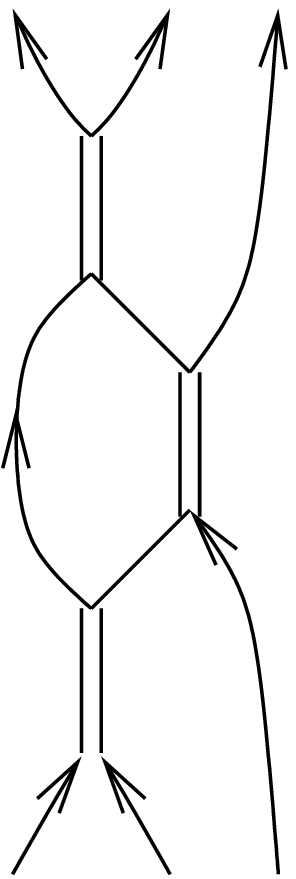} +
\figins{-17}{0.55}{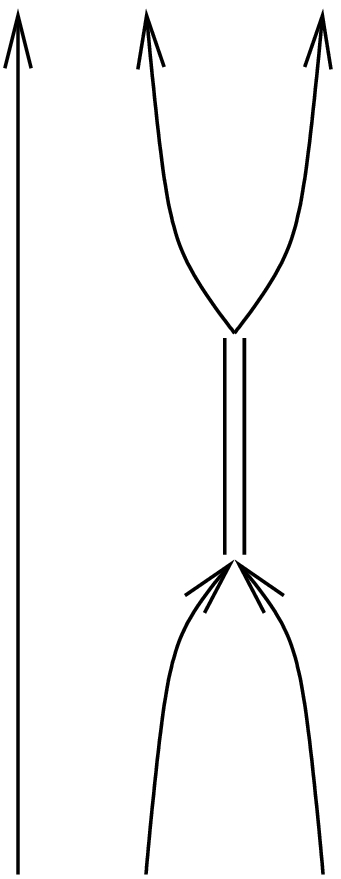} =
\figins{-17}{0.55}{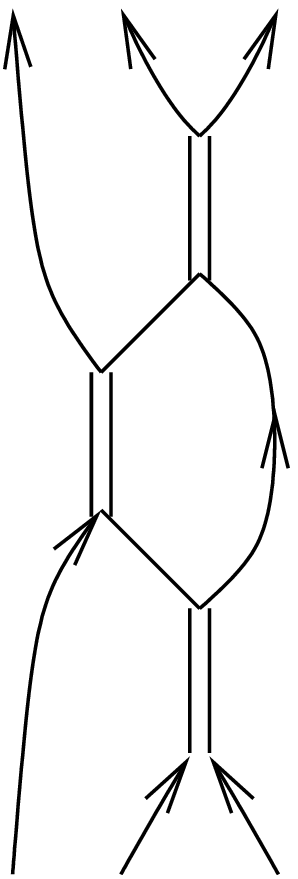}+
\figins{-17}{0.55}{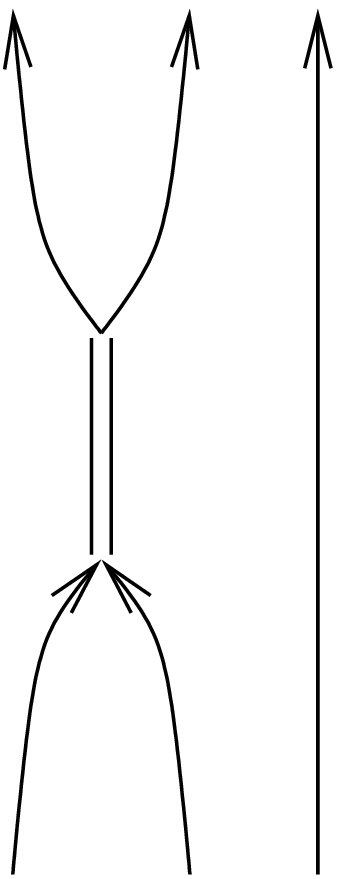}
$$
\caption{MOY web moves}
\label{fig:moy}
\end{figure}

Consistency of the relations in Figure~\ref{fig:moy} is shown in~\cite{MOY}.

Finally let us define the $\mathfrak{sl}(N)$ link polynomial. For any $i$ let $\Gamma_i$ denote a 
complete flattening of $D$. Then 
$$P_N(D)=(-1)^{n_-}q^{(N-1)n_+ - Nn_-}\sum_iq^{|i|}P_N(\Gamma_i),$$
where $|i|$ is the number of 1-flattenings in $\Gamma_i$, the sum being over all possible flattenings of $D$. 


\section{Schur polynomials and the cohomology of partial flag varieties}
\label{sec:partial flags}
In this section we recall some basic facts about Schur polynomials and 
the cohomology of partial flag varieties which we need in the rest of 
this paper.  
\subsection{Schur polynomials}
\label{sec:Schur}
A nice basis for homogeneous symmetric polynomials is given by the Schur 
polynomials. If $\lambda=(\lambda_1,\ldots,\lambda_k)$ is a partition such that $\lambda_1\ge\ldots\ge \lambda_k\ge 0$, then the Schur polynomial $\pi_{\lambda}(x_1,\ldots,x_k)$ is given by the following expression:
\begin{equation}
\pi_{\lambda}(x_1,\ldots,x_k)=\frac{|x_i^{\lambda_j+k-j}|}{\Delta},
\label{sur}
\end{equation}
where $\Delta=\prod_{i<j}(x_i-x_j)$, and by $|x_i^{\lambda_j+k-j}|$, we have denoted the determinant of the $k\times k$ matrix whose $(i,j)$ entry is equal to $x_i^{\lambda_j+k-j}$. Note that the elementary symmetric polynomials 
are given by $\pi_{1,0,0,\ldots,0}, \pi_{1,1,0,\ldots,0},\ldots,  
\pi_{1,1,1,\ldots,1}$. There are multiplication rules for the Schur polynomials which show that 
any $\pi_{\lambda_1,\lambda_2,\ldots,\lambda_k}$ can be expressed in terms of the 
elementary symmetric polynomials. 

If we do not specify the variables of the Schur polynomial $\pi_{\lambda}$, we will assume that these are exactly $x_1,\ldots,x_k$, with $k$ being the length of $\lambda$, i.e. $$\pi_{\lambda_1,\ldots,\lambda_k}:=\pi_{\lambda_1,\ldots,\lambda_k}(x_1,\ldots,x_k).$$ 

In this paper we only use Schur polynomials of two and 
three variables. In the case of two variables, the Schur polynomials are 
indexed by pairs of nonnegative integers $(i,j)$, such that $i\ge j$, and~(\ref{sur}) becomes
$$\pi_{i,j}=\sum_{\ell =j}^i{x_1^\ell x_2^{i+j-\ell}}.$$
Directly from {\em Pieri's formula} we obtain the following multiplication rule for the Schur polynomials in two variables:
\begin{equation}
\pi_{i,j}\pi_{a,b}=\sum {\pi_{x,y}},
\label{mnoz2}
\end{equation}
where the sum on the r.h.s. is over all indices $x$ and $y$ such that $x+y=i+j+a+b$ and $a+i\ge x\ge \max(a+j,b+i)$. Note that this implies $\min(a+j,b+i)\ge y\ge b+j$. Also, we shall write $\pi_{x,y}\in\pi_{i,j}\pi_{a,b}$ if $\pi_{x,y}$ belongs to the sum on the r.h.s. of (\ref{mnoz2}). Hence, we have that $\pi_{x,x}\in\pi_{i,j}\pi_{a,b}$ iff $a+j=b+i=x$ and $\pi_{x+1,x}\in\pi_{i,j}\pi_{a,b}$ iff $a+j=x+1$, $b+i=x$ or $a+j=x$, $b+i=x+1$.

We shall need the following combinatorial result which expresses the 
Schur polynomial in three variables as a combination of Schur polynomials 
of two variables.

For $i\ge j\ge k \ge 0$, and the triple $(a,b,c)$ of nonnegative integers, we define
$$(a,b,c)\sqsubset(i,j,k),$$
if $a+b+c=i+j+k$, $i\ge a \ge j$ and $j\ge b \ge k$. We note that this implies that $i\ge c\ge k$, and hence $\max\{a,b,c\}\le i$.
\begin{lem}\label{lem1}
$$\pi_{i,j,k}(x_1,x_2,x_3)=\sum_{(a,b,c)\sqsubset(i,j,k)}{\pi_{a,b}(x_1,x_2)x_3^c}.$$
\end{lem}

\begin{proof}

From the definition of the Schur polynomial, we have
$$
\pi_{i,j,k}(x_1,x_2,x_3)={\frac{(x_1x_2x_3)^k}{(x_1-x_2)(x_1-x_3)(x_2-x_3)}}\left|
\begin{array}{ccc}
x_1^{i-k+2} & x_1^{j-k+1} & 1 \\
x_2^{i-k+2} & x_2^{j-k+1} & 1 \\
x_3^{i-k+2} & x_3^{j-k+1} & 1
\end{array}
\right|.$$
After subtracting the last row from the first and the second one of the last determinant, we obtain
$$
\pi_{i,j,k}={\frac{(x_1x_2x_3)^k}{(x_1-x_2)(x_1-x_3)(x_2-x_3)}}\left|
\begin{array}{cc}
x_1^{i-k+2}-x_3^{i-k+2} & x_1^{j-k+1}-x_3^{j-k+1} \\
x_2^{i-k+2}-x_3^{i-k+2} & x_2^{j-k+1}-x_3^{j-k+1} 
\end{array}
\right|,$$
and so
$$\pi_{i,j,k}={\frac{(x_1x_2x_3)^k}{x_1-x_2}}\left|
\begin{array}{cc}
\sum_{m=0}^{i-k+1} x_1^{m}x_3^{i-k+1-m} & \sum_{n=0}^{j-k} x_1^{n}x_3^{j-k-n}\\
\sum_{m=0}^{i-k+1} x_2^{m}x_3^{i-k+1-m} & \sum_{n=0}^{j-k} x_2^{n}x_3^{j-k+n}
\end{array}
\right|.$$
Finally, after expanding the last determinant we obtain
\begin{equation}
\pi_{i,j,k}=\frac{(x_1x_2x_3)^k}{x_1-x_2}\sum_{m=0}^{i-k+1}\sum_{n=0}^{j-k}{(x_1^mx_2^n-x_1^nx_2^m)x_3^{i+j-2k+1-m-n}}.
\label{pf1}
\end{equation}
We split the last double sum into two: the first one when 
$m$ goes from $0$ to $j-k$, denoted by $S_1$, and the other 
one when $m$ goes from $j-k+1$ to $i-k+1$, denoted by $S_2$. 
To show that $S_1=0$, we split the double sum further into 
three parts: when $m<n$, $m=n$ and $m>n$. Obviously, each 
summand with $m=n$ is equal to $0$, while the summands of the sum for $m<n$ are exactly the opposite of the summands of the sum for $m>n$. Thus, by replacing only $S_2$ instead of the double sum in (\ref{pf1}) and after rescaling the indices $a=m+k-1$, $b=n+k$, we get
\begin{eqnarray*}
\pi_{i,j,k}&=&\frac{(x_1x_2x_3)^k}{x_1-x_2}\sum_{m=j-k+1}^{i-k+1}\sum_{n=0}^{j-k}{(x_1^mx_2^n-x_1^nx_2^m)x_3^{i+j-2k+1-m-n}} \\
&=&\sum_{a=j}^i\sum_{b=k}^j{\pi_{a,b}x_3^{i+j+k-a-b}}=\sum_{(a,b,c)\sqsubset(i,j,k)}\pi_{a,b}x_3^c,
\end{eqnarray*}
as wanted.  
\end{proof}

Of course there is a multiplication rule for three-variable Schur 
polynomials which is compatible with (\ref{mnoz2}) and the lemma above, 
but we do not want to discuss it here. For details see~\cite{Fulton-Harris}.

\subsection{The cohomology of partial flag varieties}

In this paper the rational cohomology rings of partial flag varieties 
play an essential role. The partial flag variety $Fl_{d_1,d_2,\ldots,d_l}$, 
for $1\le d_1<d_2<\ldots<d_l=N$, is defined by 
$$Fl_{d_1,d_2,\ldots,d_l}=\{V_{d_1}\subset V_{d_2}\subset\ldots\subset 
V_{d_l}=\bC^N|\dim(V_i)=i\}.$$
A special case is $Fl_{k,N}$, the Grassmanian 
variety of all $k$-planes in $\bC^N$, also denoted $\G_{k,N}$. The dimension of the partial 
flag variety is given by
$$\dim Fl_{d_1,d_2,\ldots,d_l}=N^2-\sum_{i=1}^{l-1}(d_{i+1}-d_{i})^2-d_1^2.$$
The rational cohomology rings of the partial flag varieties are well known 
and we only recall those facts that we need in this paper. 
\begin{lem}
$H^*(\G_{k,N})$ is isomorphic to the vector space generated by 
all $\pi_{i_1,i_2,\ldots,i_k}$ modulo the relations    
\begin{equation}
\pi_{N-k+1,0,\ldots,0}=0,\quad\pi_{N-k+2,0,\ldots,0}=0,\quad \ldots\quad,\
 \pi_{N,0,\ldots,0}=0,
\label{gras}
\end{equation}
where there are exactly $k-1$ zeros in the multi-indices of the Schur 
polynomials. 
\end{lem}
\noindent A consequence of the multiplication rules for Schur polynomials is 
that
\begin{cor} The Schur polynomials $\pi_{i_1,i_2,\ldots,i_k}$, 
for $N-k\geq i_1\geq i_2\geq \ldots\geq i_k\geq 0$, form a basis 
of $H^*(\G_{k,N})$
\end{cor}
\noindent Thus, the dimension of $\G_{k,N}$ is $N\choose k$, and up to a 
degree shift, 
its quantum dimension (or graded dimension) is $\left[{N\atop k}\right]$.

Another consequence of the multiplication rules is that 
\begin{cor} The Schur polynomials $\pi_{1,0,0,\ldots,0}, 
\pi_{1,1,0,\ldots,0},\ldots,\pi_{1,1,1,\ldots,1}$ (the elementary 
symmetric polynomials) generate $H^*(\G_{k,N})$ as a ring. 
\end{cor}

Furthermore, we can introduce a non-degenerate trace form on 
$H^*(\G_{k,N})$ 
by giving its values on the basis elements 
\begin{equation}
\epsilon(\pi_{\lambda})=\left\{\begin{array}{ll}
(-1)^{\lfloor\frac{k}{2}\rfloor}, \quad\lambda=(N-k,\ldots,N-k)\\
0, \quad\textrm{else} 
\end{array}\right.
.\label{trag}
\end{equation}
This makes $H^*(\G_{k,N})$ into a commutative Frobenius algebra. 
One can compute the basis dual to $\left\{\pi_{\lambda}\right\}$ in 
$H^*(\G_{k,N})$, 
with respect to $\epsilon$. It is given by    
\begin{equation}
\label{eq:db}
\hat{\pi}_{\lambda_1,\ldots,\lambda_k}=(-1)^{\lfloor\frac{k}{2}\rfloor}
\pi_{N-k-\lambda_k,\ldots,N-k-\lambda_1}.
\end{equation}

We can also express the cohomology rings of the partial 
flag varieties $Fl_{1,2,N}$ and $Fl_{2,3,N}$ in terms of Schur polynomials. 
Indeed, we have 
$$H^*(Fl_{1,2,N})=\bQ[x_1,x_2]/\langle \pi_{N-1,0},\pi_{N,0}\rangle,$$
$$H^*(Fl_{2,3,N})=\bQ[x_1+x_2,x_1x_2,x_3]/\langle \pi_{N-2,0,0},
\pi_{N-1,0,0},\pi_{N,0,0}\rangle.$$

The natural projection map $p_1:Fl_{1,2,N}\to\G_{2,N}$ induces 
$$p^*_1:H^*(\G_{2,N})\to H^*(Fl_{1,2,N}),$$
which is just the inclusion of the polynomial rings. Analogously, 
the natural 
projection $p_2:Fl_{2,3,N}\to\G_{3,N}$, 
induces 
$$p^*_2:H^*(\G_{3,N})\to H^*(Fl_{2,3,N}),$$
which is also given by the inclusion of the polynomial rings.

\section{Pre-foams}
\label{sec:pre-foam}

In this section we begin to define the foams we will work with. 
The philosophy behind these foams will be explained in 
section \ref{sec:KL}. 
To categorify the $\mathfrak{sl}(N)$ link polynomial we need singular 
cobordisms with two types of singularities. The basic examples are given in 
Figure~\ref{fig:elemfoams}. These foams are composed of three types of facets: simple, double and 
triple facets. The double facets are coloured and the triple facets are marked 
to show the difference.
\begin{figure}[h]
\figins{0}{0.7}{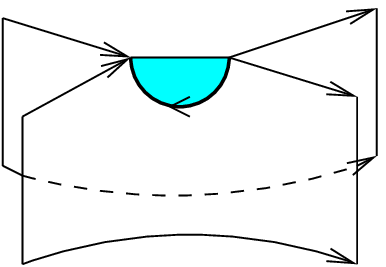} \qquad
\figins{0}{0.7}{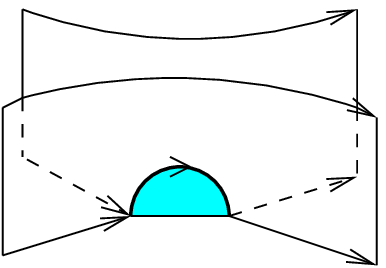} \qquad
\figins{-2}{0.8}{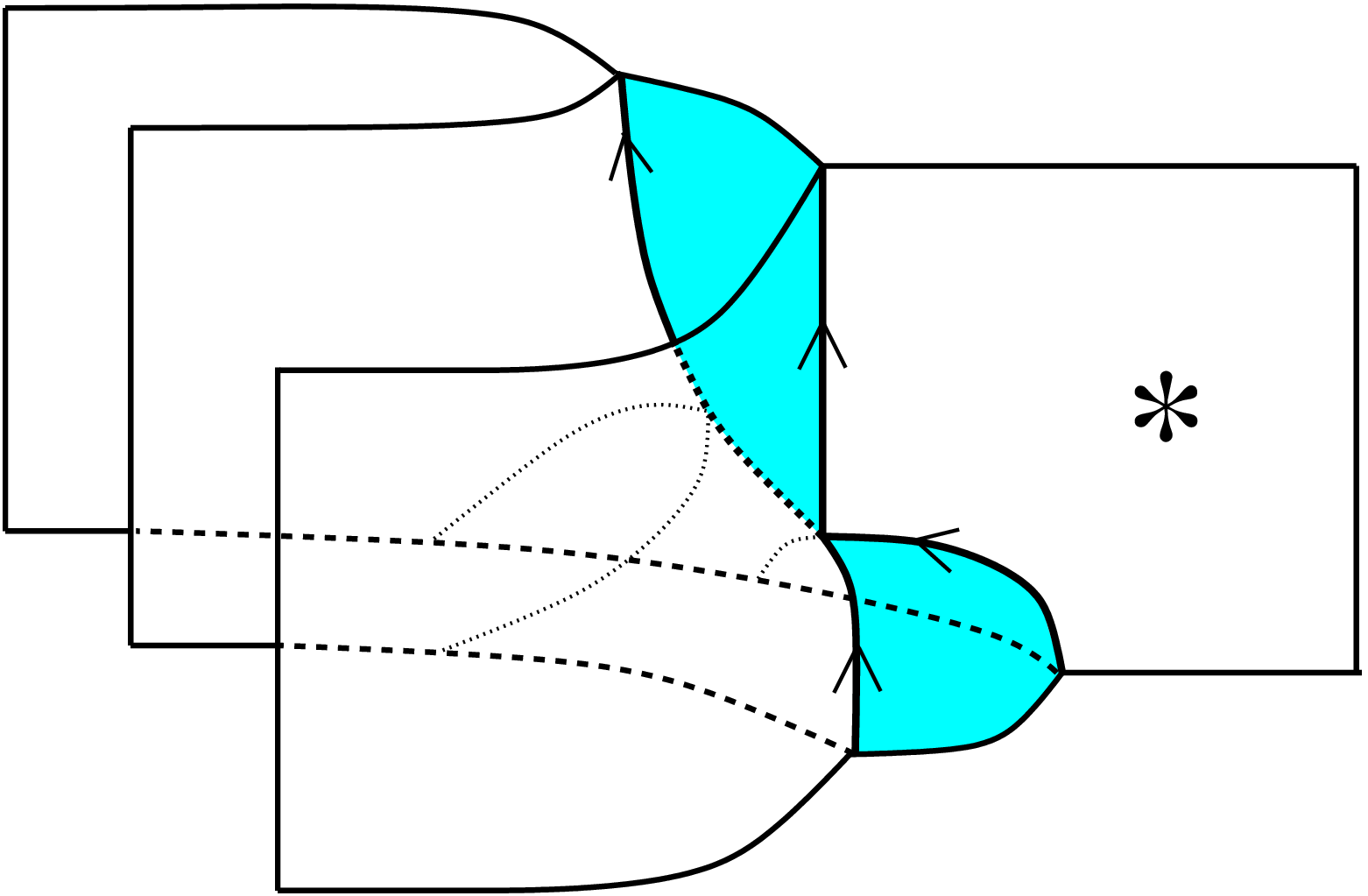} \qquad
\figins{-2}{0.8}{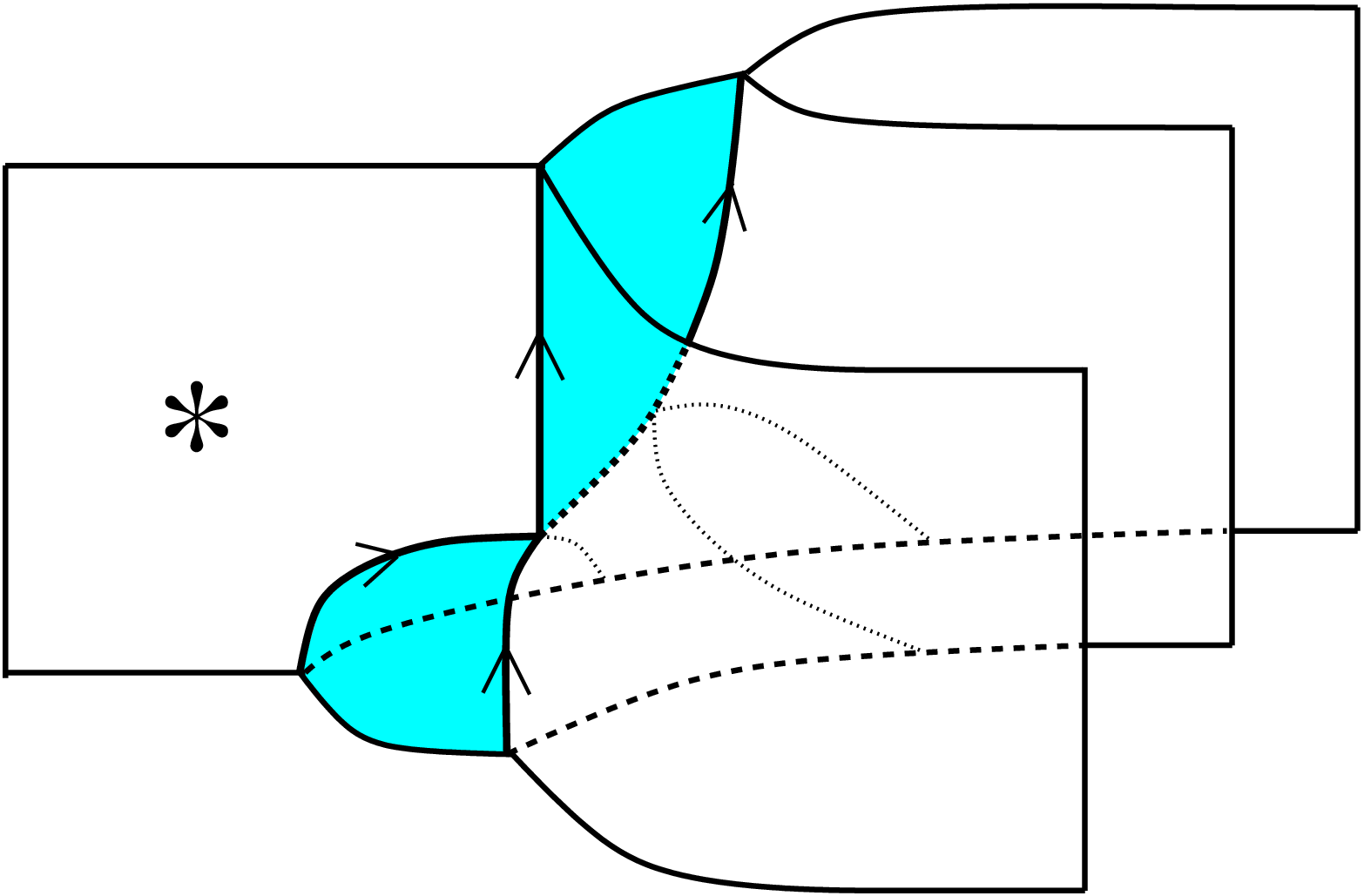}
\caption{Some elementary foams}
\label{fig:elemfoams}
\end{figure}
Intersecting such a foam with a plane results in a web, as long as the plane avoids the 
singularities where six facets meet, such as on the right in Figure~\ref{fig:elemfoams}. 

We adapt the definition of a world-sheet foam given in~\cite{Rozansky} to our setting.

\begin{defn}
\label{defn:pre-foam}
Let $\sk$ be a finite closed oriented $4$-valent graph, 
which may contain disjoint circles. We assume 
that all edges of $\sk$ are oriented.  
A cycle in $\sk$ is defined to be a circle or a closed sequence of edges which form a 
piece-wise linear circle. 
Let $\Sigma$ be a compact orientable possibly disconnected surface, 
whose connected components are white, coloured or marked, also denoted by simple, double or 
triple. Each component can have a boundary consisting of several disjoint circles and can have 
additional decorations which we discuss below.     
A closed {\em pre-foam} $u$ is the identification space $\Sigma/\sk$ obtained by glueing 
boundary circles of $\Sigma$ to cycles in $\sk$ such that every edge and circle in $\sk$ is glued 
to exactly three boundary circles of $\Sigma$ and such that for any point $p\in \sk$:    
\begin{enumerate}
\item if $p$ is an interior point of an edge, then $p$ has a neighborhood homeomorphic to the 
letter Y times an interval with exactly one of the facets being double, and at most one of them 
being triple. For an example see 
Figure~\ref{fig:elemfoams};
\item if $p$ is a vertex of $\sk$, then it has a neighborhood as shown in 
Figure~\ref{fig:elemfoams}. 
\end{enumerate}
We call $\sk$ the \emph{singular graph}, its edges and vertices \emph{singular arcs} and 
\emph{singular vertices}, and the connected components of $u - \sk$ the \emph{facets}.

Furthermore the facets can be decorated with dots. A simple facet can only have black 
dots ($\bdot$), a double facet can also have white dots ($\wdot$), and a triple facet besides 
black and white dots can have double dots ($\bwdot$). Dots can move freely on a facet 
but are not allowed to cross singular arcs. 
See Figure~\ref{fig:simpl-2pol} for examples of pre-foams. 
\end{defn}

\begin{figure}[h]
$$\xymatrix@R=1mm{
 \figins{0}{0.6}{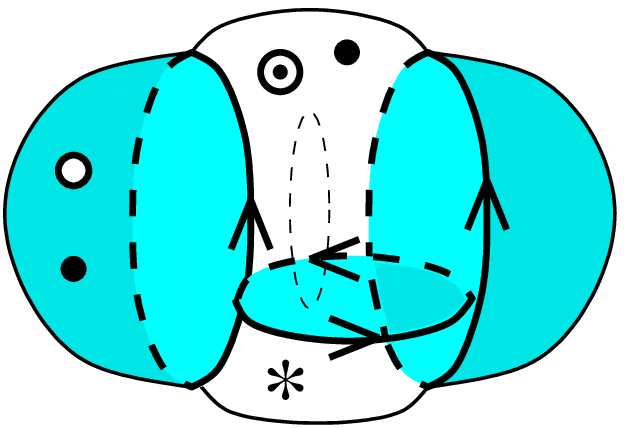} &
 \figins{-34}{0.6}{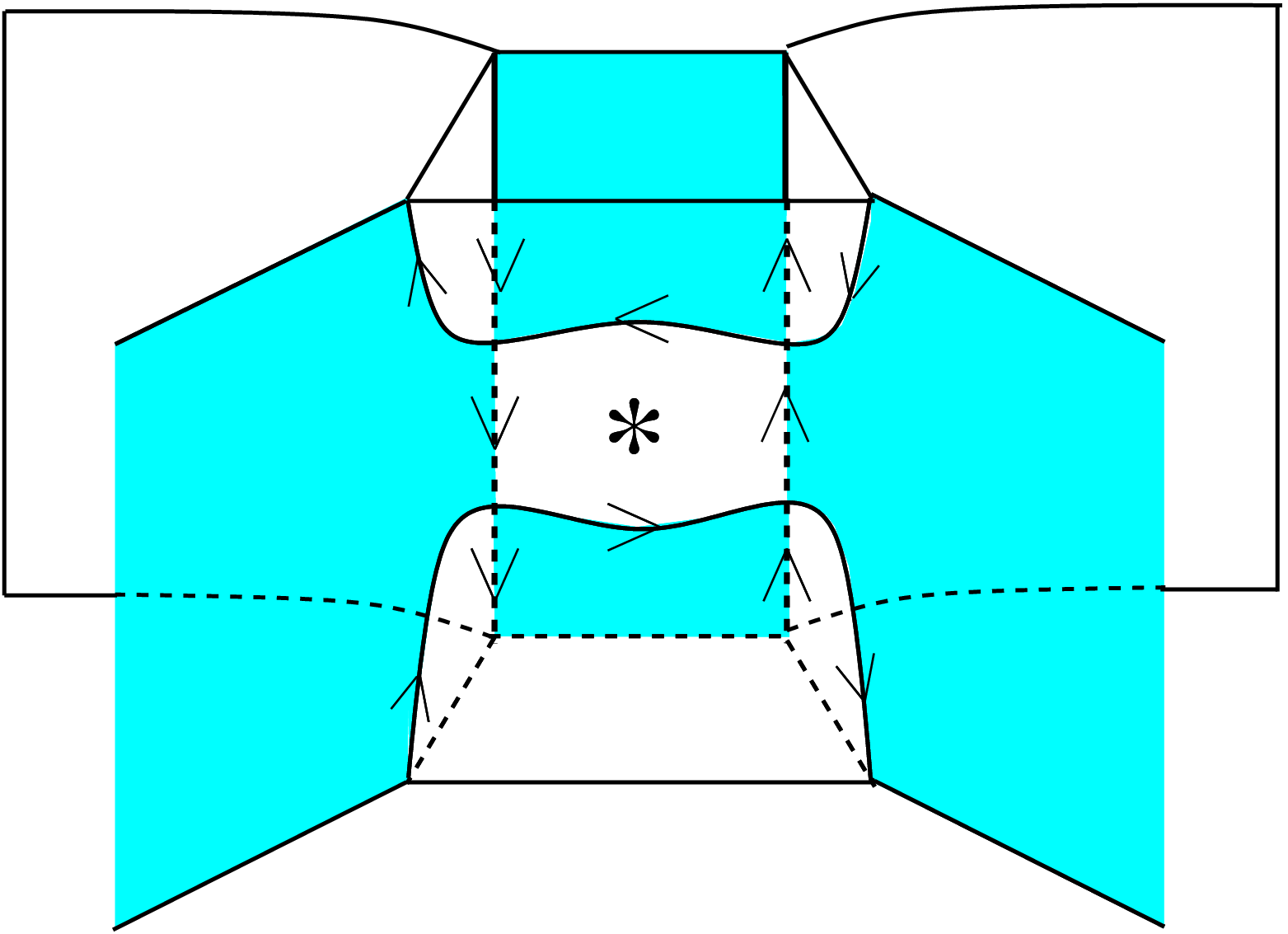} \\
a) & b)
}$$
\caption{a) A pre-foam b) An open pre-foam}
\label{fig:simpl-2pol}
\end{figure}

Note that the cycles to which the boundaries of the simple and the triple facets are glued are 
always oriented, whereas the ones to which the boundaries of the double facets are glued are not. 
Note also that there are two types of singular vertices. Given a singular vertex $v$, 
there are precisely two singular edges which meet at $v$ and bound a triple facet: one oriented 
toward $v$, denoted $e_1$, and one oriented away from $v$, denoted $e_2$. 
If we use the ``left hand rule'', then the cyclic ordering of the facets 
incident to $e_1$ and $e_2$ is either $(3,2,1)$ and $(3,1,2)$ respectively, or the other way 
around. We say that $v$ is of type I in the first case and of type II in the second case. 
When we go around a triple facet we see that there have to 
be as many singular vertices of 
type I as there are of type II for the cyclic orderings of the facets to match up. This shows 
that for a closed pre-foam the number of singular vertices of type I is equal to the number 
of singular vertices of type II.     

We can intersect a pre-foam $u$ generically by a plane $W$ in order to get a web, 
as long as the plane avoids the vertices of $\sk$. 
The orientation of $\sk$ determines the orientation of the simple edges of the web 
according to the convention in Figure~\ref{fig:orientations}.

\begin{figure}[h]
$$\xymatrix@R=1mm{
\figins{0}{0.8}{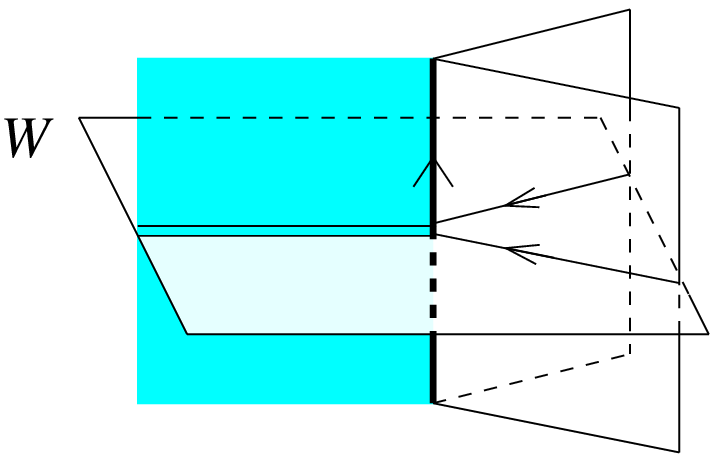} & 
\figins{0}{0.8}{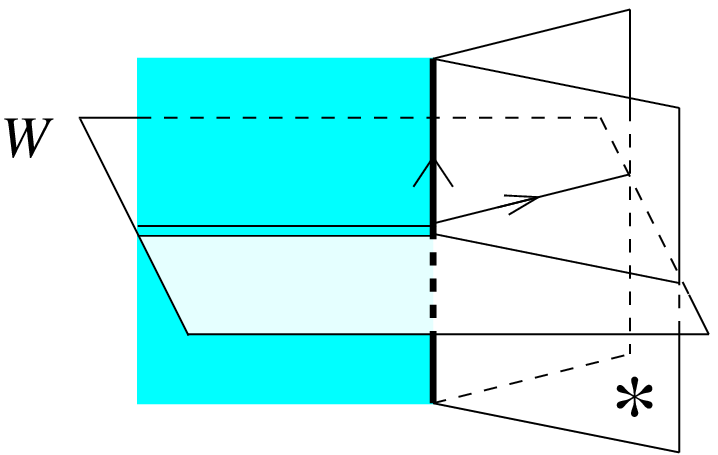} &
\figins{0}{0.8}{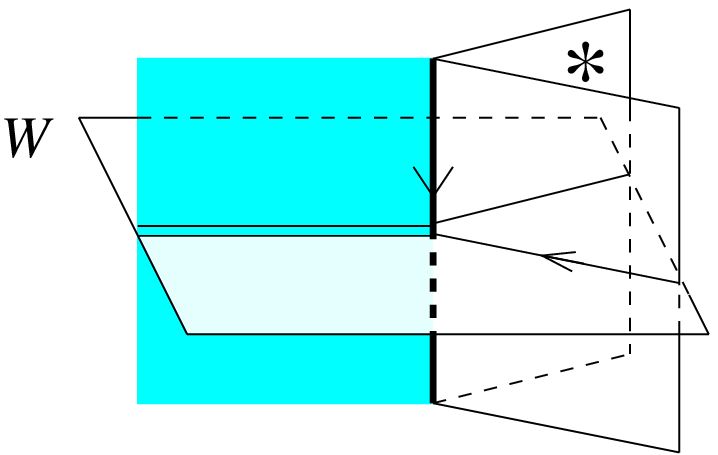} \\
\figins{0}{0.8}{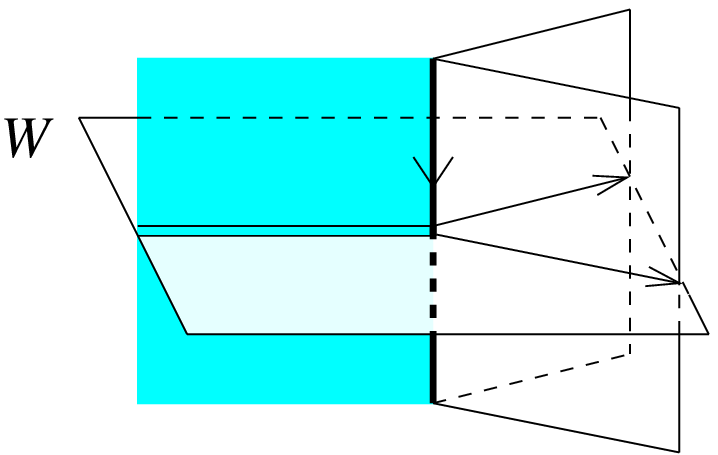}  & 
\figins{0}{0.8}{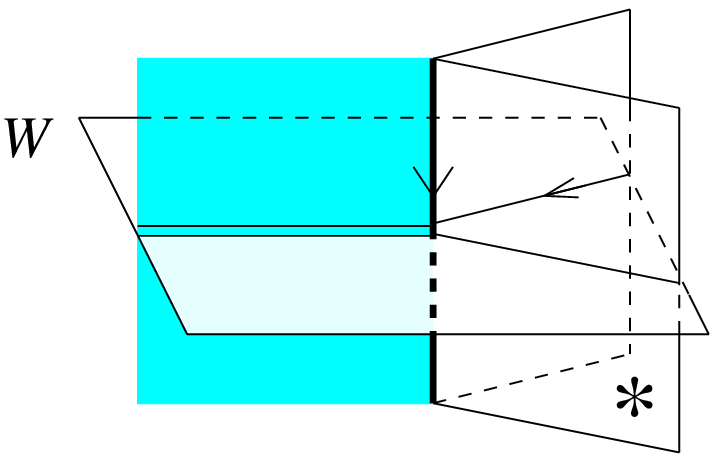} &
\figins{0}{0.8}{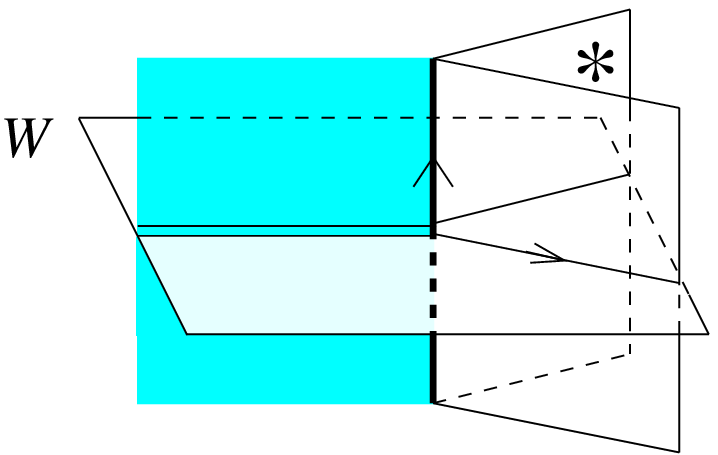}
}$$
\caption{Orientations near a singular arc}
\label{fig:orientations}
\end{figure}

Suppose that for all but a finite number of values $i\in ]0,1[$, 
the plane $W\times {i}$ intersects $u$ generically. Suppose also that 
$W\times 0$ and $W\times 1$ intersect $u$ generically and outside the vertices 
of $\sk$. We call $W\times I\cap u$ an {\em open} pre-foam.   
Interpreted as morphisms we read open pre-foams from bottom to top, and their 
composition consists of placing one pre-foam on top of the other, as long as their boundaries 
are isotopic and the orientations of the simple edges coincide.

\begin{defn} Let $\PF$ be the category whose objects are 
closed webs and whose morphisms are $\bQ$-linear combinations of 
isotopy classes of pre-foams with the obvious identity pre-foams and 
composition rule.  
\end{defn}



We now define the $q$-degree of a pre-foam. 
Let $u$ be a pre-foam, $u_1$, $u_2$ and $u_3$ the disjoint union of its 
simple and double and marked facets respectively and 
$\sk(u)$ its singular graph. Define the partial $q$-gradings of $u$ as 
\begin{eqnarray*}
q_i(u)    &=& \chi(u_i)-\frac{1}{2}\chi(\partial u_i\cap\partial u), 
\qquad i=1,2,3 \\
q_{\sk}(u) &=& \chi(\sk(u))-\frac{1}{2}\chi(\partial \sk(u)).
\end{eqnarray*}
where $\chi$ is the Euler characteristic and $\partial$ denotes the boundary.

\begin{defn}
Let $u$ be a pre-foam with $d_\bdot$ dots of type $\bdot$, $d_\wdot$ dots of type 
$\wdot$ and $d_\bwdot$ dots of type $\bwdot$. The $q$-grading of $u$ is given 
by
\begin{equation}
q(u)= -\sum_{i=1}^{3}i(N-i)q_i(u) - 2(N-2)q_{\sk}(u) + 
2d_\bdot + 4d_\wdot +6d_\bwdot.
\end{equation}
\end{defn}

The following result is a direct consequence of the definitions. 
\begin{lem}
$q(u)$ is additive under the glueing of pre-foams.
\end{lem}


\section{The Kapustin-Li formula and the evaluation of closed pre-foams}
\label{sec:KL}
Let us briefly recall the philosophy behind the pre-foams. Losely speaking, to each closed 
pre-foam should correspond an element in the cohomology ring of a configuration space of planes 
in some big $\bC^M$. The singular graph imposes certain conditions on those planes. The 
evaluation of a pre-foam should correspond to the evaluation of the corresponding element in 
the cohomology ring. Of course one would need to find a consistent way of choosing the volume forms 
on all of those configuration spaces for this to work. However, one encounters a difficult 
technical problem when working out the details of this philosophy. Without explaining all the 
details, we can say that the problem can only be solved by figuring out what to associate to 
the singular vertices. Ideally we 
would like to find a combinatorial solution to this problem, but so far it has eluded us. That 
is the reason why we are forced to use the Kapustin-Li formula.
          
We denote a simple facet with $i$ dots by 
$$\figins{-8}{0.3}{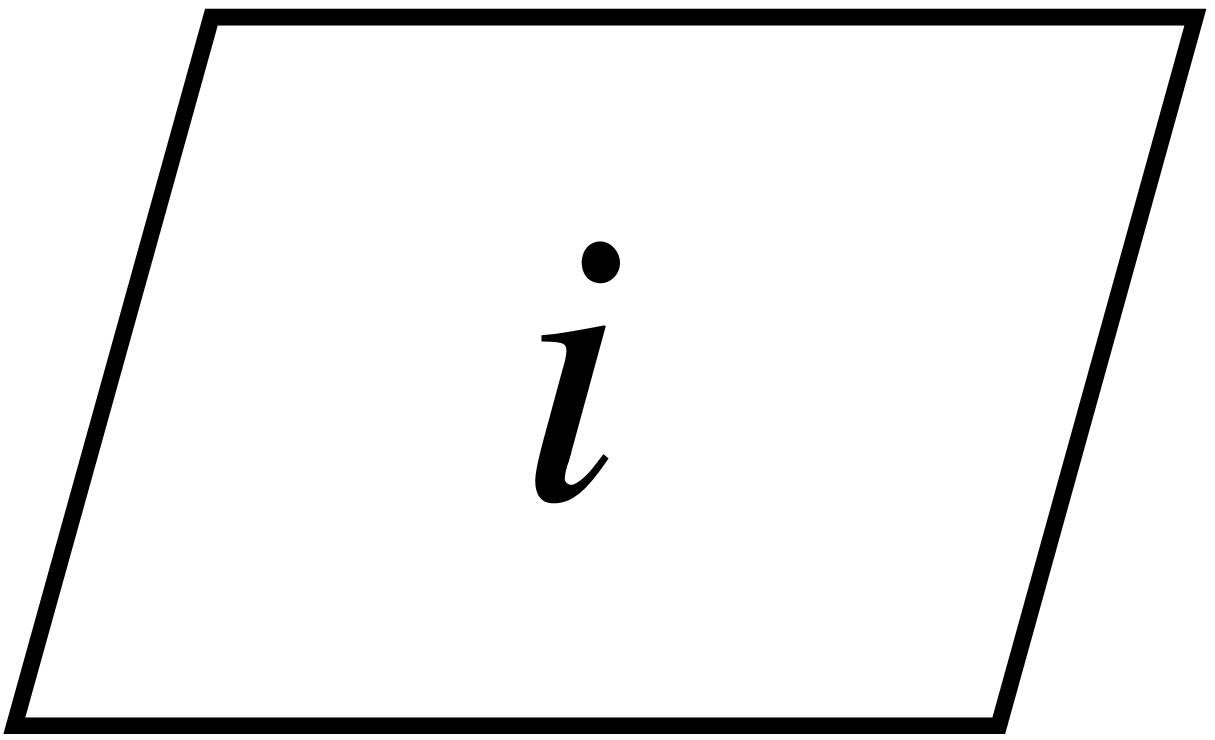}.$$  
Recall that $\pi_{k,m}$ can be expressed in terms of 
$\pi_{1,0}$ and $\pi_{1,1}$. In the philosophy explained above, the latter should correspond to 
$\bdot$ and $\wdot$ on a double facet respectively. We can then define  
$$\figins{-8}{0.3}{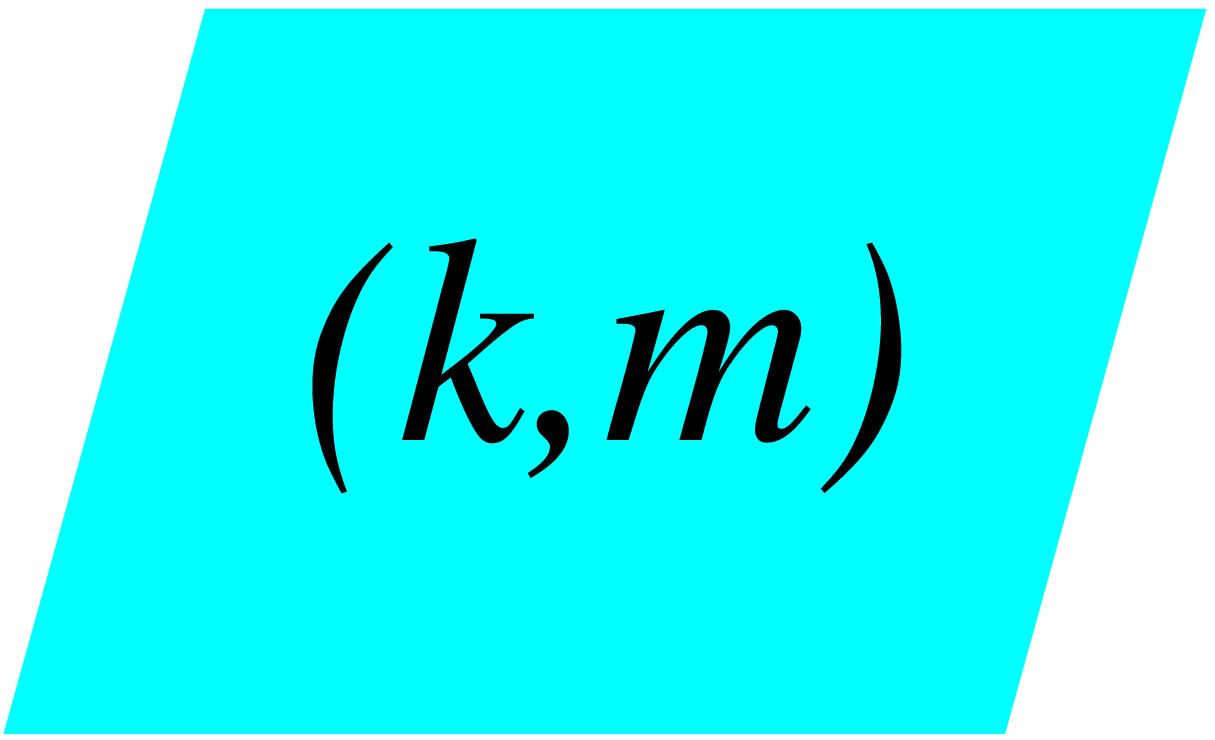}$$
as being the linear combination of dotted double facets corresponding to the expression of 
$\pi_{k,m}$ in terms of $\pi_{1,0}$ and $\pi_{1,1}$. 
Analogously we expressed $\pi_{p,q,r}$ in terms of $\pi_{1,0,0}$, $\pi_{1,1,0}$ and $\pi_{1,1,1}$ 
(see section~\ref{sec:partial flags}).  
The latter correspond to $\bdot$, $\wdot$ and $\bwdot$ on a triple facet respectively, so 
we can make sense of   
$$\figins{-8}{0.3}{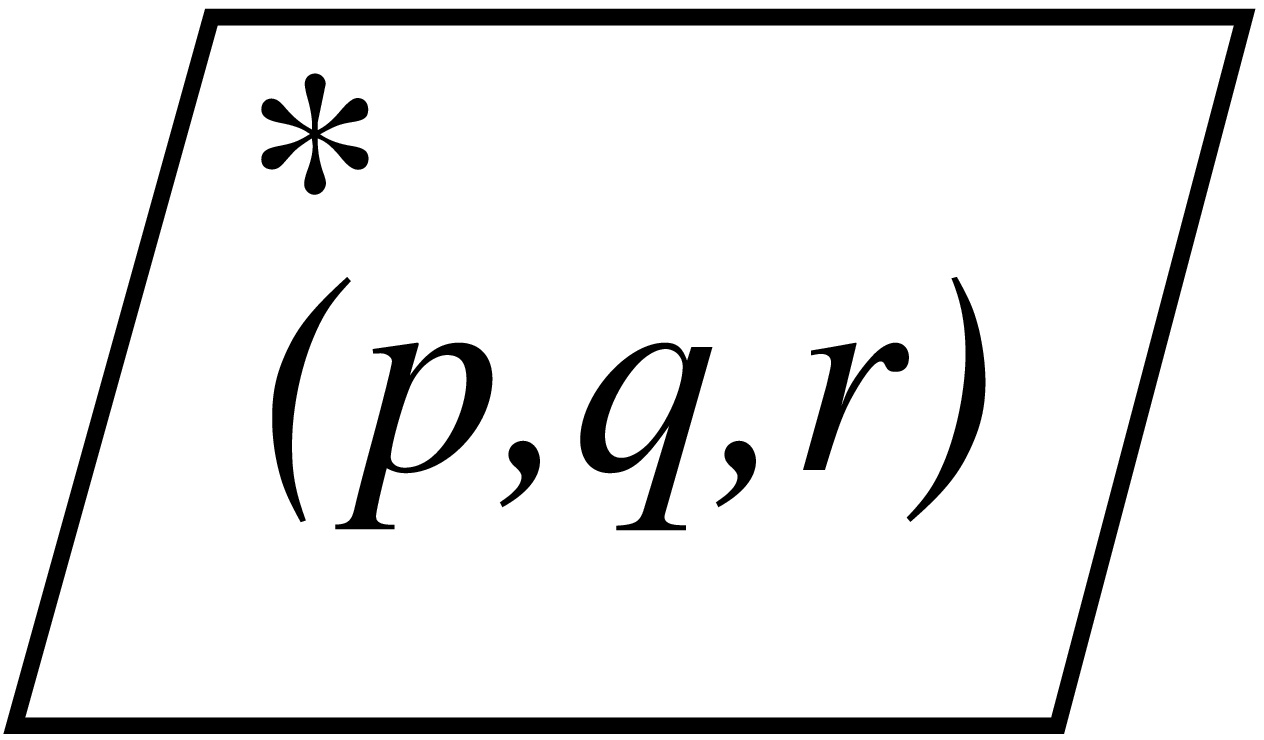}.$$
Our dot conventions and the results in 
proposition~\ref{prop:principal rels1} 
will allow us to use decorated facets in exactly the same way as we did Schur 
polynomials in the cohomology rings of partial flag varieties.

In the sequel, we shall give a working definition of the 
Kapustin-Li formula for the evaluation of pre-foams and state some 
of its basic properties. The Kapustin-Li formula was introduced 
by A. Kapustin and 
Y. Li \cite{KL} in the context of the evaluation of 2-dimensional TQFTs 
and extended to the case of pre-foams by M. Khovanov and L. Rozansky 
in \cite{KR-LG}.
\subsection{The general framework}
Let $u=\Sigma/s_{\gamma}$ be a closed pre-foam with 
singular graph $s_{\gamma}$ and without 
any dots on it. Let $F$ denote an arbitrary $i$-facet, $i\in\{1,2,3\}$, 
with a $1$-facet being a simple facet, a $2$-facet being a double facet 
and a $3$-facet being a triple facet.

Recall that to each $i$-facet we associated the rational cohomology ring 
of the Grassmanian $\G_{i,N}$, i.e. $H^*(\G_{i,N},\bQ)$. 
Alternatively, we can associate to every $i$-facet 
$F$, $i$ variables $x^F_1\ldots,x^F_i$, with $\deg x^F_i=2i$, and 
the potential 
$W(x^F_1,\ldots,x^F_i)$, which is the polynomial defined such that
$$W(\sigma_1,\ldots,\sigma_i)=y_1^{N+1}+\ldots+y_i^{N+1},$$
where $\sigma_j$ is the $j$-th elementary symmetric polynomial in the 
variables $y_1,\ldots,y_i$. The Jacobi 
algebra $J_W$, which is given by 
$$J_W=\bQ[x^F_1,\ldots,x^F_i]/\langle \mathbf{\partial} W\rangle$$
where we mod out by the ideal generated by the partial derivatives of $W$, 
is isomorphic to 
$H^*(\G_{i,N},\bQ)$. Note that the top 
degree nonvanishing element in this Jacobi algebra is $\pi_{N-i,\ldots,N-i}$ 
(multiindex of length $i$), i.e. the polynomial 
in variables $x^F_1,\ldots,x^F_i$ which gives $\pi_{N-i,\ldots,N-i}$ after replacing the variable $x^F_j$ by 
$\pi_{1,\ldots,1,0,\ldots,0}$ with exactly $j$ $1$'s, $1\le j\le i$ (see also 
subsection \ref{sec:Schur}). We define the trace 
(volume) form, $\epsilon$, on the cohomology ring of the 
Grassmanian, 
by giving it on the basis of the Schur polynomials: 
$$\epsilon(\pi_{j_1,\ldots,j_i})=
\begin{cases}(-1)^{\lfloor\frac{i}{2}\rfloor}&\text{if}\quad (j_1,\ldots,j_i)=(N-i,\ldots,N-i)\\
0&\text{else}
\end{cases}.
$$

The Kapustin-Li formula associates to $u$ an 
element in the product of the cohomology rings of the Grassmanians 
(or Jacobi 
algebras), $J$, over all the facets in the pre-foam. 
Alternatively, we can see this element as a polynomial, $KL_u\in J$, in 
all the variables associated to the facets. Now, let us put some dots on 
$u$. Recall that a dot corresponds to an elementary symmetric polynomial. 
So a linear combination of dots on $u$ is equivalent to a polynomial, $f$, 
in the variables of the dotted facets. 
The value of this dotted pre-foam we define to be
\begin{equation}
\langle u\rangle_{KL}:=\epsilon\left(\prod_{F}\dfrac{\det(\partial_i\partial_j W_F)^{g(F)}}{(N+1)^{g'(F)}}
KL_u\, f\right).
\label{ev}
\end{equation}
\noindent The product is over all facets $F$ and $W_F$ is the potential 
associated to $F$. For any $i$-facet $F$, $i=1,2,3$, the symbol $g(F)$ 
denotes the genus of $F$ and $g'(F)=ig(F)$.

Having explained the general idea, we are left with defining the element 
$KL_u$ for a dotless pre-foam. For that we have to explain Khovanov and 
Rozansky's extension of the Kapustin-Li formula to pre-foams \cite{KR-LG}, 
which uses the theory of matrix factorizations.

\subsection{Matrix factorizations}

Let $R=\bQ[x_1,\ldots,x_k]$ be a polynomial ring, and $W\in R$. By a matrix 
factorization over ring $R$ with the potential $W$ we mean a triple $(M,D,W)$, where $M=M^0\oplus M^1$ 
($\rank M^0=\rank M^1$) is a finite-dimensional $\bZ/2\bZ$-graded free $R$-module, while 
the (twisted) differential $D\in \End(M)$ is such that $\deg D=1$ and 
\begin{equation}
D^2=W \Id.
\label{matf}
\end{equation}

In other words, a matrix factorization is given by the following square 
matrix with polynomial entries
$$D=\left[\begin{matrix}
0& D_0 \\
D_1 & 0 
\end{matrix}
\right],$$
such that $D_0D_1=D_1D_0=W \Id$. Matrix factorizations are also represented in the 
following form:
$$M^0 {\xrightarrow{\ \ D_0\ \ }} M^1 {\xrightarrow{\ \ D_1\ \ }} M^0.$$

The tensor product of two matrix factorizations with potentials $W_1$ and $W_2$ is a matrix factorization with potential $W_1+W_2$.

The dual of the matrix factorization $(M,D,W)$ is given by
$$(M,D,W)^*=(M^*,D^*,-W),$$
where
$$D^*=\left[\begin{matrix}
0 & D_1^*\\
-D_0^* & 0
\end{matrix} \right]$$
and $D_i^*$, $i=0,1$, is the dual map (transpose matrix) of $D_i$.

Throughout the paper we shall use a particular type of matrix 
factorizations - namely the tensor products of Koszul factorizations. 
For two elements $a,b\in R$, the 
Koszul factorization $\{a,b\}$ is defined as the matrix factorization
$$R {\xrightarrow{\ \ a \ \ }} R {\xrightarrow{\ \ b \ \ }} R.$$

Moreover if $a=(a_1,\ldots,a_m)\in R^m$ and $b=(b_1,\ldots,b_m)\in R^m$, 
then the tensor product of the Koszul factorization $\{a_i,b_i\}$, 
$i=1,\ldots,m$, is denoted by
\begin{equation}
\left(\begin{array}{cc}
a_1& b_1\\
a_2& b_2\\
\vdots & \vdots \\
a_m& b_m
\end{array}
\right):=
\bigotimes_{i=1}^m \{a_i,b_i\}.
\label{Kosz}
\end{equation}

Sometimes we also write $\{a,b\}=\otimes_{i=1}^m \{a_i,b_i\}$. If 
$\sum_{i=1}^m a_ib_i=0$ then $\{a,b\}$ is a 2-periodic complex, and 
its homology is an $R/\langle a_1,\ldots,a_m,b_1,\ldots,b_m \rangle$-module.

\subsection{Decoration of pre-foams}

As we said, to each facet we associate certain variables (depending on the 
type of facet), a potential and the corresponding Jacobi algebra. 
If the variables associated to a facet $F$ are $x_1,\ldots,x_i$, then we 
define $R_F=\bQ[x_1,\ldots,x_i]$.

Now we pass to the edges. To each edge, we associate 
a matrix factorization whose potential is equal to the signed sum of the 
potentials of the facets that are glued along this edge. We define it to 
be a certain tensor product of Koszul factorizations.
In the cases we are interested in, there are always three 
facets glued along an edge, with two possibilities: either two 
simple facets and one double facet, or one simple, one double and one 
triple facet. 

In the first case, we denote the variables of the two simple facets by 
$x$ and $y$ and the potentials by $x^{N+1}$ and $y^{N+1}$ respectively. To 
the double facet we associate the variables $s$ and $t$ and the potential 
$W(s,t)$. To the edge we associate the matrix factorization 
which is the tensor product of Koszul factorizations given by
\begin{equation}
MF_1=\left( \begin{array}{cc}
x+y-s & A'\\
xy-t & B'
\end{array}\right),
\label{MF1}
\end{equation}
where $A'$ and $B'$ are given by
\begin{eqnarray*}
A'&=&\frac{W(x+y,xy)-W(s,xy)}{x+y-s},\\
B'&=&\frac{W(s,xy)-W(s,t)}{xy-t}. 
\end{eqnarray*}
Note that $(x+y-s)A'+(xy-t)B'=x^{N+1}+y^{N+1}-W(s,t)$.

In the second case, the variable of the simple facet is $x$ and 
the potential is $x^{N+1}$, the variables of 
the double facet are $s$ and $t$ and the potential is $W(s,t)$, and the 
variables of the triple face are $p$, $q$ and $r$ and the potential is 
$W(p,q,r)$. Define 
the polynomials
\begin{eqnarray}
A&=&\frac{W(x+s,xs+t,xt)-W(p,xs+t,xt)}{x+s-p},\label{MF21}\\
B&=&\frac{W(p,xs+t,xt)-W(p,q,xt)}{xs+t-q},\\
C&=&\frac{W(p,q,xt)-W(p,q,r)}{xt-r},
\end{eqnarray}
so that 
$$(x+s-p)A+(xs+t-q)B+(xt-r)C=x^{N+1}+W(s,t)-W(p,q,r).$$ 

To such an edge we associate the matrix factorization given by the 
following tensor product of Koszul factorizations:
\begin{equation}
MF_2=\left(\begin{array}{cc}x+s-p & A\\
xs+t-q & B\\
xt-r & C\end{array}\right).
\label{MF2}
\end{equation}

In both cases, to an edge with the opposite orientation we associate the 
dual matrix factorization.

\medskip

Next we explain what we associate to a singular vertex. First of all,
for each vertex $v$, we define its local graph $\gamma_v$ to be 
the intersection of a small sphere centered at $v$ with the pre-foam. 
Then the vertices of $\gamma_v$ correspond to the edges of $u$ that are 
incident to $v$, to which we had associated matrix factorizations. 

In this paper all local graphs are in fact tetrahedrons. However, recall 
that there are two types of vertices (see the remarks below definition~\ref{defn:pre-foam}). 
Label the six facets that are incident to a vertex $v$ by the numbers 
$1,2,3,4,5$ and $6$. Furthermore, 
denote the edge along which are glued the facets $i$, $j$ and $k$ by 
$(ijk)$. Denote the matrix factorization associated to the edge 
$(ijk)$ by $M_{ijk}$, if the edge points toward $v$, and by 
$M_{ijk}^*$, if the edge points away from $v$. Note that $M_{ijk}$ and 
$M_{ijk}^*$ are both defined over $R_i\otimes R_j \otimes R_k$. 

Now we can take the tensor product of these four matrix factorizations, 
over the polynomial rings of the facets of the pre-foam, that correspond to the 
vertices of $\gamma_v$. This way we obtain the matrix factorization $M_v$, 
whose potential is equal to $0$, and so it is a 
2-periodic chain complex 
and we can take its homology. To each vertex $v$ we associate an 
element $O_v\in H^*(M_v)$.

More precisely, if $v$ is of type I, then  
\begin{equation}
\begin{array}{rcl}
\label{eq:vertex}
H^*(M_v)&\cong& \Ext\left(MF_1(x,y,s_1,t_1)\otimes_{s_1,t_1} MF_2(z,s_1,t_1,p,q,r)\right. ,\\ 
& & \left. MF_1(y,z,s_2,t_2)\otimes_{s_2,t_2} MF_2(x,s_2,t_2,p,q,r)\right).
\end{array}
\end{equation}

If $v$ is of type II, then 
\begin{equation}
\begin{array}{rcl}
H^*(M_v)&\cong& \Ext\left(MF_1(y,z,s_2,t_2)\otimes_{s_2,t_2} MF_2(x,s_2,t_2,p,q,r)\right. ,\\
& & \left. MF_1(x,y,s_1,t_1)\otimes_{s_1,t_1} MF_2(z,s_1,t_1,p,q,r)\right).
\end{array}
\end{equation}
Both isomorphisms hold up to a global shift in $q$. Note that 
$$
MF_1(x,y,s_1,t_1)\otimes_{s_1,t_1} MF_2(z,s_1,t_1,p,q,r) \simeq
MF_1(y,z,s_2,t_2)\otimes_{s_2,t_2} MF_2(x,s_2,t_2,p,q,r),
$$
because both tensor products are homotopy equivalent to 
$$
\left(\begin{array}{cc}x+y+z-p & *\\
xy+xz+yz-q & *\\
xyz-r & *\end{array}\right).
$$
We have not specified the r.h.s. of the latter Koszul factorizations, because by theorem 2.1 in 
\cite{KR2} 
we have $\{a,b\}\simeq \{a,b'\}$ if $\sum a_ib_i= \sum a_ib_i'$ and if the sequence 
$\{a_i\}$ is regular. If $v$ is of type I, we take $O_v$ to be the cohomology class of 
a fixed degree $0$ homotopy equivalence
$$
w_v\colon MF_1(x,y,s_1,t_1)\otimes_{s_1,t_1} MF_2(z,s_1,t_1,p,q,r)\to
MF_1(y,z,s_2,t_2)\otimes_{s_2,t_2} MF_2(x,s_2,t_2,p,q,r).
$$
The choice of $O_v$ is unique up to a scalar, because the $q$-dimension 
of the $\Ext$-group in (\ref{eq:vertex}) is equal to 
$$q^{3N-6}\qdim(H^*(M_v))=q^{3N-6}[N][N-1][N-2]=1+q(\ldots),$$
where $(\ldots)$ is a polynomial in $q$. Note that $M_v$ is homotopy equivalent to 
the matrix factorization which corresponds to the closure of $\Upsilon$ in \cite{KR}, which 
allows one to compute the $q$-dimension above using the results in the latter paper.
If $v$ is of type II, we take $O_v$ to be the cohomology class of the homotopy 
inverse of $w_v$. Note that a particular choice of $w_v$ fixes $O_v$ for 
both types of vertices and that the value of the Kapustin-Li formula for a closed 
pre-foam does not depend on that choice because there are as many singular vertices of type I 
as there are of type II (see the remarks below definition~\ref{defn:pre-foam}). 
We do not know an explicit formula for $O_v$. Although such a formula would be very interesting 
to have, we do not need it for the purposes of this paper.

\subsection{The Kapustin-Li derivative and the evaluation of closed pre-foams}

From the definition, every boundary component of each facet $F$ is 
either a circle or a cyclicly ordered finite sequence
of edges, such that the beginning of the next edge corresponds to the end 
of the previous
edge. For every boundary component we choose an edge $e$ - the value of 
the Kapustin-Li formula does not depend on this choice. Denote the 
differential of the matrix factorization associated to this edge by $D_e$.   

The associated Kapustin-Li derivative of $D_e$ in the 
variables $x_1,\ldots,x_k$ associated to the facet $F$, is an element 
from $\End(M)\cong M\otimes M^*$, given by:
\begin{equation}
O_{F,e}=\partial D_{e} \hat{\ } =\frac{1}{k!}\sum_{\sigma\in S_k} 
(\sgn{\sigma}){\partial_{\sigma(1)
}D_{e} \partial_{\sigma(2)}D_{e}\ldots\partial_{\sigma(k)}D_{e}},
\label{tw}
\end{equation}
where $S_k$ is the set of all permutations of the set $\{1,\ldots,k\}$, 
and $\partial_i D$ is the partial derivative of $D$ with respect to 
the variable $x_i$. Note that $e$ can be the preferred edge for more than 
one facet. In general, let $O_e$ be the product of 
$O_{F,e}$ over all facets $F$ for which $e$ is the preferred edge. 
The order of the factors in this product is irrelevant, because they commute (see \cite{KR-LG}). 
If $e$ is not the preferred edge for any $F$, we take $O_e$ to be the identity. 

Finally, around each boundary component of $\partial F$, for each 
facet $F$, we contract all tensor factors $O_e$ and $O_v$. 
Note that one has to use super-contraction in order to 
get the right signs.   

For a better understanding of the Kapustin-Li formula, consider the 
special case of a theta pre-foam $\Theta$. There are three facets $F_1$, $F_2$, $F_3$ which are 
glued along a common circle $c$, which is the preferred edge for all three. We associated a 
certain matrix factorization $M$ to $c$ with differential $D$. Let 
$\partial D_1 \hat{\ }$, $\partial D_2 \hat{\ }$ and $\partial D_3 \hat{\ }$ 
be the Kapustin-Li derivatives of $D$ with respect to the variables of the 
facets $F_1$, $F_2$ and $F_3$, respectively. Then we have 
\begin{equation}
KL_{\Theta}={\Tr}^s ({\partial} D_1\hat{\ }{\partial} D_2\hat{\ }{\partial} 
D_3\hat{\ }).
\end{equation}

As a matter of fact we will see that we have to {\em normalize} the Kapustin-Li formula in 
order to get ``nice values''.

\subsection{Dot conversion and dot migration}
The pictures related to the computations in this subsection and the next three can be found in 
Proposition~\ref{prop:principal rels1}.

Since $KL_u$ takes values in the tensor product of the 
Jacobian algebras of the potentials associated to the facets of $u$, we see that 
for a simple facet we have $x^N=0$, for a double facet $\pi_{i,j}=0$ if $i\geq N-1$, and 
for a triple facet $\pi_{p,q,r}=0$ if $p\geq N-2$. We call these the 
{\em dot conversion relations}.   
 
To each edge along which two simple facets with variables $x$ and $y$ and one double 
facet with the variables $s$ and $t$ are glued, we 
associated the matrix factorization $MF_1$ with 
entries $x+y-s$ and $xy-t$. Therefore $\Ext(MF_1,MF_1)$ 
is a module over $R/\langle x+y-s,xy-t\rangle$. Hence, 
we obtain the {\em dot migration relations} along this edge.

Analogously, to the other type of singular edge along 
which are glued a simple facet with variable $x$, 
a double facet with variable $s$ and $t$, and a triple 
facet with variables $p$, $q$ and $r$,
we associated the matrix factorization $MF_2$ 
and $\Ext(MF_2,MF_2)$ is a module over 
$R/\langle x+s-p, xs+t-q,xt-r\rangle$, and hence we 
obtain the {\em dot migration relations} along this edge.

\subsection{$(1,1,2)-$Theta}

Recall that $W(s,t)$ is the polynomial such that 
$W(x+y,xy)=x^{N+1}+y^{N+1}$. More precisely, 
we have
\begin{equation}
W(s,t)=\sum_{i+2j=N+1}a_{ij}s^it^j,
\end{equation}
with $a_{N+1,0}=1$, $a_{N+1-2j,j}=\frac{(-1)^j}{j}(N+1){N-j\choose j-1}$, for $2\le 
2j\le N+1$, and $a_{ij}=0$ otherwise. In particular $a_{N-1,1}=-(N+1)$. Then we have
\begin{eqnarray}
W'_1(s,t)=\sum_{i+2j=N+1}ia_{ij}s^{i-1}t^j,\\
W'_2(s,t)=\sum_{i+2j=N+1}ja_{ij}s^it^{j-1}.
\end{eqnarray}
 
By $W'_1(s,t)$ and $W'_2(s,t)$, we denote the partial derivatives of $W(s,t)$ with respect to the first and the second variable, respectively.

To the singular circle of a standard theta pre-foam with two 
simple facets, with variables $x$ and $y$ respectively, and 
one double facet, with variables $s$ and $t$, we assign the  
matrix factorization $MF_1$:
\begin{equation}
MF_{1}=\left(\begin{array}{cc}
x+y-s & A'\\
xy-t & B'
\end{array} \right).
\end{equation}
Recall that  
\begin{eqnarray}
A'&=&\frac{W(x+y,xy)-W(s,xy)}{x+y-s}, \label{Adef}\\
B'&=&\frac{W(s,xy)-W(s,t)}{xy-t}. \label{Bdef}
\end{eqnarray}
Hence, the differential of this matrix factorization is 
given by the following 4 by 4 matrix:
\begin{equation}
D=\left[\begin{array}{cc}
0 & D_0\\
D_1 & 0
\end{array}
\right],
\end{equation}
where
\begin{equation}
D_0=\left[\begin{array}{cc}
x+y-s & -B'\\
xy-t & A'
\end{array}
\right],\quad
D_1=\left[\begin{array}{cc}
A' & B'\\
t-xy & x+y-s
\end{array}
\right].
\end{equation}
The Kapustin-Li formula assigns the polynomial, $KL_{\Theta_1}(x,y,s,t)$, which is given by the supertrace of the 
twisted differential of $D$
\begin{equation}
KL_{\Theta_1}={\Tr}^s\left(\partial_x D\partial_y D \frac{1}{2}(\partial_s D\partial_t D-\partial_t D \partial_s D)\right).
\end{equation}
Straightforward computation gives
\begin{equation}
KL_{\Theta_1}=B'_s(A'_x-A'_y)+(A'_x+A'_s)(B'_y+xB'_t)-(A'_y+A'_s)(B'_x+yB'_t),
\label{KL1}
\end{equation}
where by $A'_i$ and $B'_i$ we have denoted the partial derivatives with respect to the variable $i$.
From the definitions (\ref{Adef}) and (\ref{Bdef}) we have 
\begin{eqnarray*}
A'_x-A'_y&=&(y-x)\frac{W'_2(x+y,xy)-W'_2(s,xy)}{x+y-s},\\
A'_x+A'_s&=&\frac{W'_1(x+y,xy)-W'_1(s,xy)+y(W'_2(x+y,xy)-W'_2(s,xy))}{x+y-s},\\
A'_y+A'_s&=&\frac{W'_1(x+y,xy)-W'_1(s,xy)+x(W'_2(x+y,xy)-W'_2(s,xy))}{x+y-s},\\
B'_s&=&\frac{W'_1(s,xy)-W'_1(s,t)}{xy-t},\\
B'_x+yB'_t&=&y\frac{W'_2(s,xy)-W'_2(s,t)}{xy-t},\\
B'_y+xB'_t&=&x\frac{W'_2(s,xy)-W'_2(s,t)}{xy-t}.
\end{eqnarray*}
After substituting this back into (\ref{KL1}), we obtain
\begin{equation}
KL_{\Theta_1}=(y-x)\left|\begin{array}{cc}
\alpha &\beta\\
\gamma & \delta 
\end{array}\right|,
\label{eq1}
\end{equation}  
where
\begin{eqnarray*}
\alpha&=&\frac{W'_1(x+y,xy)-W'_1(s,xy)}{x+y-s},\\
\beta&=&\frac{W'_2(x+y,xy)-W'_2(s,xy)}{x+y-s},\\
\gamma&=&\frac{W'_1(s,xy)-W'_1(s,t)}{xy-t},\\
\delta&=&\frac{W'_2(s,xy)-W'_2(s,t)}{xy-t}.
\end{eqnarray*}
From this formula we see that $KL_{\Theta_1}$ is homogeneous of degree $4N-6$ (remember that 
$\deg x=\deg y=\deg s=2$ and $\deg\,t=4$).

Since the evaluation is in the product of the Grassmanians corresponding to the three disks, i.e. in the ring
$\bQ[x]/x^N \times \bQ[y]/y^N \times \bQ[s,t]/\langle W'_1(s,t),W'_2(s,t) \rangle$, we have 
$x^N=y^N=0=W'_1(s,t)=W'_2(s,t)$. Also, we can express the monomials in $s$ and $t$ as linear 
combinations of the Schur polynomials $\pi_{k,l}$ (writing $s=\pi_{1,0}$ and $t=\pi_{1,1})$), and 
we have $W'_1(s,t)=(N+1)\pi_{N,0}$ and $W'_2(s,t)=-(N+1)\pi_{N-1,0}$. Hence, we can write 
$KL_{\Theta_1}$ as
$$KL_{\Theta_1}=(y-x)\sum_{N-2\ge k\ge l\ge 0} {\pi_{k,l} p_{kl}(x,y)},$$
with $p_{kl}$ being a polynomial in $x$ and $y$. We want to determine which combinations of dots 
on the simple facets give rise to non-zero evaluations, so our aim is to compute the coefficient of $\pi_{N-2,N-2}$ in the sum on the r.h.s. of the above equation 
(i.e. in the determinant in (\ref{eq1})). For degree reasons, this coefficient is of degree zero, 
and so we shall only compute the parts of $\alpha$, $\beta$, $\gamma$ and $\delta$ which do not 
contain $x$ and $y$. We shall denote these parts by putting 
a bar over the Greek letters. Thus we have
\begin{eqnarray*}
\bar{\alpha}&=&(N+1)s^{N-1},\\
\bar{\beta}&=&-(N+1)s^{N-2},\\
\bar{\gamma}&=&\sum_{i+2j=N+1,\,j\ge 1}ia_{ij}s^{i-1}t^{j-1},\\
\bar{\delta}&=&\sum_{i+2j=N+1,\,j\ge 2}ja_{ij}s^it^{j-2}.\\
\end{eqnarray*}
Note that we have 
$$t\bar{\gamma}+(N+1)s^N=W'_1(s,t),$$
and 
$$t\bar{\delta}-(N+1)s^{N-1}=W'_2(s,t),$$
and so in the cohomology ring of the Grassmanian $\G_{2,N}$, we have $t\bar{\gamma}=-(N+1)s^N$ and $t\bar{\delta}=(N+1)s^{N-1}$.
On the other hand, by using $s=\pi_{1,0}$ and $t=\pi_{1,1}$, we obtain that in 
$H^*(\G_{2,N})\cong\bQ[s,t]/\langle\pi_{N-1,0},\pi_{N,0}\rangle$, the following holds:
$$s^{N-2}=\pi_{N-2,0}+tq(s,t),$$
for some polynomial $q$, and so
$$s^{N-1}=s^{N-2}s=\pi_{N-1,0}+\pi_{N-2,1}+stq(s,t)=t(\pi_{N-3,0}+sq(s,t)).$$
Thus, we have
\begin{eqnarray}
\left|
\begin{array}{cc}
\bar{\alpha} &\bar{\beta}\\
\bar{\gamma} &\bar{\delta} 
\end{array}\right|&=&
(N+1)(\pi_{N-3,0}+sq(s,t))t\bar{\delta} +(N+1)\pi_{N-2,0}\bar{\gamma} +(N+1)q(s,t)t\bar{\gamma} \nonumber \\
&=&(N+1)^2(\pi_{N-3,0}+sq(s,t))s^{N-1} + (N+1)\pi_{N-2,0}\bar{\gamma} - (N+1)^2 q(s,t) s^N \nonumber \\
&=&(N+1)^2\pi_{N-3,0}s^{N-1}+ (N+1)\pi_{N-2,0}\bar{\gamma}. \label{eq2}
\end{eqnarray}
Since 
$$\bar{\gamma}=(N-1)a_{N-1,1}s^{N-2} + t r(s,t)$$
holds in the cohomology ring of Grassmanian, 
for some polynomial $r(s,t)$, we have
$$\pi_{N-2,0}\bar{\gamma}=\pi_{N-2,0}(N-1)a_{N-1,1}s^{N-2}=
-\pi_{N-2,0}(N-1)(N+1)s^{N-2}.$$
Also, we have that for every $k\ge 2$,
$$s^k= \pi_{k,0}+(k-1)\pi_{k-1,1}+t^2 w(s,t),$$
for some polynomial $w$. Replacing this in (\ref{eq2}) and bearing in mind that $\pi_{i,j}=0$, for $i\ge N-1$, we get
\begin{eqnarray}
\left|
\begin{array}{cc}
\bar{\alpha} &\bar{\beta}\\
\bar{\gamma} &\bar{\delta} 
\end{array}\right|
&=& (N+1)^2 s^{N-2} (\pi_{N-2,0}+\pi_{N-3,1}-(N-1)\pi_{N-2,0}) \nonumber\\
&=& (N+1)^2 (\pi_{N-2,0}+(N-3)\pi_{N-3,1}+\pi_{2,2}w(s,t)) (\pi_{N-3,1}-(N-2)\pi_{N-2,0})\nonumber\\
&=& -(N+1)^2 \pi_{N-2,N-2}.  
\end{eqnarray}
Hence, we have
$$KL_{\Theta_1}=(N+1)^2 (x-y) \pi_{N-2,N-2} + \sum_{
N-2\ge k \ge l\ge 0 \atop  N-2>l}
c_{i,j,k,l} \pi_{k,l} x^i y^j.$$
Recall that in the 
product of the Grassmanians corresponding to the three 
disks, i.e. in the ring 
$\bQ[x]/x^N \times \bQ[y]/y^N \times \bQ[s,t]/\langle \pi_{N-1,0},\pi_{N,0} \rangle$, we have 
$$\epsilon(x^{N-1}y^{N-1}\pi_{N-2,N-2})=-1.$$
Therefore the only monomials $f$ in $x$ and $y$ such that 
$\langle KL_{\Theta_1}f\rangle_{KL}\ne 0$ 
are $f_1=x^{N-2}y^{N-1}$ and $f_2=x^{N-1}y^{N-2}$, and  
$\langle KL_{\Theta_1}f_1\rangle_{KL}=-(N+1)^2$ and 
$\langle KL_{\Theta_1}f_2\rangle=(N+1)^2$. Thus, we have that the value of 
the theta pre-foam with unlabelled 2-facet is nonzero only when the 
first 1-facet has $N-2$ dots and the second one has $N-1$ 
dots (and has the value $-(N+1)^2$) and when the 
first 1-facet has $N-1$ dots and the second one has $N-2$ dots 
(and has the value $(N+1)^2$). The evaluation of this theta foam with 
other labellings can be obtained from the result above by dot migration.

\subsection{$(1,2,3)-$Theta}
For this theta the method is the same as in the previous case, just the computations are more complicated.
In this case, we have one 1-facet, to which we associate the 
variable $x$, one 2-facet, with variables $s$ and $t$ and 
the 3-facet with variables $p$, $q$ and $r$. Recall that the 
polynomial $W(p,q,r)$ is such that 
$W(a+b+c,ab+bc+ac,abc)=a^{N+1}+b^{N+1}+c^{N+1}$. We denote 
by $W'_i(p,q,r)$, $i=1,2,3$, the partial derivative of $W$ 
with respect to $i$-th variable. Also, let $A$, $B$ and $C$ 
be the polynomials such that
\begin{eqnarray}
A&=&\frac{W(x+s,xs+t,xt)-W(p,xs+t,xt)}{x+s-p},\label{1}\\
B&=&\frac{W(p,xs+t,xt)-W(p,q,xt)}{xs+t-q},\\
C&=&\frac{W(p,q,xt)-W(p,q,r)}{xt-r}.\label{3}
\end{eqnarray}

To the singular circle of this theta pre-foam, we associated 
the matrix factorization 
(see (\ref{MF21})-(\ref{MF2})):
$$MF_2=\left(\begin{array}{cc}x+s-p & A\\
xs+t-q & B\\
xt-r & C\end{array}\right).$$
The differential of this matrix factorization is the 8 by 8 matrix
\begin{equation}
D=\left[\begin{array}{cc}
0 & D_0\\
D_1 & 0
\end{array}
\right],
\end{equation}
where
\begin{equation}
D_0=\left[\begin{array}{c|c}
d_0 & -C\, I_2 \\\hline
(xt-r)I_2 & d_1 
\end{array}
\right],
\end{equation}
\begin{equation}
D_1=\left[\begin{array}{c|c}
d_1 & C\, I_2\\
\hline
(r-xt)I_2 & d_0 \\
\end{array}
\right].
\end{equation}
Here $d_0$ and $d_1$ are the differentials of the matrix factorization
$$\left(\begin{array}{cc}x+s-p & A\\
xs+t-q & B
\end{array}\right),$$
i.e.
$$d_0=\left[\begin{array}{cc}x+s-p & -B\\
xs+t-q & A
\end{array}\right],\quad\quad
d_1=\left[\begin{array}{cc}A & B\\
q-xs-t & x+s-p
\end{array}\right].$$
The Kapustin-Li formula assigns to this theta pre-foam the polynomial 
$KL_{\Theta_2}(x,s,t,p,q,r)$ 
given as the supertrace of the twisted differential of $D$, i.e.
$$KL_{\Theta_2}={\Tr}^s\left(\partial_x D\frac{1}{2}(\partial_s D\partial_t D-\partial_t D \partial_s D)
\partial_3D\hat{\ }\right),$$
where
\begin{eqnarray*}\partial_3D\hat{\ } &=&\frac{1}{3!}\left(\partial_p D\partial_qD 
\partial_rD-\partial_p D\partial_rD \partial_qD+\partial_q D\partial_rD \partial_pD-\right. \\
&&\left.\partial_q D\partial_pD \partial_rD+\partial_r D\partial_pD \partial_qD-\partial_r D\partial_qD \partial_pD\right).
\end{eqnarray*}
After straightforward computations and some grouping, we obtain
\begin{eqnarray*}
KL_{\Theta_2}
&=& -(A_p+A_s)[(B_t+B_q)(C_x+tC_r)-(B_x+sB_q)(C_t+xC_r)-(B_x-sB_t)C_q]\\
&&-\ (A_p+A_x)[(B_s+xB_q)(C_t+xC_r)+(B_s-xB_t)C_q]\\
&&-\ (A_x-A_s)[B_p(C_t+xC_r)-(B_t+B_q)C_p+B_pC_q]\\
&&+\ A_t[((B_s+xB_q)+B_p)(C_x+tC_r)+((B_s+xB_q)\\
&&+\ (B_x+sB_q))C_p+((sB_s-xB_x)+(s-x)B_p)C_q].
\end{eqnarray*}
In order to simplify this expression, we introduce the following polynomials
\begin{eqnarray*}
a_{1i} &=& \frac{W'_i(x+s,xs+t,xt)-W'_i(p,xs+t,xt)}{x+s-p},\quad i=1,2,3, \\
a_{2i} &=& \frac{W'_i(p,xs+t,xt)-W'_i(p,q,xt)}{xs+t-q},\qquad\qquad i=1,2,3, \\
a_{3i} &=& \frac{W'_i(p,q,xt)-W'_i(p,q,r)}{xt-r},\qquad\qquad\qquad i=1,2,3.
\end{eqnarray*}

Then from (\ref{1})-(\ref{3}), we have
 $$A_x+A_p=a_{11}+sa_{12}+ta_{13},\quad A_p+A_s=a_{11}+xa_{12},$$
 $$A_x-A_s=(s-x)a_{12}+ta_{13},\quad A_t=a_{12}+xa_{13},$$
 $$B_p=a_{21},\quad B_s-xB_t=-x^2a_{23},$$
 $$sB_s-xB_x=xta_{23},\quad B_x-sB_t=(t-sx)a_{23},$$
 $$B_t+B_q=a_{22}+xa_{23},\quad B_x+sB_q=sa_{22}+ta_{23}, B_s+xB_q=xa_{22},$$
 $$C_p=a_{31},\quad C_q=a_{32},\quad C_x+tC_r=ta_{33},\quad 
C_t+xC_r=xa_{33}.$$

\medskip

Using this $KL_{\Theta_2}$ becomes
$$KL_{\Theta_2}=-(t-sx+x^2)\left| \begin{array}{ccc}a_{11}&a_{12}&a_{13}\\
a_{21}&a_{22}&a_{23}\\
a_{31}&a_{32}&a_{33}\end{array}\right| .$$

Now the last part follows analogously as in the case of the 
$(1,1,2)$-theta pre-foam. For degree reasons  
the coefficient of $\pi_{N-3,N-3,N-3}$ in the latter 
determinant is of degree zero, and one can obtain that it is 
equal to $(N+1)^3$. Thus, the coefficient of $\pi_{N-3,N-3,N-3}$ in $KL_{\Theta_2}$ is 
$-(N+1)^3(t-sx+x^2)$ from which we obtain the value of the theta pre-foam when the 3-facet is 
undotted. For example, we see that 
$$\epsilon(KL_{\Theta_2}\pi_{1,1}(s,t)^{N-3}x^{N-1})=-(N+1)^3.$$
It is then easy to obtain the values when the 3-facet is 
labelled by $\pi_{N-3,N-3,N-3}(p,q,r)$ using dot migration.  
The example above implies that 
$$\epsilon(KL_{\Theta_2}\pi_{N-3,N-3,N-3}(p,q,r)x^2)=-(N+1)^3.$$

\subsection{Spheres}

The values of dotted spheres are easy to compute. Note that for any sphere with dots $f$ 
the Kapustin-Li formula gives 
$$\epsilon(f).$$ 
Therefore for a simple sphere we get $1$ if $f=x^{N-1}$, for a double 
sphere we get 
$-1$ if $f=\pi_{N-2,N-2}$ and for a triple sphere we get $-1$ if 
$f=\pi_{N-3,N-3,N-3}$. 

\subsection{Normalization}
\label{sec:norm}

It will be convenient to normalize the Kapustin-Li evaluation. Let $u$ be a closed pre-foam 
with graph $\Gamma$. Note that $\Gamma$ has two types of edges: the ones incident to two 
simple facets and one double facet and the ones incident to one simple, one double and one triple facet. 
Edges of the same type form cycles in $\Gamma$. Let $e_{112}(u)$ be the total number of cycles in $\Gamma$ 
with edges of the first type and $e_{123}(u)$ the total number of cycles with edges of the second type. 
We normalize the Kapustin-Li formula by dividing $KL_u$ by 
$$(N+1)^{2e_{112}+3e_{123}}.$$
In the sequel we only use this normalized Kapustin-Li 
evaluation keeping the same notation $\langle u \rangle_{KL}$.   
Note that the numbers $e_{112}(u)$ and $e_{123}(u)$ are invariant under the relation (MP). 
Note also that with this normalization the KL-evaluation in the examples above always gives 
$0,-1$ or $1$.

\subsection{The glueing property}
\label{sec:glueing}

If $u$ is an open pre-foam whose boundary consists of two parts $\Gamma_1$ and $\Gamma_2$, then 
the Kapustin-Li formula associates 
to $u$ an element from $\Ext(M_1,M_2)$, where $M_1$ and $M_2$ are matrix factorizations associated to $\Gamma_1$ and $\Gamma_2$ respectively. If $u'$ is another pre-foam whose boundary consists of $\Gamma_2$ and $\Gamma_3$, then it corresponds to an element from $\Ext(M_2,M_3)$, while the element associated to the pre-foam $uu'$, which is obtained by gluing the pre-foams $u$ and $u'$ along $\Gamma_2$, is equal to the composite of the elements associated to $u$ and $u'$.

On the other hand, we can see $u$ as a morphism from the empty set 
to its boundary $\Gamma=\Gamma_2\cup \overline{\Gamma_1}$, where $\overline{\Gamma_1}$ is 
equal to $\Gamma_1$ but with the opposite orientation. In that case, the Kapustin-Li formula 
associates to it an element from 
$$\Ext(\emptyset,M_{\Gamma_2}\otimes M_{\Gamma_1}^*)\cong H^*(\Gamma).$$ 
Of course both ways of applying the Kapustin-Li formula are equivalent up to a global 
$q$-shift by corollary 6 in \cite{KR}. 

In the case of a pre-foam $u$ with corners, i.e. a pre-foam with two horizontal boundary components 
$\Gamma_1$ and $\Gamma_2$ which are connected by vertical edges, one has to ``pinch'' the 
vertical edges. This way one can consider $u$ to be a morphism from the empty set to 
$\Gamma_2\cup_{v}\overline{\Gamma_1}$, where $\cup_v$ means that the webs 
are glued at their vertices. The same observations as above hold, except that 
$M_{\Gamma_2}\otimes M_{\Gamma_1}^*$ is now the tensor product over the polynomial ring 
in the variables associated to the horizontal edges with corners.


\section{The category $\foam$}
\label{sec:foamN}
Recall that $\langle u \rangle_{KL}$ denotes the Kapustin-Li evaluation of a closed 
pre-foam $u$.  
\begin{defn}
\label{defn:foam}
The category $\foam$ is the quotient of the category $\PF$ by the 
kernel of $\langle\;\rangle_{KL}$, i.e. by the following identifications: for any webs $\Gamma$, 
$\Gamma'$ 
and finite sets $f_i\in\mbox{Hom}_{\PF}\left(\Gamma,\Gamma'\right)$ 
and $c_i\in\bQ$ we impose the relations 
$$\sum\limits_{i}c_if_i=0\quad \Leftrightarrow\quad
\sum\limits_{i}c_i\langle g'f_ig\rangle_{KL}=0,$$ 
for all 
$g\in\mbox{Hom}_{\PF}\left(\emptyset,\Gamma\right)$ and  
$g'\in\mbox{Hom}_{\PF}\left(\Gamma',\emptyset\right)$. The 
morphisms of $\foam$ are called {\em foams}. 
\end{defn}

In the next two propositions we prove the ``principal'' relations in $\foam$. 
All other relations that we need are consequences of these and will be proved in 
subsequent lemmas and corollaries.     

\begin{prop}
\label{prop:principal rels1}
The following identities hold in $\foam$: 

 
(The {\em dot conversion} relations) 

$$\figins{-8}{0.3}{plan-i}=0\quad\mbox{if}\quad i\geq N.$$ 
$$\quad\ \
\figins{-8}{0.3}{dplan-km}=0\quad\mbox{if}\quad k\geq N-1.$$
$$\quad\ \
\figins{-8}{0.3}{plan-pqr}=0\quad\mbox{if}\quad p\geq N-2.$$


\bigskip

(The {\em dot migration} relations)
\begin{eqnarray*}
\figins{-12}{0.4}{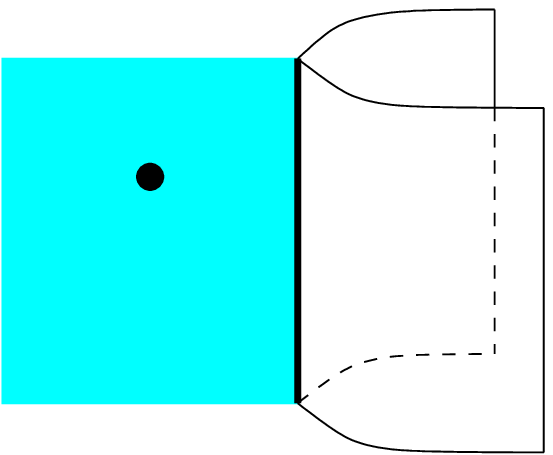} &=&  \quad
\figins{-12}{0.4}{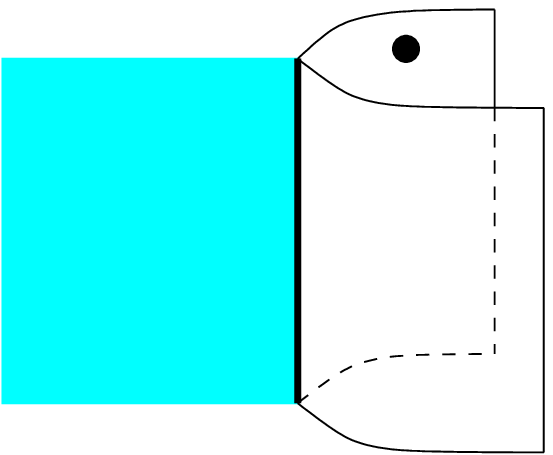} \ + \
\figins{-12}{0.4}{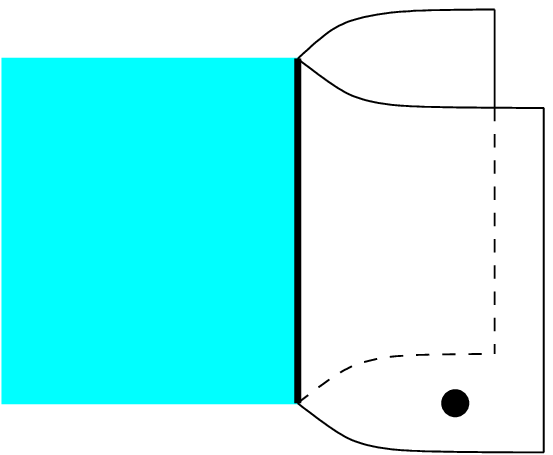}
\\ 
\figins{-12}{0.4}{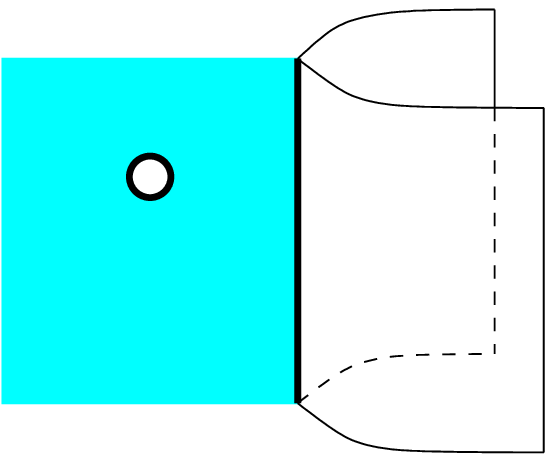} &=&  \quad
\figins{-12}{0.4}{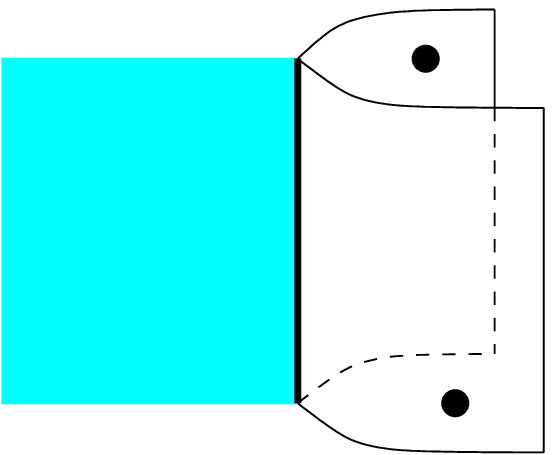}
\\ \\
\figins{-12}{0.4}{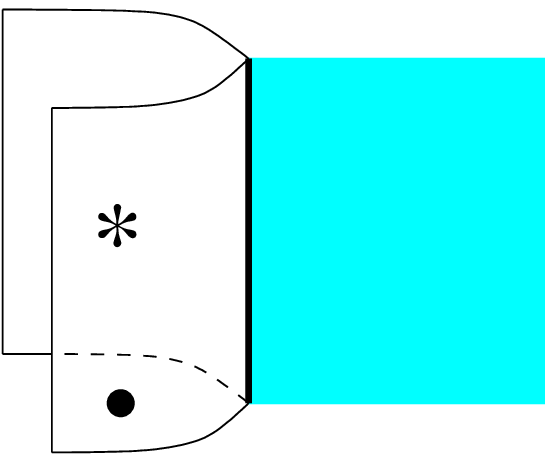} \ \ \ &=&  \
\figins{-12}{0.4}{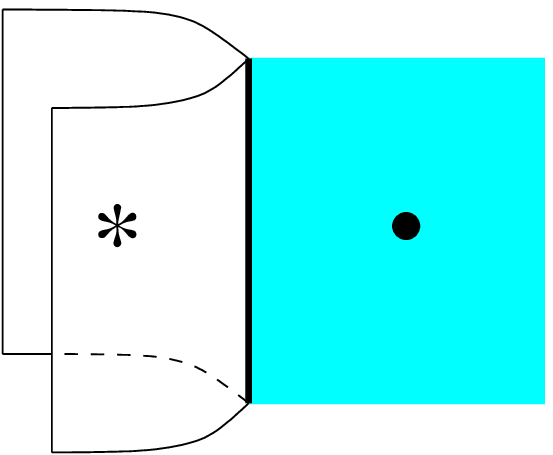} \ + \
\figins{-12}{0.4}{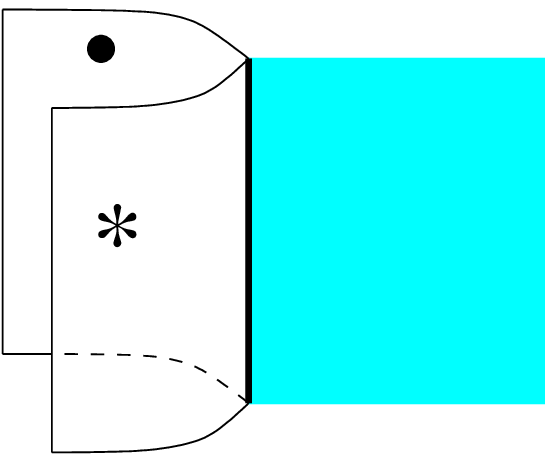} 
\\
\figins{-12}{0.4}{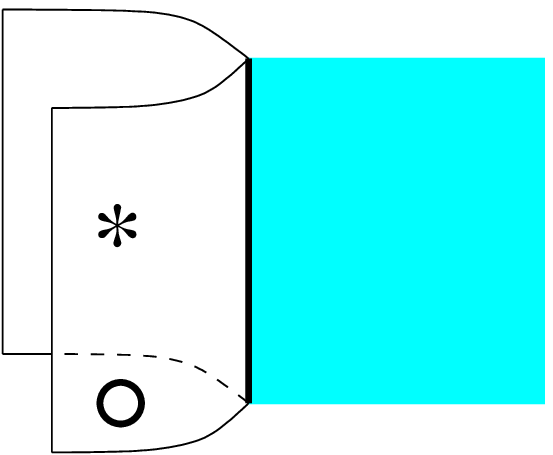} \ \ \ &=&  \
\figins{-12}{0.4}{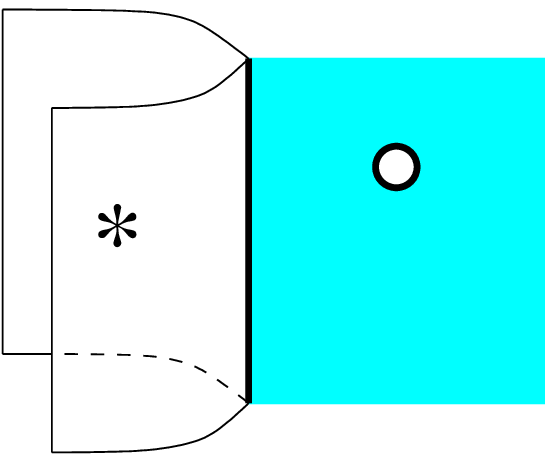} \ + \
\figins{-12}{0.4}{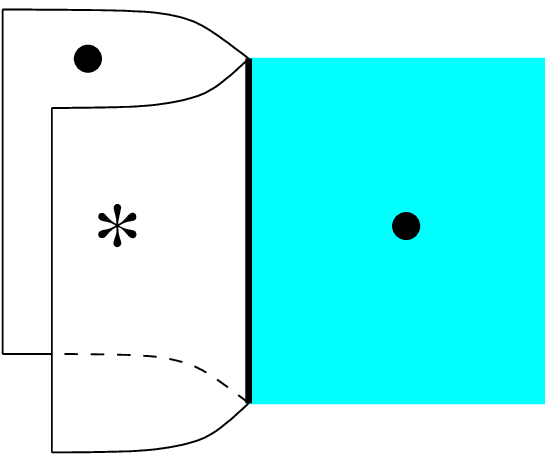}
\\
\figins{-12}{0.4}{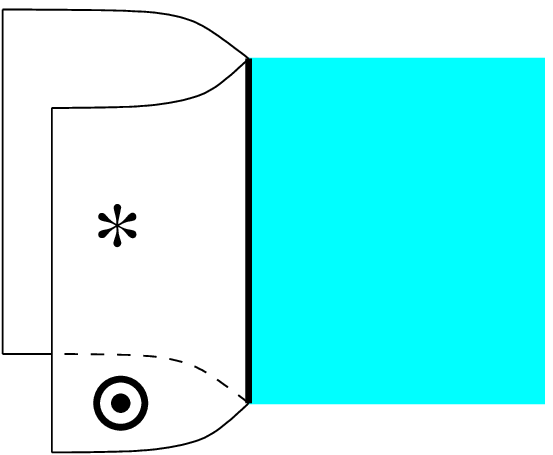} \ \ \ &=&  \
\figins{-12}{0.4}{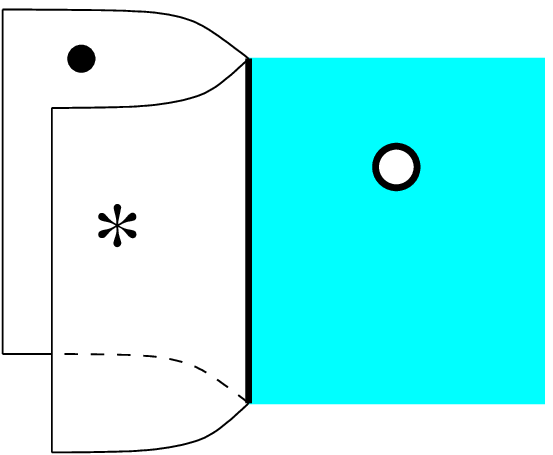}
\end{eqnarray*}

\bigskip

(The {\em cutting neck} relations) 
$$
\figins{-27}{0.8}{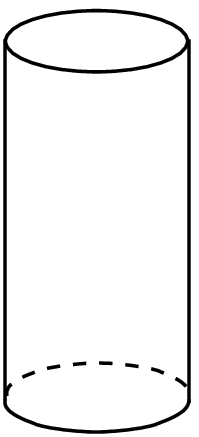}=
\sum\limits_{i=0}^{N-1}
\figins{-27}{0.8}{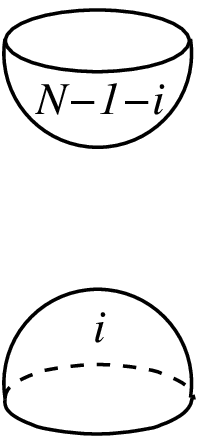}
\quad\text{(CN$_1$)}
$$
$$   
\figins{-28}{0.8}{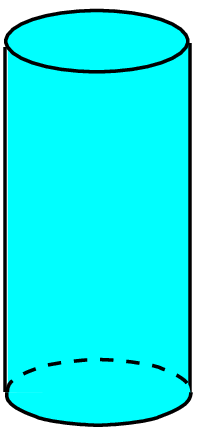}=-
\sum\limits_{0\leq j\leq i\leq N-2}
\figins{-28}{0.8}{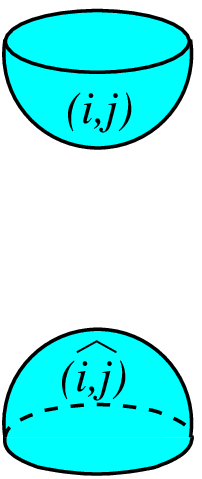}
\quad\text{(CN$_2$)}
\qquad\qquad
\figins{-28}{0.8}{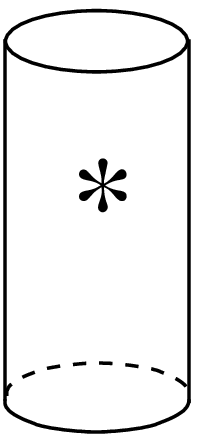}=
-\sum\limits_{0\leq k\leq j\leq i\leq N-3}
\figins{-28}{0.8}{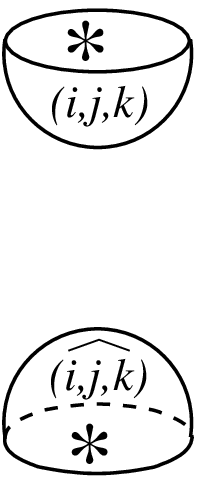}
\quad\text{(CN$_*$)}
$$

\bigskip

(The {\em sphere} relations)
$$
\figins{-10}{0.35}{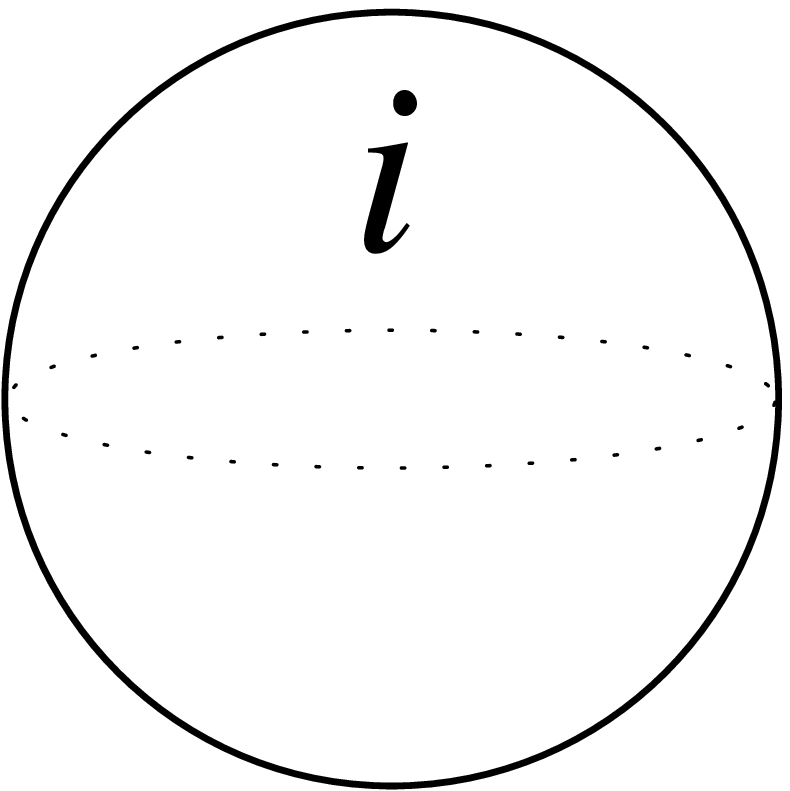}=
\begin{cases}1, & i=N-1 \\ 0, & \text{else}\end{cases}\quad 
\text{(S$_1$)} \qquad\quad
\figins{-10}{0.35}{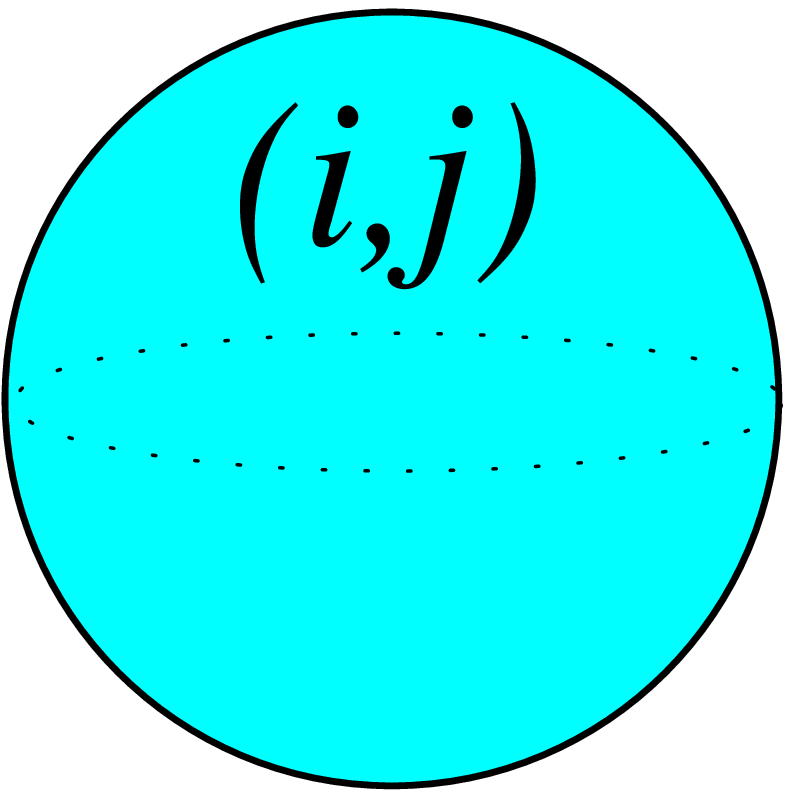}=
\begin{cases}-1, & i=j=N-2 \\ 0, & \text{else}\end{cases}\quad 
\text{(S$_2$)}
$$

$$
\figins{-10}{0.35}{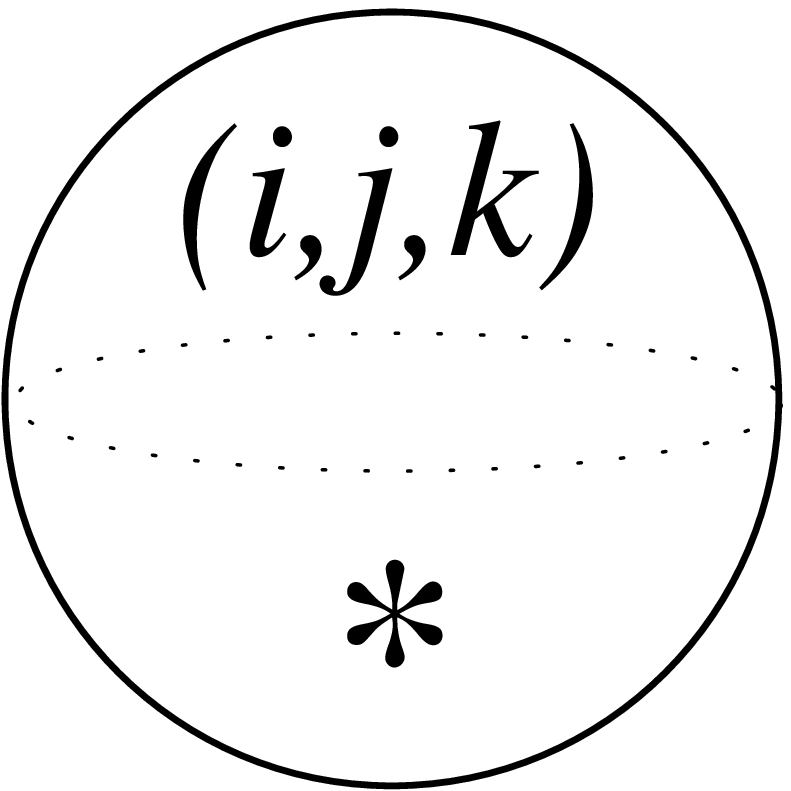}=
\begin{cases}-1, & i=j=k=N-3 \\ 0, & \text{else}.\end{cases}\quad
\text{(S$_*$)}
$$

\bigskip

(The {\em $\Theta$-foam} relations)
$$
\figins{-10}{0.35}{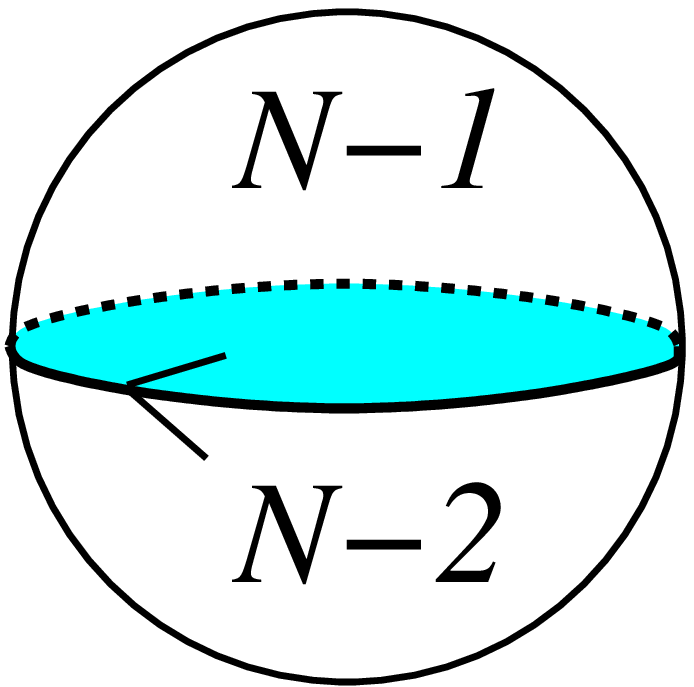} = -1  = - \
\figins{-10}{0.35}{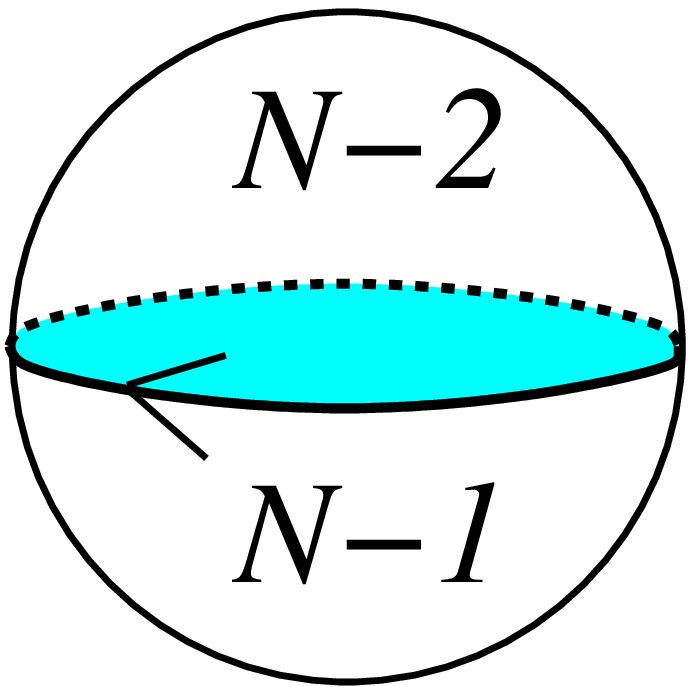} \quad
 (\ThetaGraph)
\qquad \ \text{and} \ \qquad
\figins{-10}{0.42}{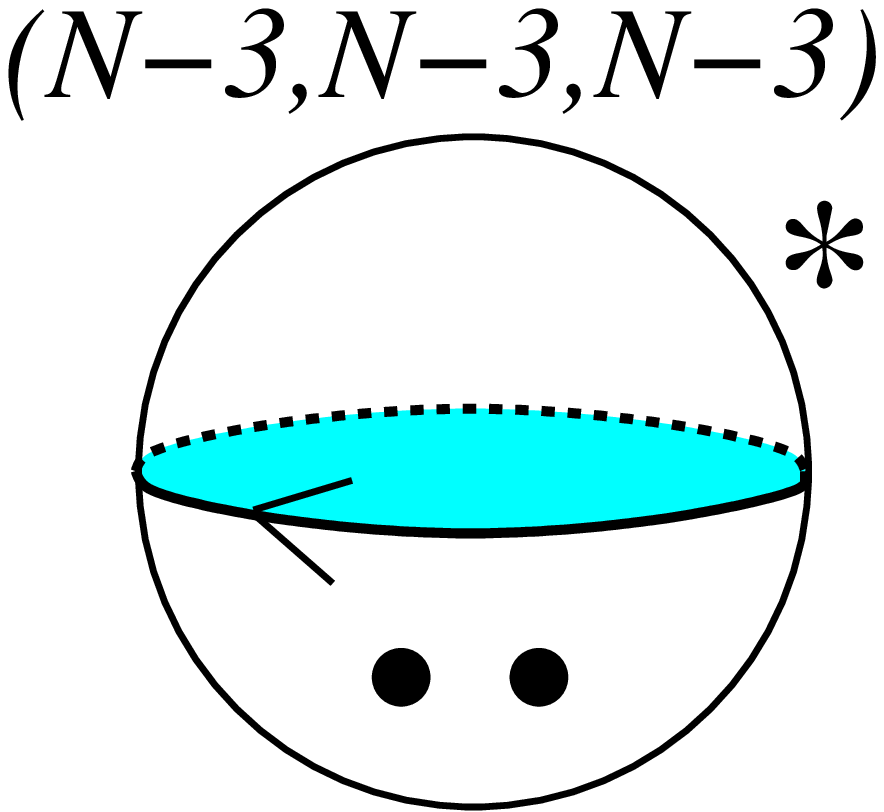} = -1  = - \
\figins{-16}{0.42}{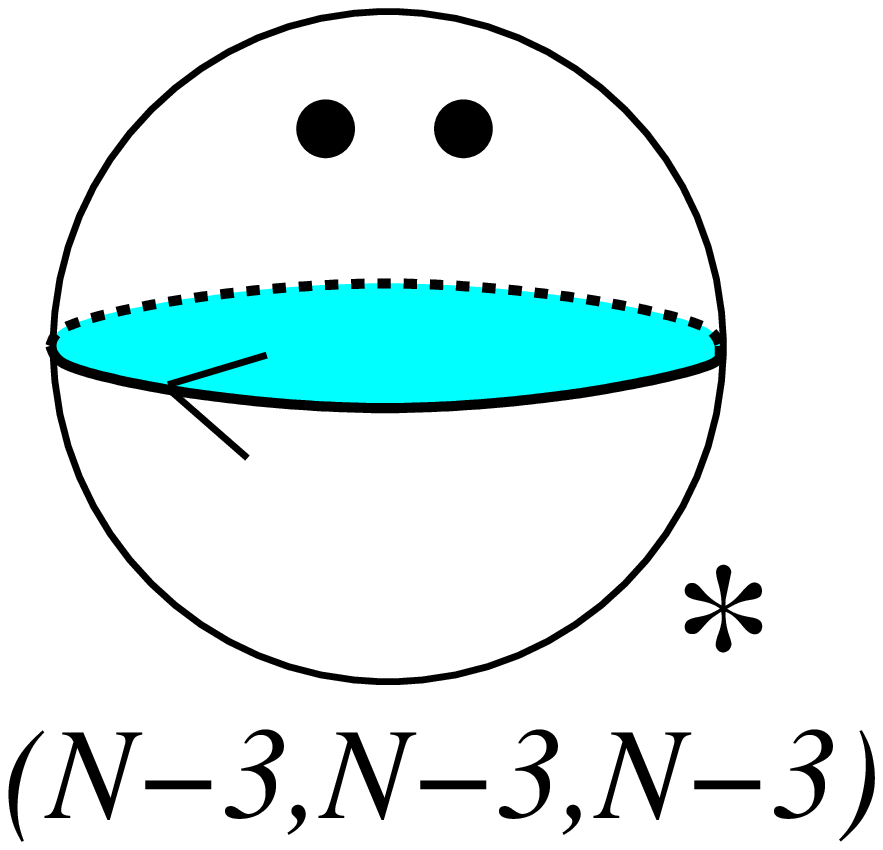} \quad
 (\ThetaGraph_*).$$
Inverting the orientation of the singular circle of $(\ThetaGraph_*)$ inverts the sign of the 
corresponding foam. A theta-foam with dots on the double facet can be transformed into a 
theta-foam with dots only on the other two facets, using the dot migration relations.


\bigskip

(The {\em Matveev-Piergalini} relation) 
$$
\figins{-24}{0.8}{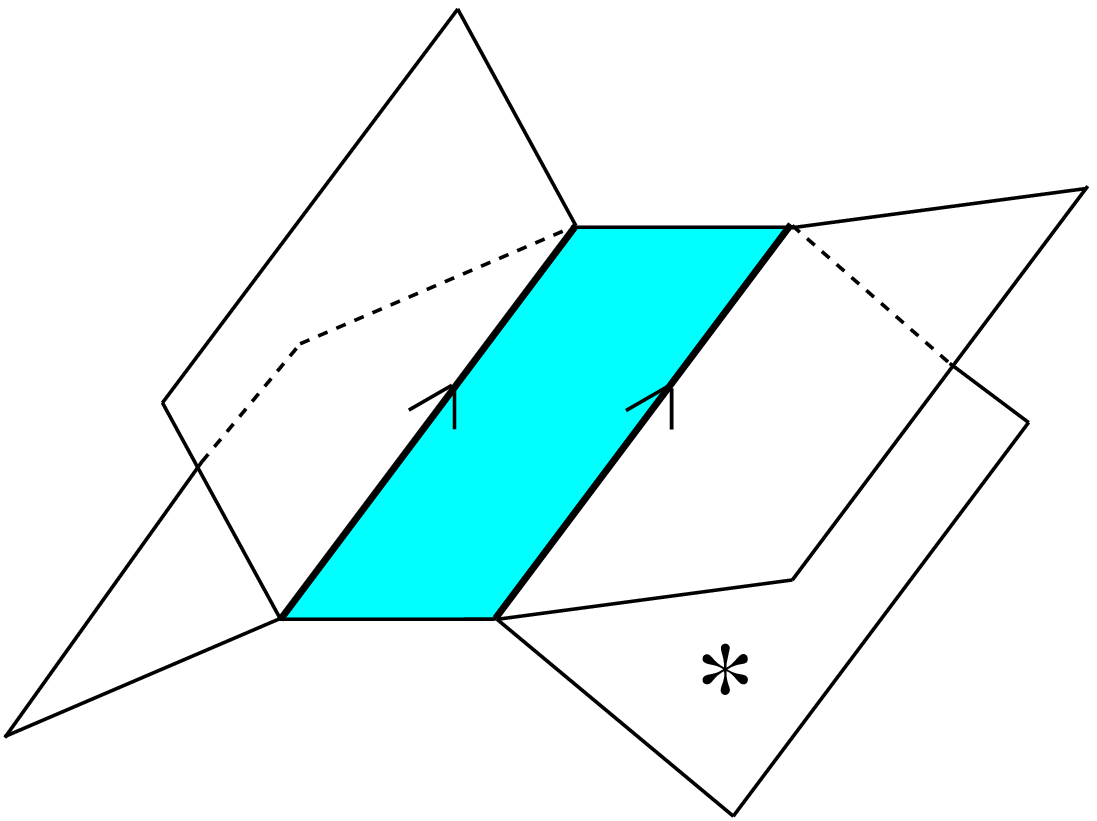} =
\figins{-24}{0.8}{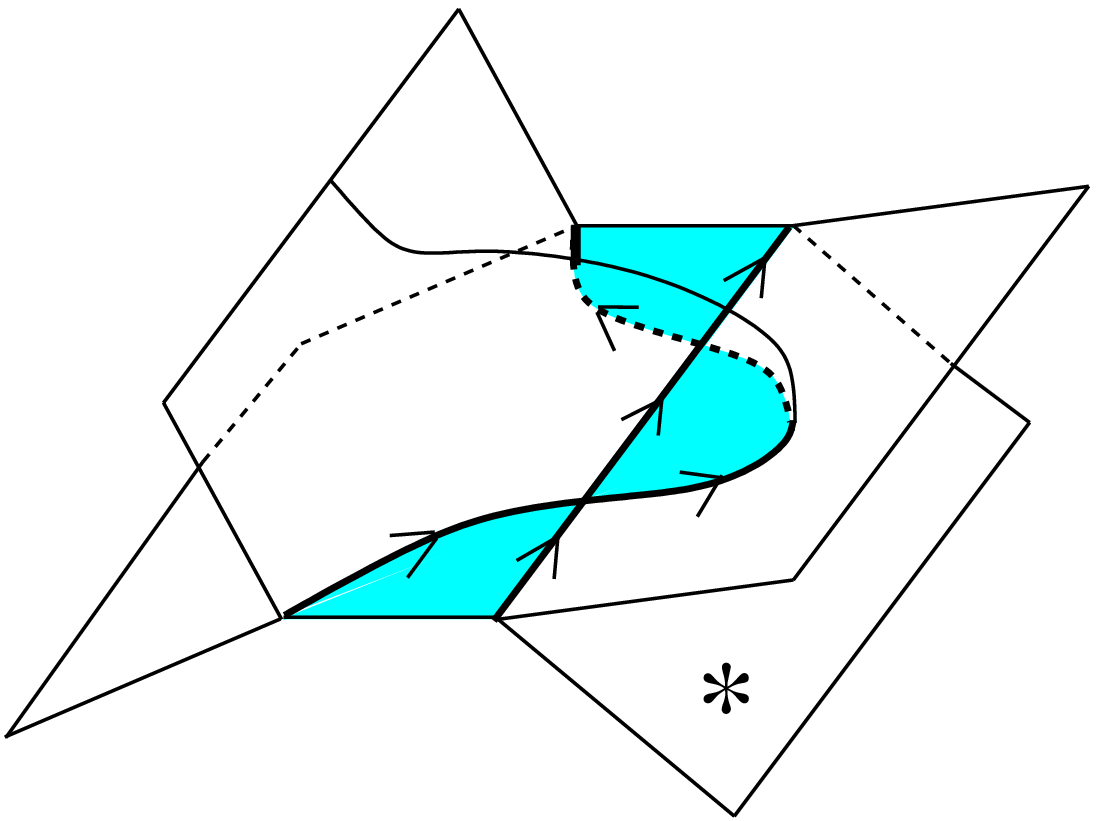},\quad
\figins{-24}{0.8}{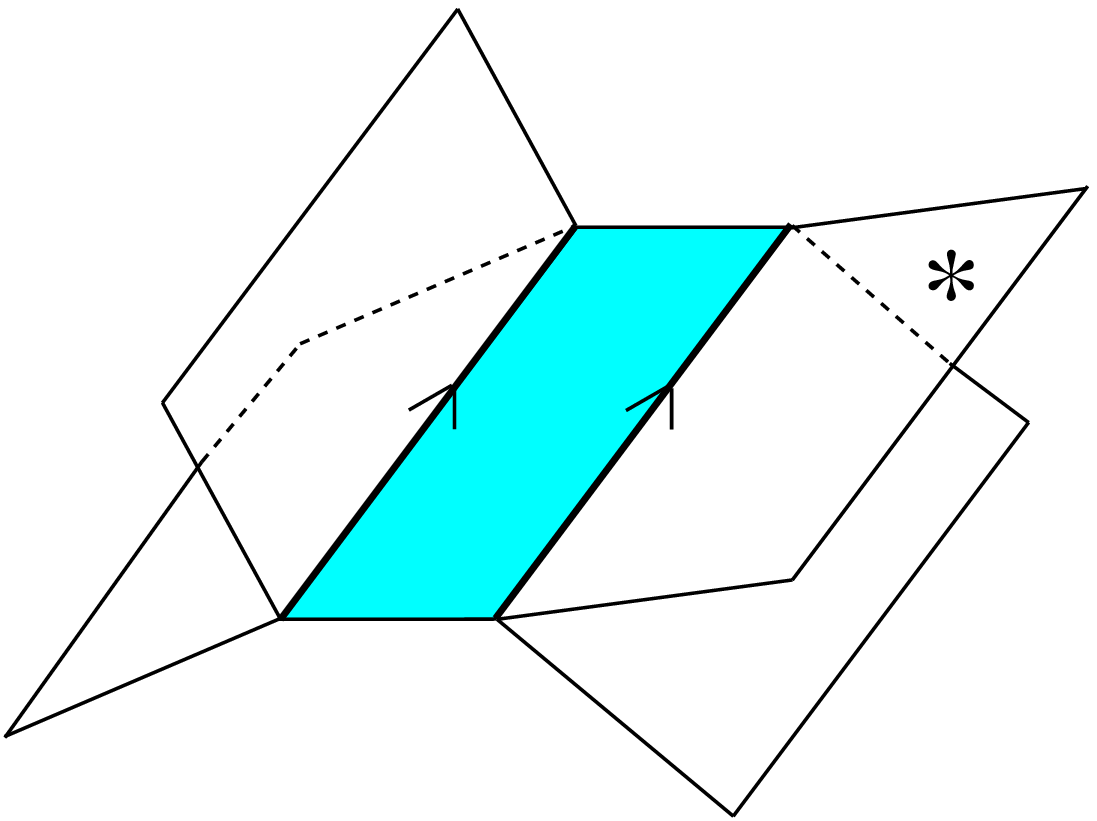} =
\figins{-24}{0.8}{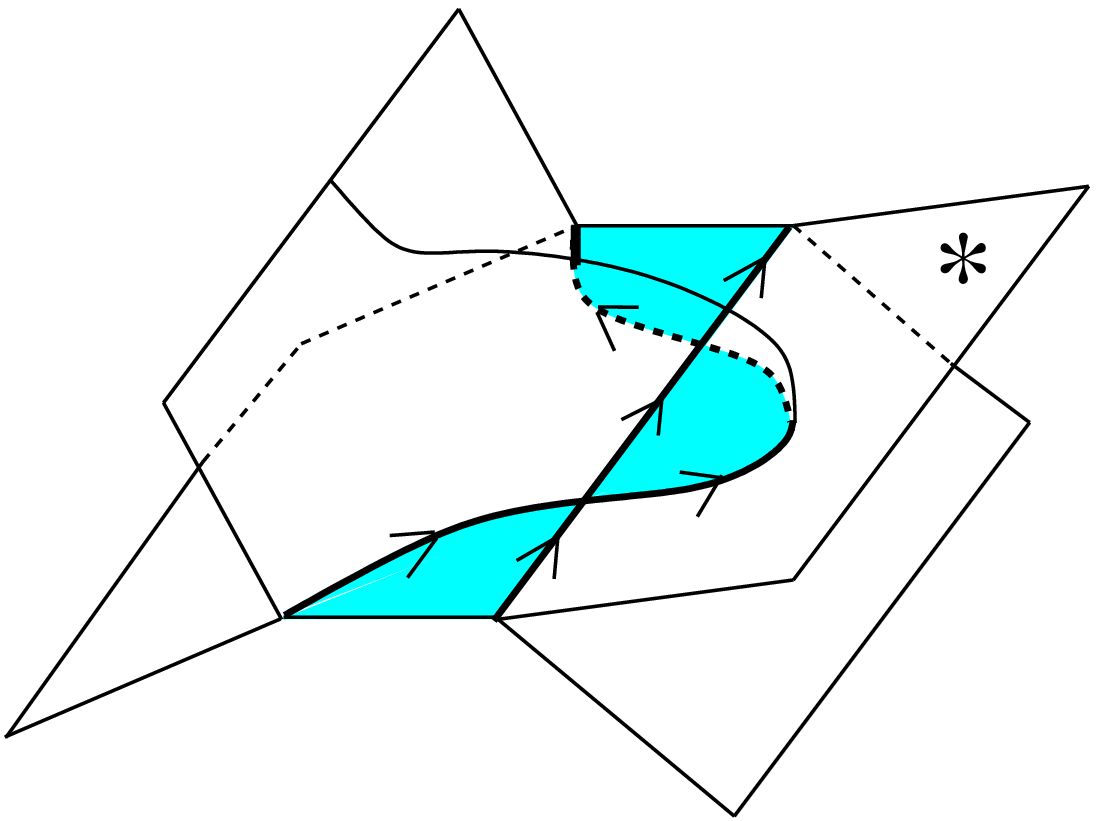}.
\qquad\text{(MP)}$$
\end{prop}

\begin{proof}
The dot conversion and migration relations, the sphere relations, 
the theta foam relations have already been proved in section~\ref{sec:KL}. 

The cutting neck relations are special cases of formula (5.68) in 
\cite{KR-LG}, where $O_j$ and $O_j^*$ can be read off from our equations (\ref{eq:db}).

The Matveev-Piergalini (MP) relation is an immediate consequence of the choice of input 
for the singular vertices. Note that in this relation there are always two singular 
vertices of different type. The elements in the $\Ext$-groups associated to those two 
types of singular vertices are inverses of each other, which implies exactly the 
(MP) relation by the glueing properties explained in subsection~\ref{sec:glueing}.  
\end{proof}

The following identities are a consequence of the dot and the theta relations. 

\begin{lem} 
\label{lem:theta-pqrkli}
$$
\figins{-17}{0.55}{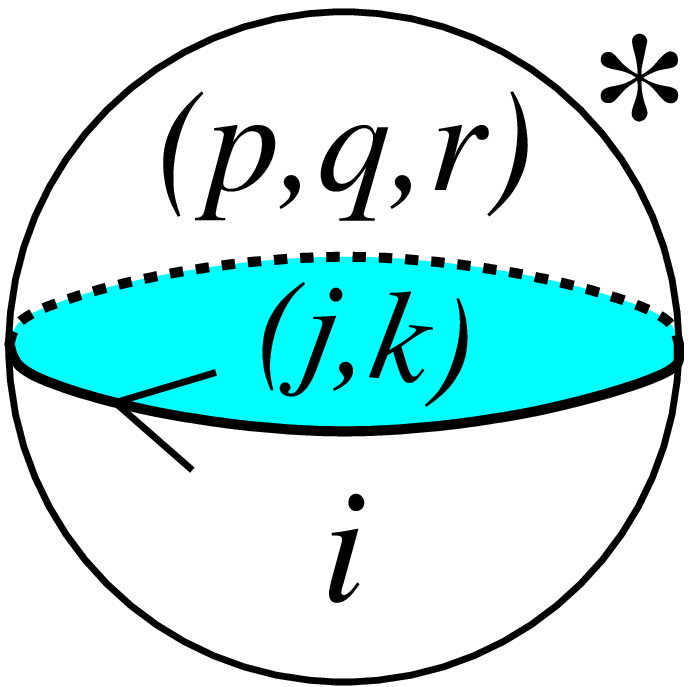}=
\begin{cases}
  -1&\text{if}\quad(p,q,r)=
(N-3-i,N-2-k,N-2-j)\\
-1&\text{if}\quad(p,q,r)=(N-3-k, N-3-j,N-1-i)\\
1&\text{if}\quad (p,q,r)=(N-3-k,N-2-i,N-2-j)\\
0&\text{else}
\end{cases}
$$
Note that the first three cases only make sense if 
\begin{eqnarray*}
&N-2\geq j\geq k\geq i+1
\geq 1\\
&N-1\geq i\geq j+2\geq k+2\geq 2\\
&N-2\geq j\geq i\geq k+1\geq 1
\end{eqnarray*}
respectively.  
\end{lem}

\begin{proof}
We denote the value of a theta foam by 
$\Theta(\pi_{p,q,r},\pi_{j,k},i)$. Since the $q$-degree of a non-decorated theta foam is 
equal to $-(N-1)-2(N-2)-3(N-3)=-(6N-14)$, we can have nonzero values of 
$\Theta(\pi_{p,q,r},\pi_{j,k},i)$ only if $p+q+r+j+k+i=3N-7$. Thus, if the 3-facet is not 
decorated, i.e. $p=q=r=0$, we have only four possibilities for the triple $(j,k,i)$ -- 
namely $(N-2,N-2,N-3)$, $(N-2,N-3,N-2)$, $(N-2,N-4,N-1)$ and 
$(N-3,N-3,N-1)$.
By Proposition~\ref{prop:principal rels1} we have 
$$\Theta(\pi_{0,0,0},\pi_{N-2,N-2},N-3)=-1.$$
However by dot migration, Lemma \ref{lem1} and the fact that 
$\pi_{p,q,r}=0$ if $p\ge N-2$, we have
\begin{eqnarray*}
0=\Theta(\pi_{N-2,N-2,N-3},\pi_{0,0},0)&=&\Theta(\pi_{0,0,0},
\pi_{N-2,N-2},N-3)+\Theta(\pi_{0,0,0},\pi_{N-2,N-3},N-2),\\
0=\Theta(\pi_{N-1,N-3,N-3},\pi_{0,0},0)&=&\Theta(\pi_{0,0,0},
\pi_{N-2,N-3},N-2)+\Theta(\pi_{0,0,0},\pi_{N-3,N-3},N-1),\\
0=\Theta(\pi_{N-1,N-2,N-4},\pi_{0,0},0)&=&\Theta(\pi_{0,0,0},
\pi_{N-2,N-2},N-3)+\Theta(\pi_{0,0,0},\pi_{N-2,N-3},N-2)+\\
&&+\Theta(\pi_{0,0,0},\pi_{N-2,N-4},N-1).
\end{eqnarray*}
Thus, the only nonzero values of the theta foams, when the 3-facet is nondecorated are
$$\Theta(\pi_{0,0,0},\pi_{N-2,N-2},N-3)=\Theta(\pi_{0,0,0},
\pi_{N-3,N-3},N-1)=-1,$$
$$\Theta(\pi_{0,0,0},\pi_{N-2,N-3},N-2)=+1.$$

Now we calculate the values of the general theta foam. Suppose first that $i\le k$. 
Then we have 
\begin{equation}
\Theta(\pi_{p,q,r},\pi_{j,k},i)=\Theta(\pi_{p,q,r},\pi_{i,i}\pi_{j-i,k-i},i)=
\Theta(\pi_{p+i,q+i,r+i},\pi_{j-i,k-i},0),
\label{th1}
\end{equation}
by dot migration. In order to calculate $\Theta(\pi_{x,y,z},\pi_{w,u},0)$
for $N-3\ge x\ge y\ge z \ge 0$ and $N-2 \ge w\ge u\ge 0$, we use Lemma \ref{lem1}. By 
dot migration we have
\begin{equation}
\Theta(\pi_{x,y,z},\pi_{w,u},0)=\sum_{(a,b,c)\sqsubset(x,y,z)}\Theta(\pi_{0,0,0},\pi_{w,u}\pi_{a,b},c).
\label{f1}
\end{equation}
Since $c\le p\le N-3$, a summand on the r.h.s. of (\ref{f1}) can be nonzero only 
for $c=N-3$ and $a$ and $b$ such that $\pi_{N-2,N-2}\in \pi_{w,u}\pi_{a,b}$, i.e. $a=N-2-u$ and $b=N-2-w$. Hence the value of (\ref{f1}) is equal to $-1$ if 
\begin{equation}
(N-2-u,N-2-w,N-3)\sqsubset(x,y,z),
\label{f2}
\end{equation}
and $0$ otherwise. Finally, (\ref{f2}) is equivalent to $x+y+z+w+u=3N-7$, $x\ge N-2-u\ge y$, $y\ge N-2-w\ge z$ and $x\ge N-3 \ge z$, and so we must have $u>0$ and
\begin{eqnarray*}
x&=&N-3,\\
y&=&N-2-u,\\
z&=&N-2-w.
\end{eqnarray*}
Going back to (\ref{f1}), we have that the value of theta is equal to $0$ if $l=i$, and in the case $l>i$ it is nonzero (and equal to $-1$) iff
\begin{eqnarray*}
p&=&N-3-i,\\
q&=&N-2-k,\\
r&=&N-2-j,
\end{eqnarray*}
which gives the first family.

Suppose now that $k<i$. As in (\ref{th1}) we have
\begin{equation}
\Theta(\pi_{p,q,r},\pi_{j,k},i)=\Theta(\pi_{p+k,q+k,r+k},\pi_{j-k,0},i-k).
\label{th2}
\end{equation}
Hence, we now concentrate on $\Theta(\pi_{x,y,z},\pi_{w,0},u)$
for $N-3\ge x\ge y\ge z \ge 0$, $N-2 \ge w\ge 0$ and $N-1\ge u\ge 1$. Again, by using Lemma \ref{lem1} we have
\begin{equation}
\Theta(\pi_{x,y,z},\pi_{w,0},0)=\sum_{(a,b,c)\sqsubset(x,y,z)}\Theta(\pi_{0,0,0},\pi_{w,0}\pi_{a,b},u+c).
\label{f3}
\end{equation}
Since $a\le N-3$, we cannot have $\pi_{N-2,N-2}\in \pi_{w,0}\pi_{a,b}$ and we can have 
$\pi_{N-2,N-3}\in \pi_{w,0}\pi_{a,b}$ iff $a=N-3$ and $b=N-2-w$. In this case we have a nonzero 
summand (equal to $1$) iff $c=N-2-u$. Finally $\pi_{N-3,N-3}\in \pi_{w,0}\pi_{a,b}$ iff 
$a=N-3$ and $b=N-3-w$. In this case we have a nonzero summand (equal to $-1$) iff $c=N-1-u$. 
Thus we have a summand on the r.h.s. of (\ref{f3}) equal to $+1$ iff 
\begin{equation}
(N-3,N-2-w,N-2-u)\sqsubset (x,y,z),
\label{rel1}
\end{equation} 
and a summand equal to $-1$ iff 
\begin{equation}
(N-3,N-3-w,N-1-u)\sqsubset (x,y,z).
\label{rel2}
\end{equation} 
Note that in both above cases we must have $x+y+z+w+u=3N-7$, $x=N-3$ and $u\ge 1$. Finally, 
the value of the r.h.s of (\ref{f3}) will be nonzero iff exactly one of (\ref{rel1}) and 
(\ref{rel2}) holds.

In order to find the value of the sum on r.h.s. of (\ref{f3}), we split the rest of the proof in three cases according to the relation between $w$ and $u$.

If $w\ge u$, (\ref{rel1}) is equivalent to $y\ge N-2-w$, $z\le N-2-w$, while (\ref{rel2}) is 
equivalent to $y\ge N-3-w$, $z\le N-3-w$ and $u\ge 2$. Now, we can see that the sum is nonzero 
and equal to $1$ iff $z=N-2-w$ and so $y=N-2-u$. Returning to (\ref{th2}), we have that 
the value of $\Theta(\pi_{p,q,r},\pi_{j,k},i)$ is equal to $1$ for 
\begin{eqnarray*}
p&=&N-3-k,\\
q&=&N-2-i,\\
r&=&N-2-j,
\end{eqnarray*}
for $N-2\ge j \ge i \ge k+1 \ge 1$, which is our third family.

If $w\le u-2$, (\ref{rel1}) is equivalent to $y\ge N-2-w$, $z\le N-2-u$ and $u\le N-2$ while (\ref{rel2}) is equivalent to $y\ge N-3-w$, $z\le N-1-u$. Hence, in this case we have that the total sum is nonzero and equal to $-1$ iff $y=N-3-w$ and $z=N-1-u$, which by returning to (\ref{th2}) gives that the value of  $\Theta(\pi_{p,q,r},\pi_{j,k},i)$ is equal to $-1$ for\begin{eqnarray*}
p&=&N-3-k,\\
q&=&N-3-j,\\
r&=&N-1-i,
\end{eqnarray*}
for $N-3\ge i-2\ge j \ge k \ge 0$, which is our second family.

Finally, if $u=w+1$ (\ref{rel1}) becomes equivalent to $y\ge N-2-w$ and $z\le N-3-w$, 
while (\ref{rel2}) becomes $y\ge N-3-w$ and $z\le N-3-w$. Thus, in order to have a nonzero sum, 
we must have $y=N-3-w$. But in that case, because of the fixed total sum of indices, we 
would have $z=N-1-u=N-2-w>N-3-w$, which contradicts (\ref{rel2}). Hence, in this case, 
the total value of the theta foam is $0$. 
\end{proof}
As a direct consequence of the previous theorem, we have
\begin{cor}
\label{cor:theta-pqrkli}
For fixed values of $j$, $k$ and $i$, if $j\ne i-1$ and $k\ne i$, there is exactly one triple $(p,q,r)$ such that the value of $\Theta(\pi_{p,q,r},\pi_{j,k},i)$ is nonzero. Also, if $j=i-1$ or 
$k=i$, the value of $\Theta(\pi_{p,q,r},\pi_{j,k},i)$ is equal to $0$ for every triple $(p,q,r)$. 
Hence, for fixed $i$, there are $n-1\choose 2$ 5-tuples $(p,q,r,j,k)$ such that 
$\Theta(\pi_{p,q,r},\pi_{j,k},i)$ is nonzero.

Conversely, for fixed $p$, $q$ and $r$, there always exist three different triples $(j,k,i)$ 
(one from each family), such that $\Theta(\pi_{p,q,r},\pi_{j,k},i)$ is nonzero.

Finally, for all $p$, $q$, $r$, $j$, $k$ and $i$, we have
$$\Theta(\pi_{p,q,r},\pi_{j,k},i)=\Theta(\hat{\pi}_{p,q,r},\hat{\pi}_{j,k},N-1-i).$$
\end{cor} 
The following relations are an immediate consequence of 
Lemma~\ref{lem:theta-pqrkli}, 
Corollary~\ref{cor:theta-pqrkli} and (CN$_i$), $i=1,2,*$.
\begin{cor} 
\label{cor:bubble-star}

\begin{equation}
\label{eq:bubble-pqri}
\figins{-18}{0.6}{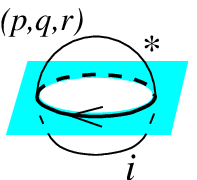}=
\begin{cases}
-\figins{-17}{0.4}{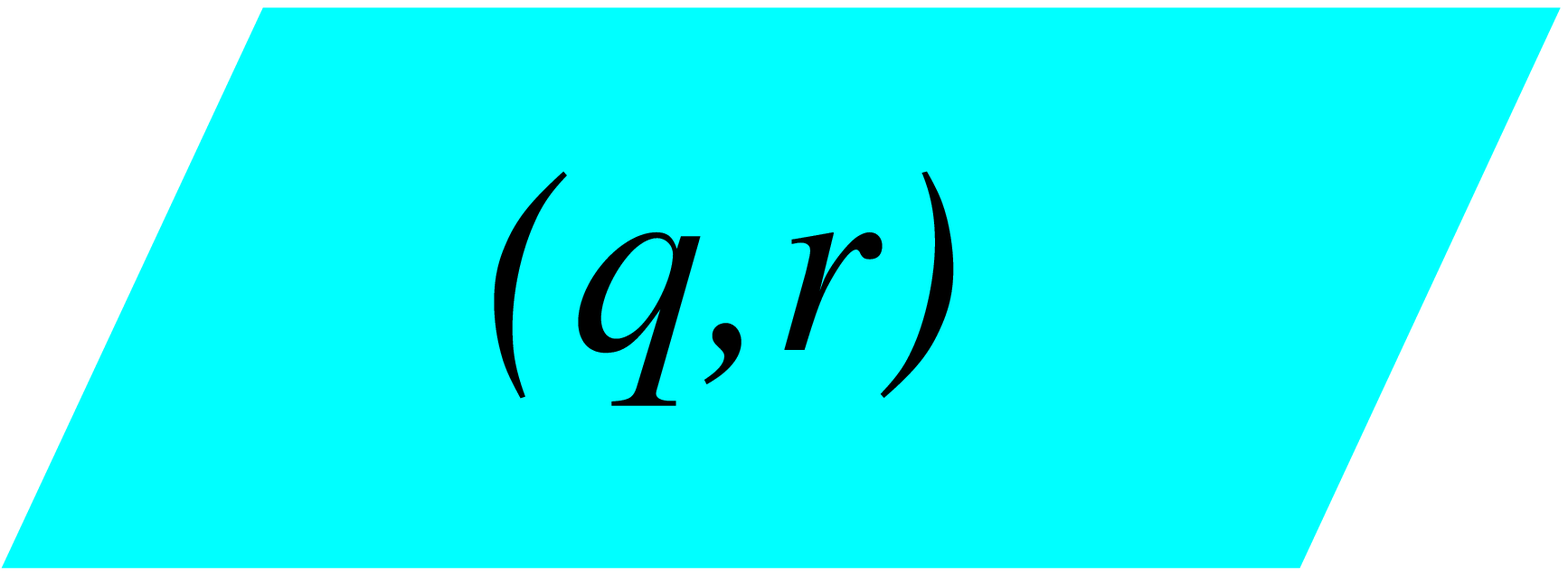}   &\text{if}\quad p=N-3-i\\ 
-\figins{-17}{0.4}{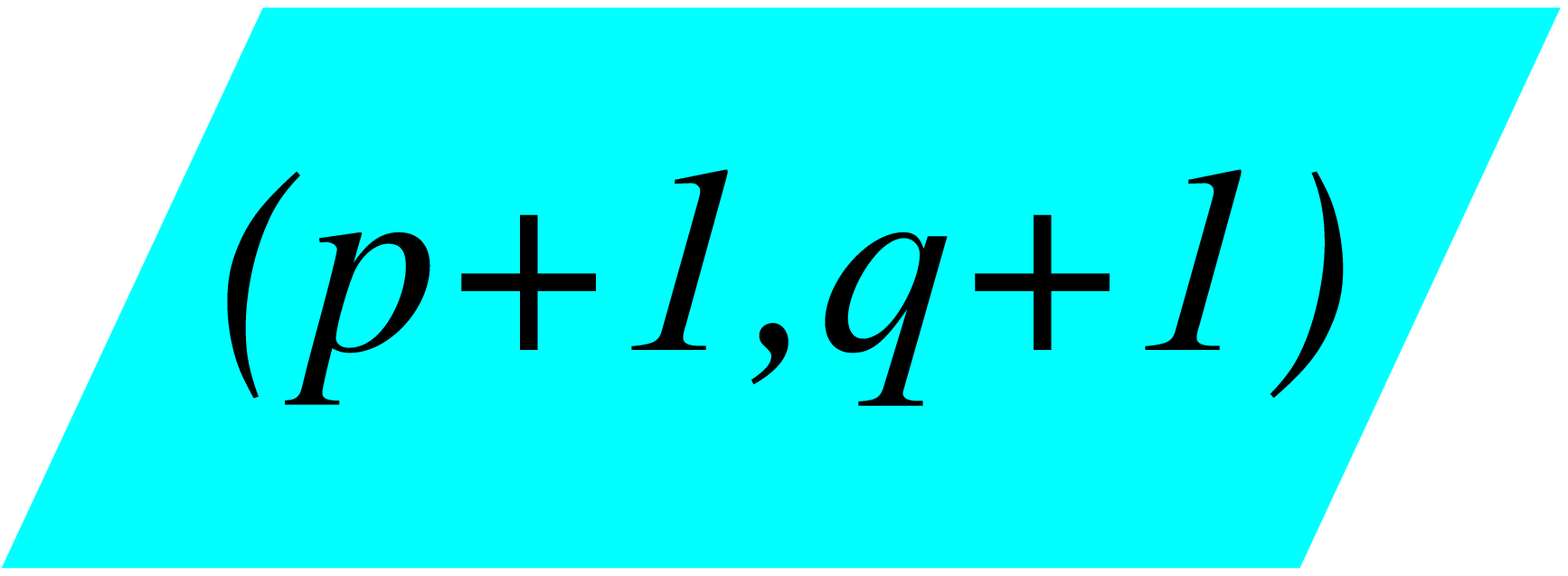} &\text{if}\quad r=N-1-i\\ 
 \figins{-17}{0.4}{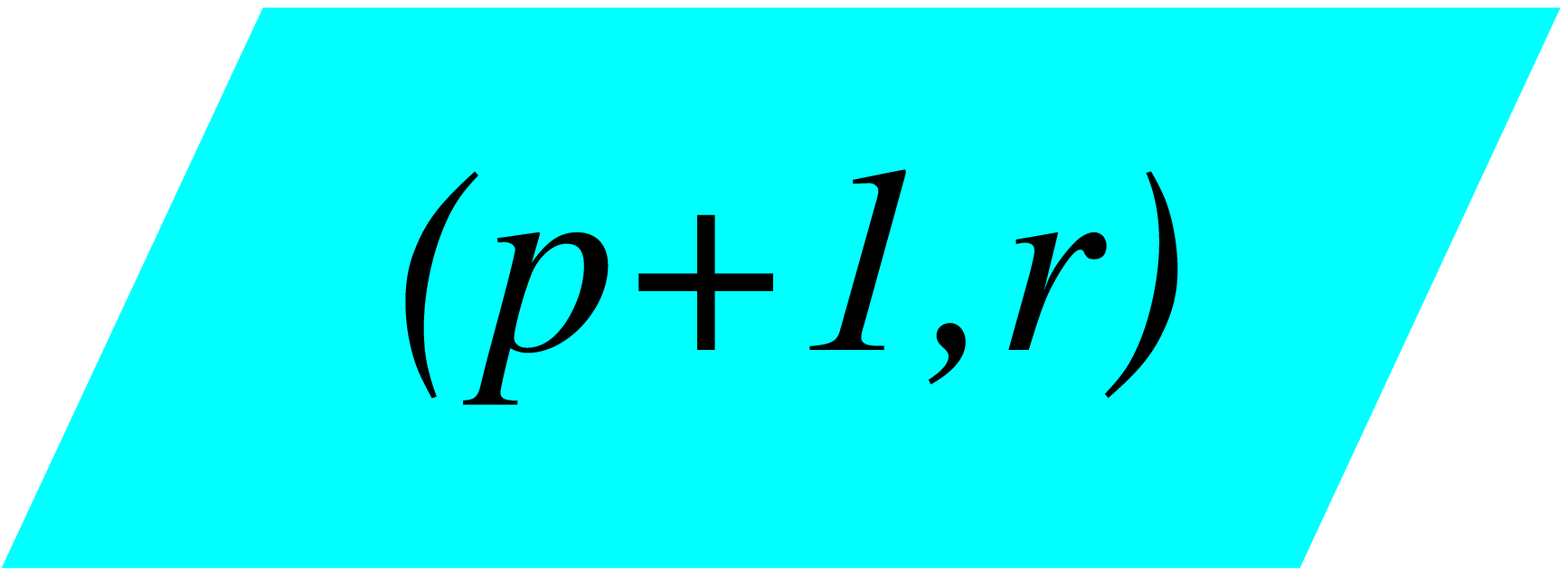}  &\text{if}\quad q=N-2-i\\
 \qquad 0                      &\text{else}
\end{cases}
\end{equation}

\bigskip

\begin{equation}
\label{eq:bubble-ikm-star}
\figins{-18}{0.6}{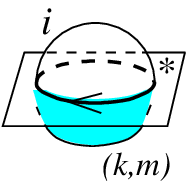}=
\begin{cases}
-\figins{-15}{0.4}{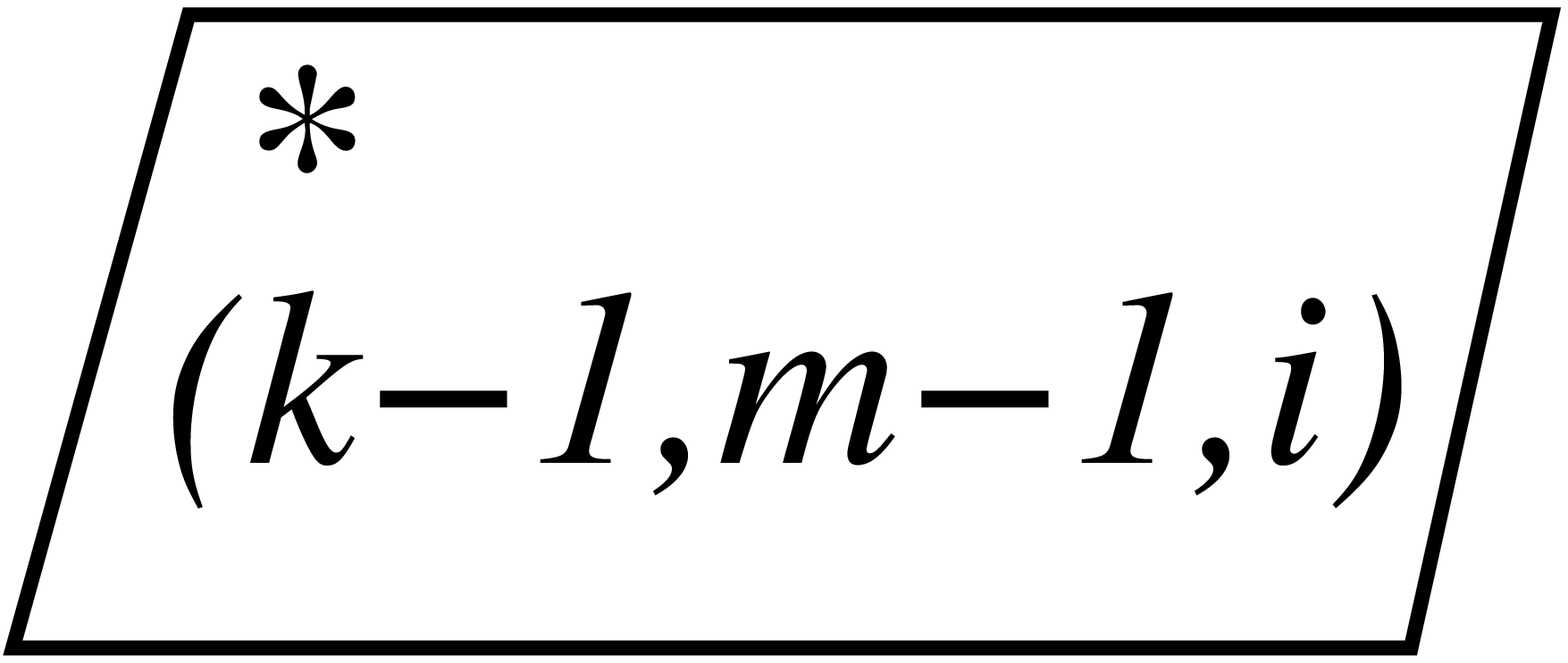} & \text{if}\quad N-2\geq k\geq m\geq i+1\geq 1\  \\
-\figins{-15}{0.4}{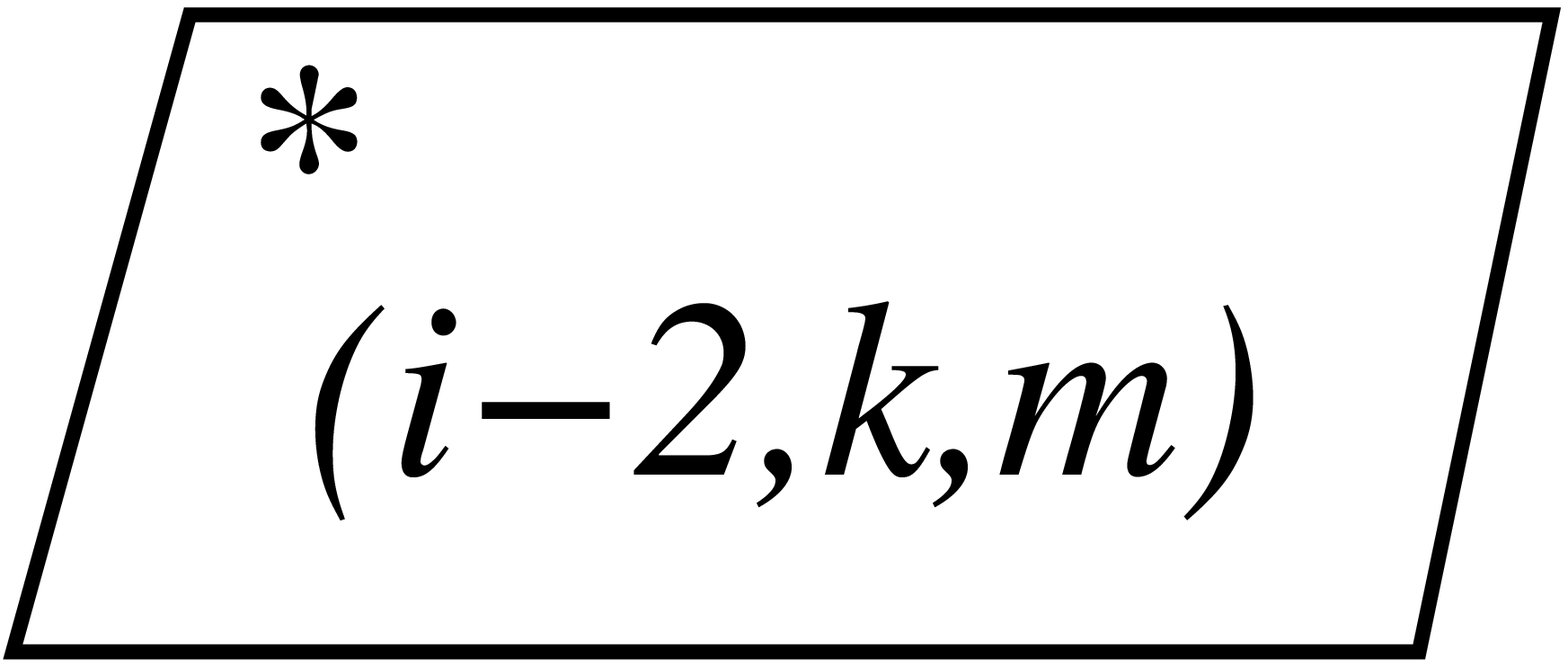} & \text{if}\quad N-1\geq i\geq k+2\geq m+2\geq 2\ \\ 
\ \ \ 
\figins{-15}{0.4}{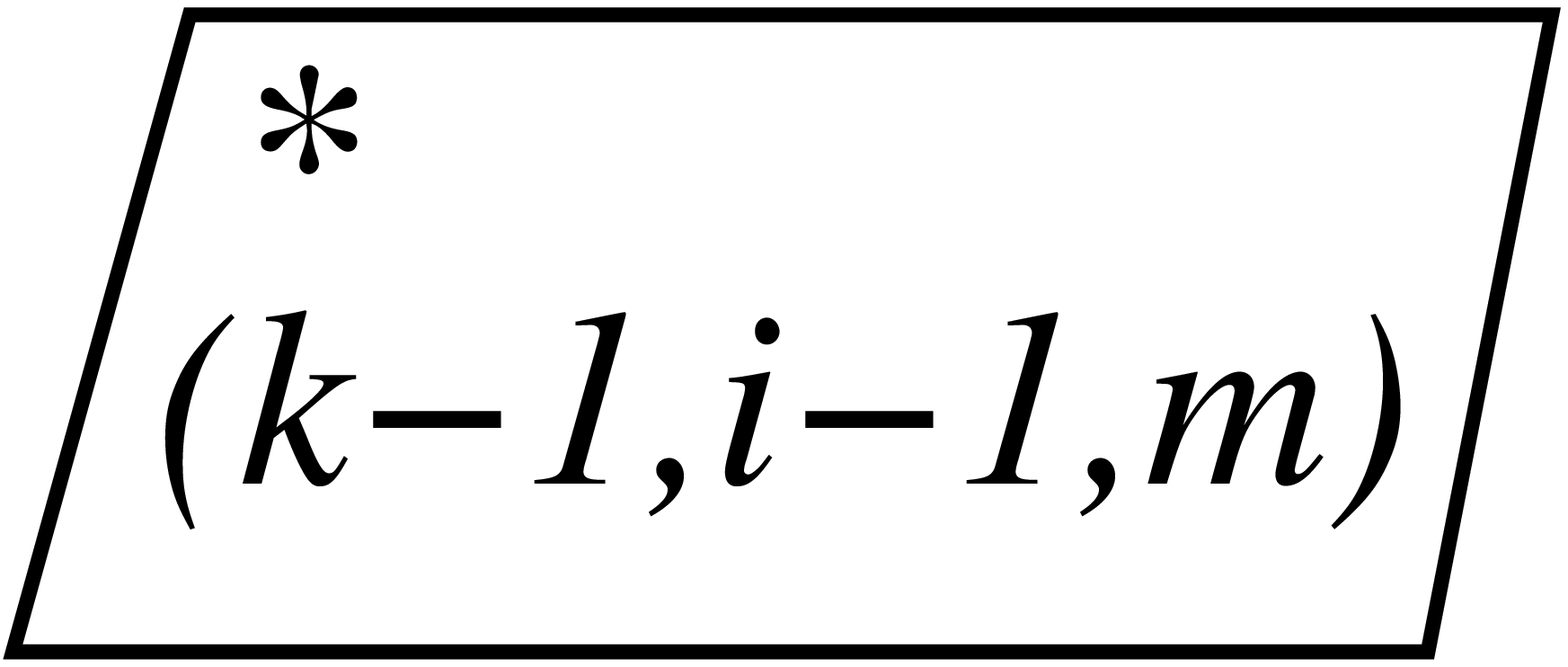} & \text{if}\quad N-2\geq k\geq i\geq m+1\geq 1\ \\ 
\qquad 0 &\text{else}
\end{cases}
\end{equation}

\begin{equation}
\label{eq:bubble-pqrkm-star}
\figins{-18}{0.64}{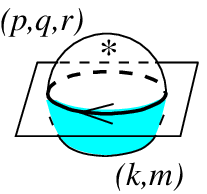}=
\begin{cases}\ \ \
- \figins{-13}{0.35}{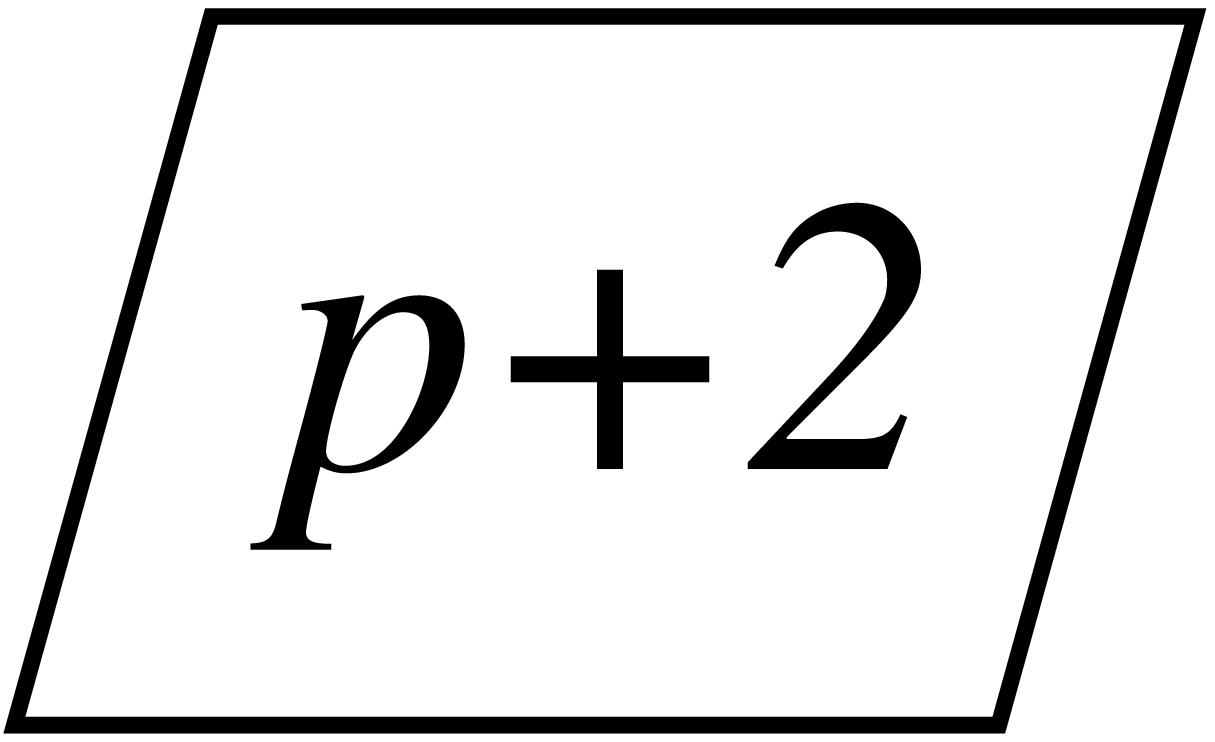} & \text{if}\quad q= N-2-m,\,\, r=N-2-k\  \\ \ \ \
- \figins{-13}{0.35}{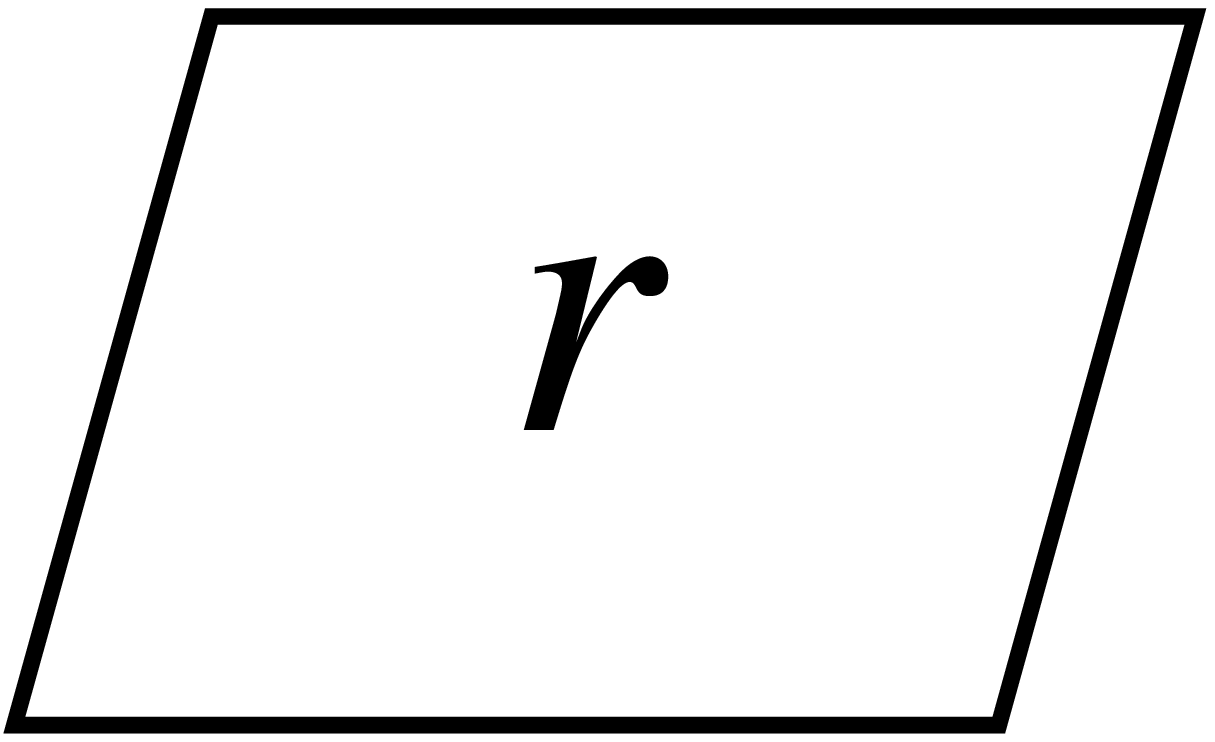} & \text{if}\quad p=N-3-m,\,\, q=N-3-k\ \\ 
\figins{-13}{0.35}{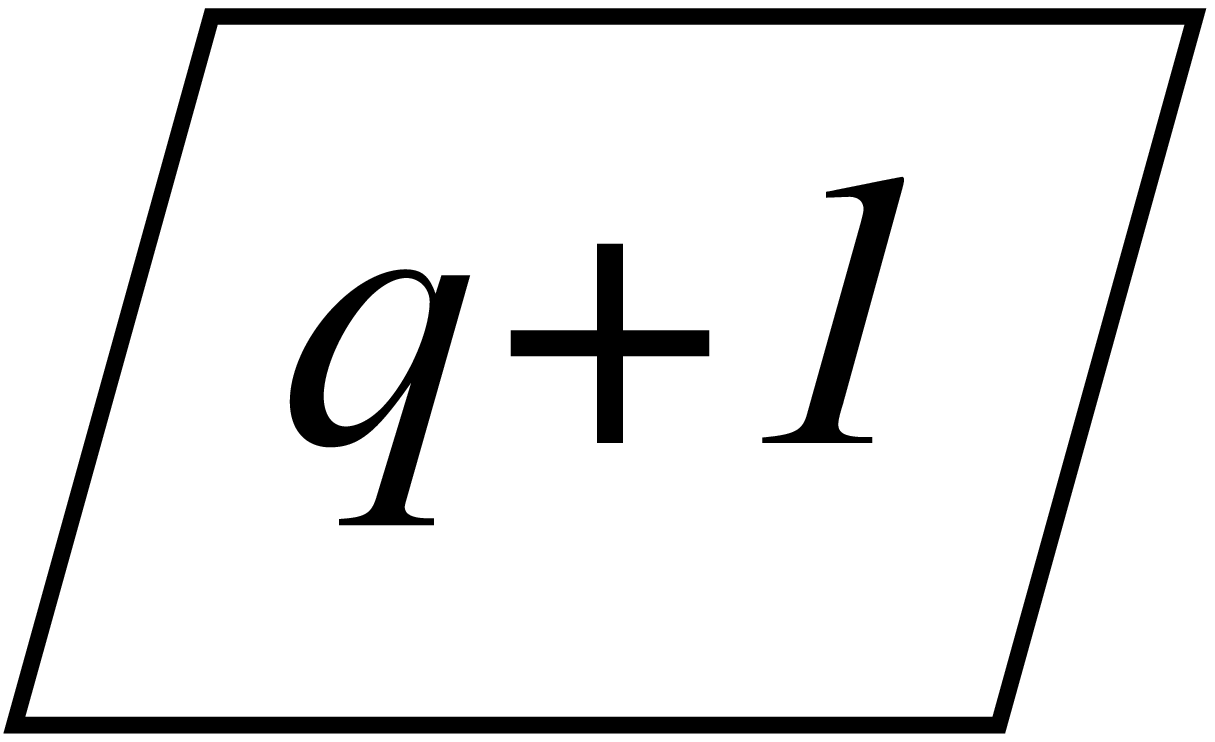} & \text{if}\quad p=N-3-m,\,\, r=N-2-k \\ 
\qquad 0 &\text{else}
\end{cases}
\end{equation}
\end{cor}

\begin{lem}
\label{lem:theta-ijkl}
$$\figins{-17}{0.55}{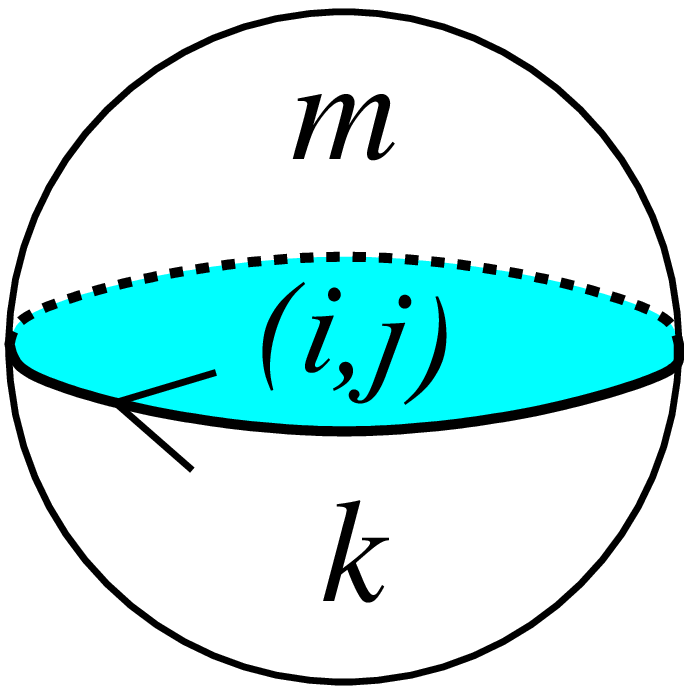}=
\begin{cases}
-1&\mbox{if}\quad m+j=N-1=i+k+1\\
+1&\mbox{if}\quad j+k=N-1=i+m+1, 
\end{cases}$$
\end{lem}
\begin{proof} By the dot conversion formulas, 
we get 
$$
\figins{-17}{0.55}{theta-mijkk}=\sum\limits_{\alpha =0}^{i-j}
\figins{-17}{0.55}{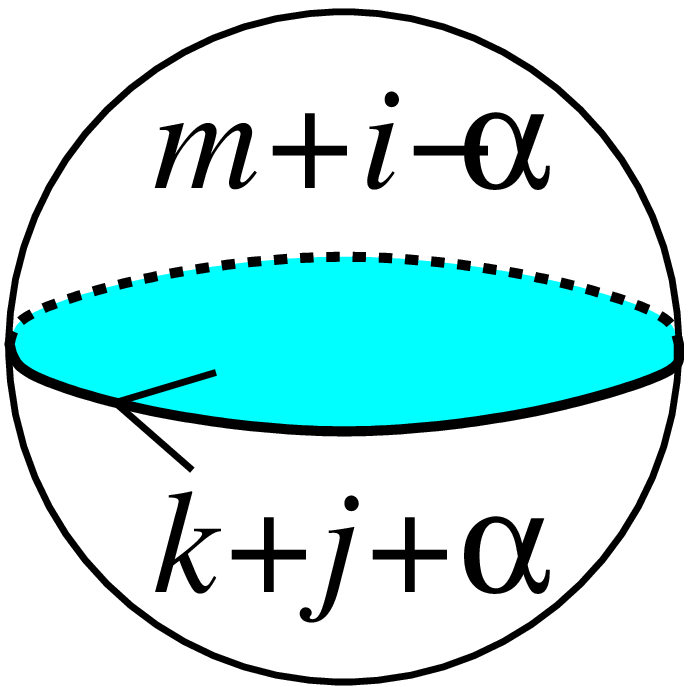}.$$
By $(\ThetaGraph)$ we have
$$\figins{-17}{0.55}{theta-miakja}=
\begin{cases}
-1&\text{if}\quad m+i-(N-1)=\alpha=N-2-(k+j)\\
+1&\text{if}\quad m+i-(N-2)=\alpha=N-1-(k+j)\\
0&\text{else}.
\end{cases}
$$
We see that, in the sum above, the summands for two consecutive values of $\alpha$ will cancel 
unless one of them is zero and the other is not. We see that the total sum is equal to $-1$ 
if the first non-zero summand is at $\alpha=i-j$ and $+1$ if the last non-zero summand 
is at $\alpha=0$.   
\end{proof}

The following bubble-identities are an immediate consequence of Lemma~\ref{lem:theta-ijkl} and 
(CN$_1$) and (CN$_2$).

\begin{cor}
\label{cor:bubble-12}
\begin{equation}
\label{eq:bubble-ij}
\figins{-20}{0.6}{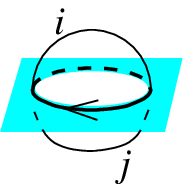}=
\begin{cases}
\quad -\figins{-13}{0.35}{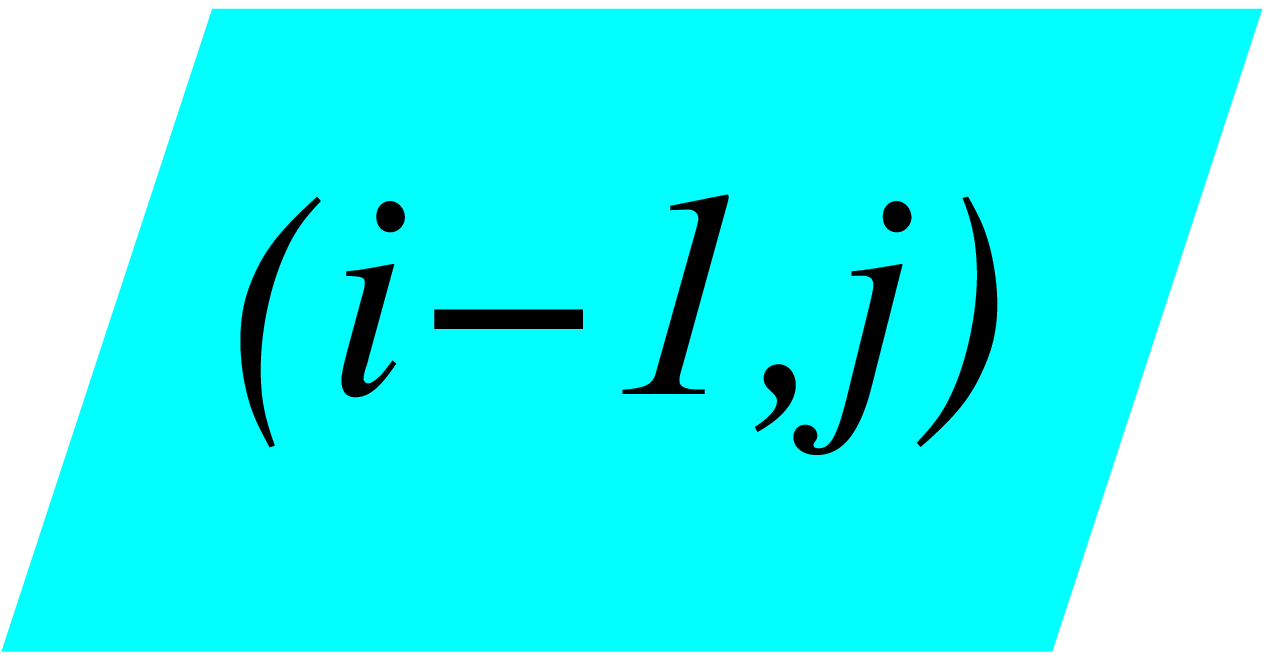} & \text{if}\quad i>j\geq 0 \\
\figins{-13}{0.35}{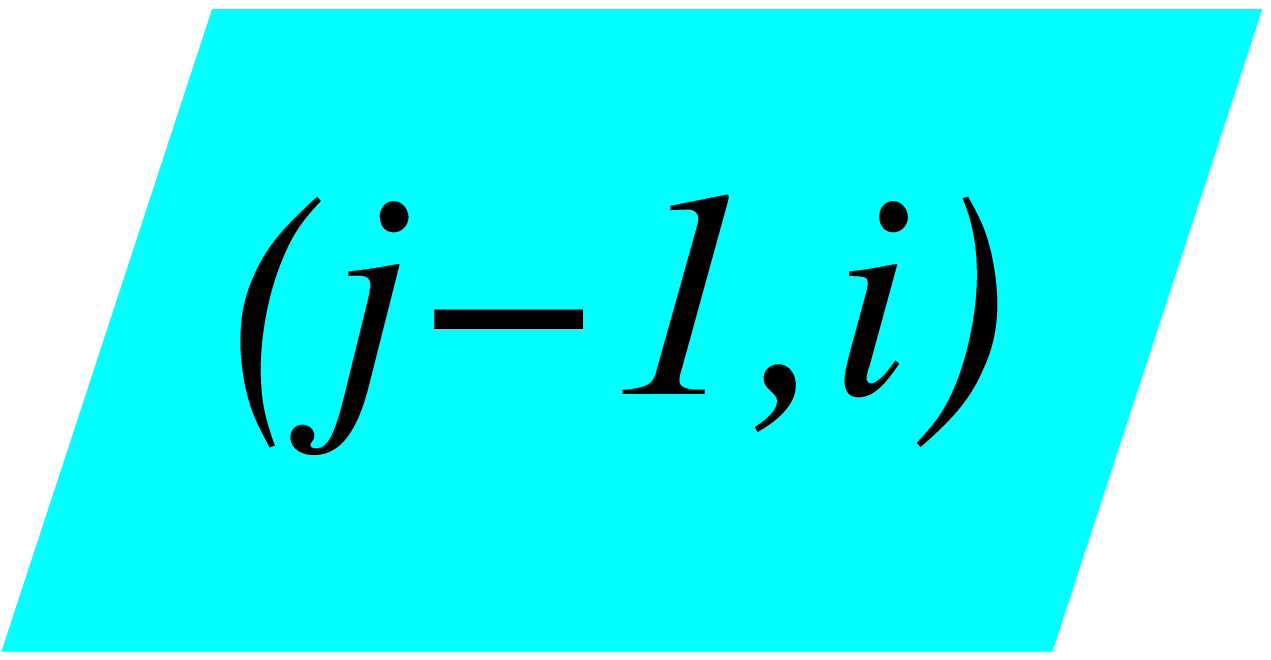} & \text{if}\quad j>i\geq 0 \\
\qquad 0 & \text{if}\quad i=j
\end{cases} 
\end{equation}

\begin{equation}
\label{eq:bubble-ikj}
\figins{-18}{0.6}{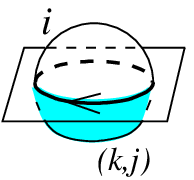}=
\begin{cases}
-   \figins{-8}{0.3}{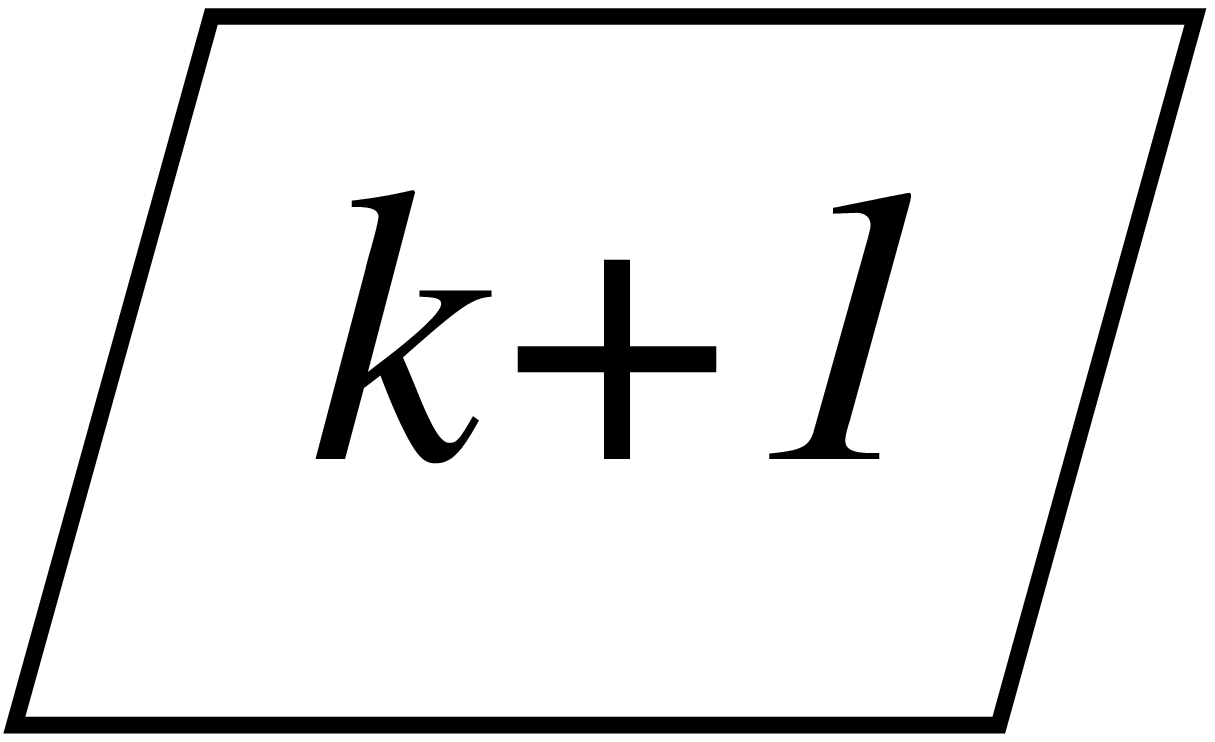} & \text{if }\  i+j=N-1 \\ \\
   \figins{-8}{0.3}{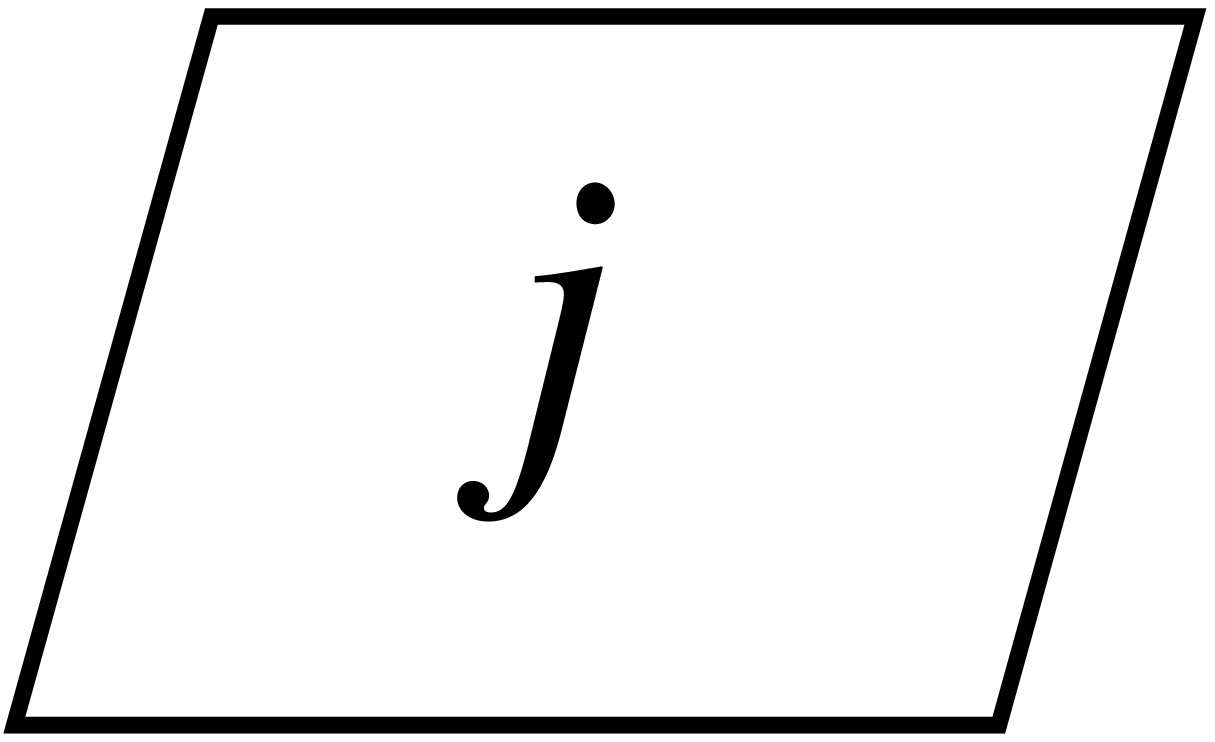} & \text{if }\ i+k=N-2   \\ \\
 \quad\     0 & \text{ else }
\end{cases}
\end{equation}
\end{cor}

The following identities follow easily from (\emph{CN}$_1$), (\emph{CN}$_2$),  
Lemma~\ref{lem:theta-pqrkli}, Lemma~\ref{lem:theta-ijkl} and their corollaries.

\begin{cor}
\label{cor:tubes}

$$
\figins{-22}{0.6}{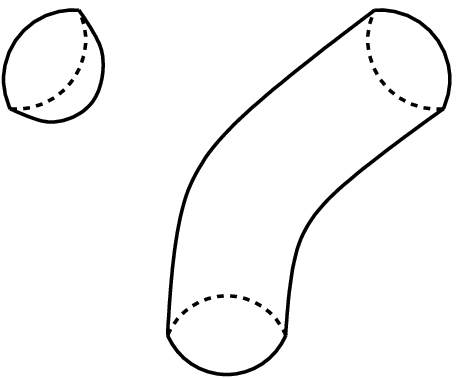} + \sum_{a+b+c=N-2} 
\figins{-22}{0.6}{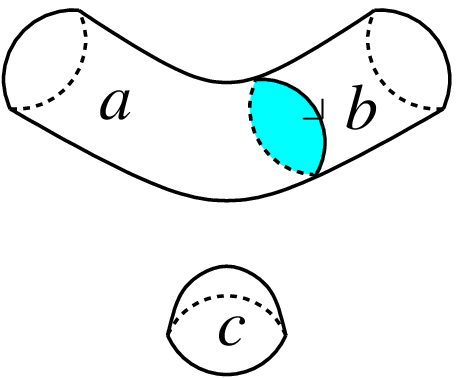} =
\figins{-22}{0.6}{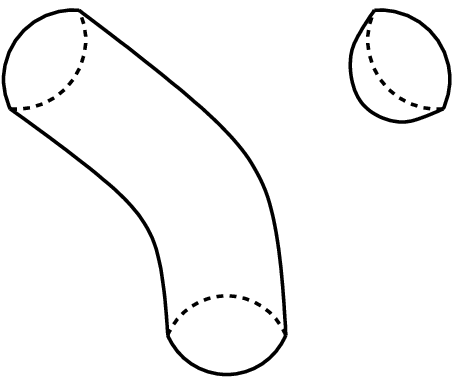}
\rlap{\hspace{11ex}\text{\emph{(3C)}}}
$$

\medskip


$$
\figins{-8}{0.3}{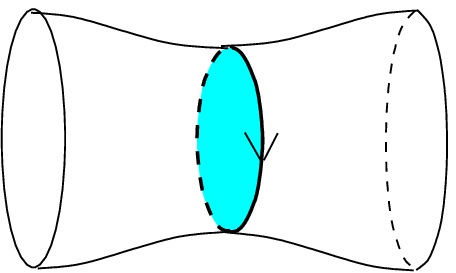} \ =  
\figins{-8}{0.3}{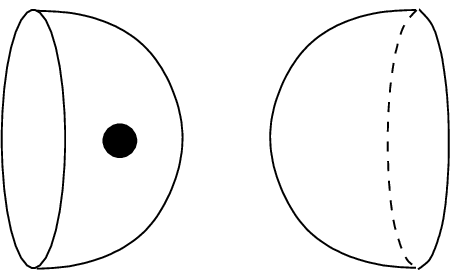} \ - \ 
\figins{-8}{0.3}{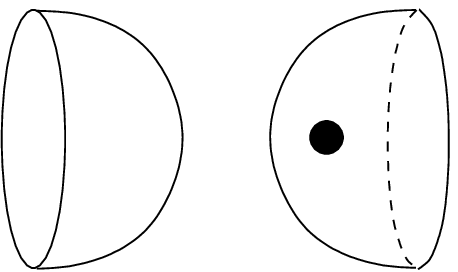}
\rlap{\hspace{19.5ex}\text{(RD$_1$)}}
$$

\medskip

$$
\figins{-8}{0.3}{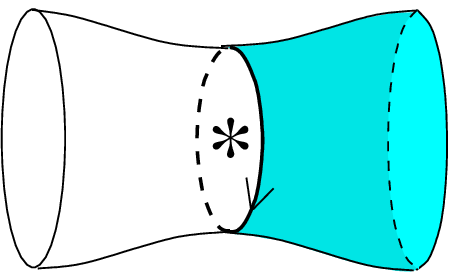} \ =  
\figins{-8}{0.3}{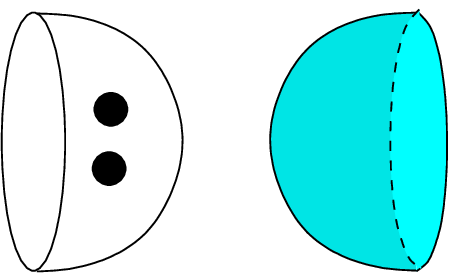} \ - \ 
\figins{-8}{0.3}{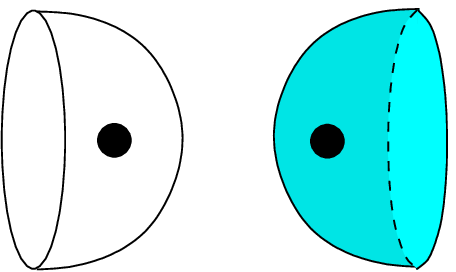} \ + \ 
\figins{-8}{0.3}{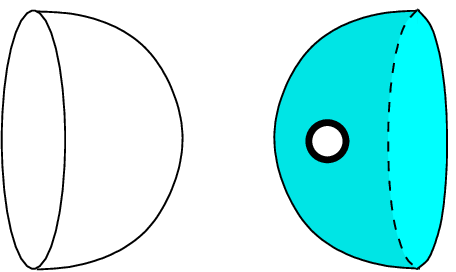}
\rlap{\hspace{14ex}\text{(RD$_2$)}}
$$

\medskip


$$
\figins{-8}{0.3}{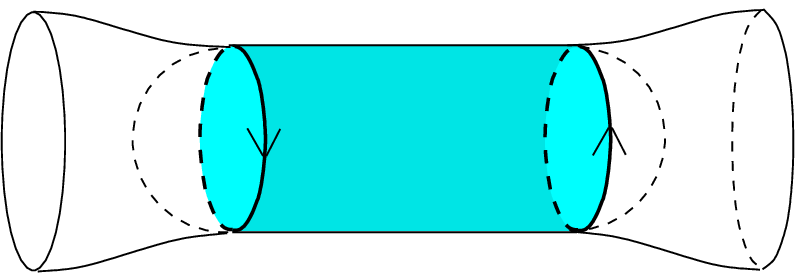} \ = - \ 
\figins{-8}{0.3}{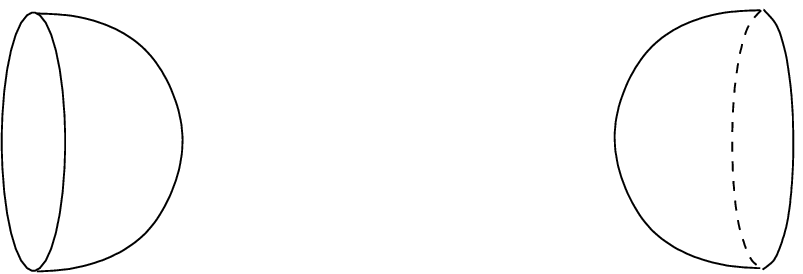}
\rlap{\hspace{18.5ex}\text{(FC)}}
$$

\end{cor}

\bigskip

Note that by the results above we are able to compute 
$\langle u\rangle_{KL}$ combinatorially, for any 
closed foam $u$ whose singular graphs has no vertices, 
simply by using the cutting neck relations near all singular 
circles and evaluating the resulting spheres and theta 
foams. If the singular graph of $u$ has vertices, then we 
do not know if our relations are sufficient to evaluate $u$. 
We conjecture that they are sufficient, and that therefore 
our theory is strictly combinatorial, but we do not have 
a complete proof.  

\begin{prop}
\label{prop:principal rels2}
The following identities hold in $\foam$: 

\medskip

(The {\em digon removal} relations)
$$
\figins{-20}{0.6}{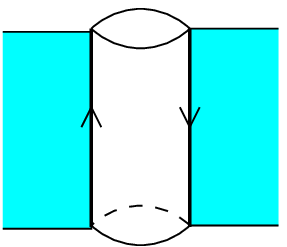}=
\figins{-20}{0.6}{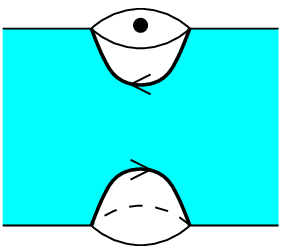}-
\figins{-20}{0.6}{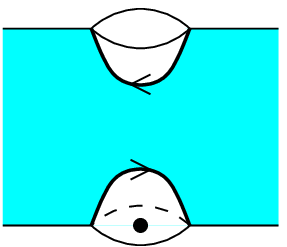}
\rlap{\hspace{16ex}\text{(DR$_1$)}}
$$

\medskip

$$
\figins{-27}{0.7}{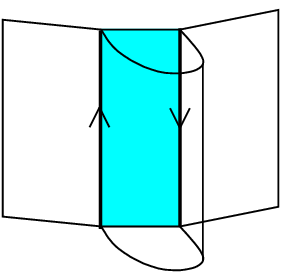} = \sum_{a+b+c=N-2}\
\figins{-27}{0.7}{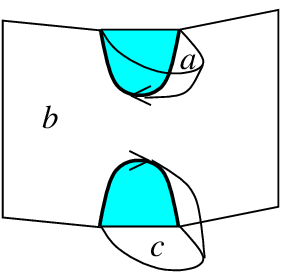} \ = \sum_{i=0}^{N-2}\
\figins{-27}{0.7}{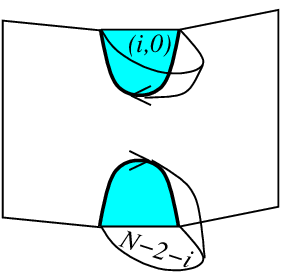}
\rlap{\hspace{7ex}\text{(DR$_2$)}}
$$

\medskip

$$
\figins{-30}{0.7}{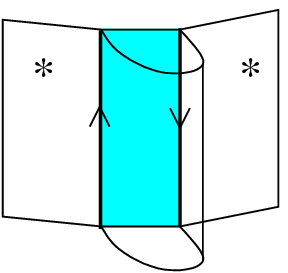} =\
-\figins{-30}{0.7}{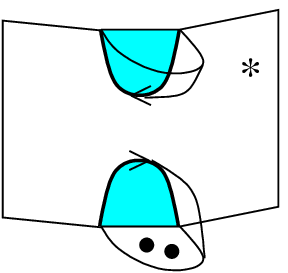}+
\figins{-30}{0.7}{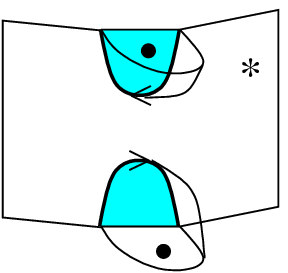}-
\figins{-30}{0.7}{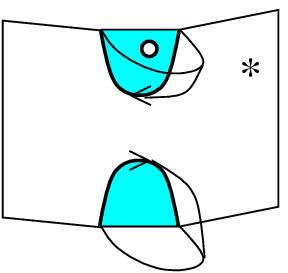}
\rlap{\hspace{7.5ex}\text{(DR$_{3_1}$)}}
$$

\medskip

$$
\figins{-30}{0.7}{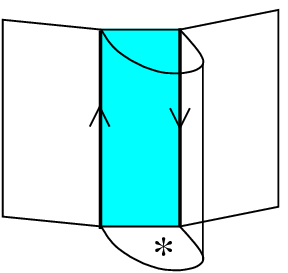} = - \sum_{0\leq j\leq i\leq N-3}\
\figins{-30}{0.7}{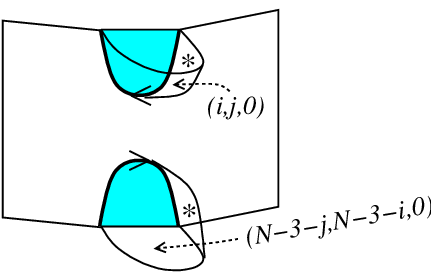}
\rlap{\hspace{13ex}\text{(DR$_{3_2}$)}}
$$

\medskip

$$
\figins{-20}{0.6}{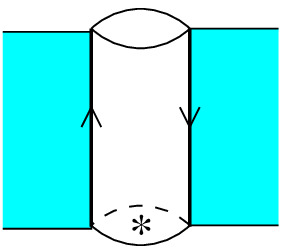}=
- \sum_{i=0}^{N-3} 
\figins{-20}{0.6}{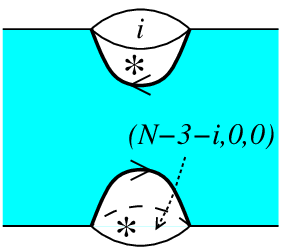}
\rlap{\hspace{18.2ex}\text{(DR$_{3_3}$)}}
$$

\bigskip

(The {\em first square removal} relation)

$$
\figins{-31}{0.9}{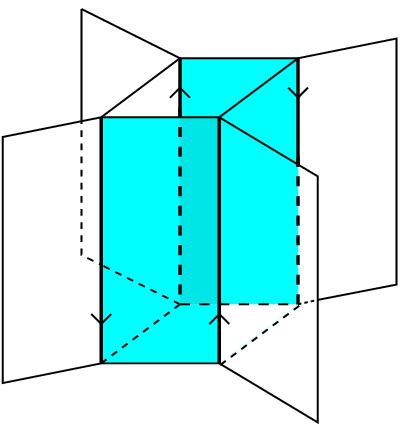}= 
-\;\figins{-31}{0.9}{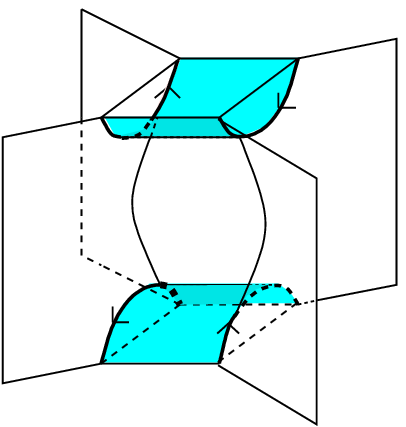}\quad + \sum_{a+b+c+d=N-3}\
\figins{-31}{0.9}{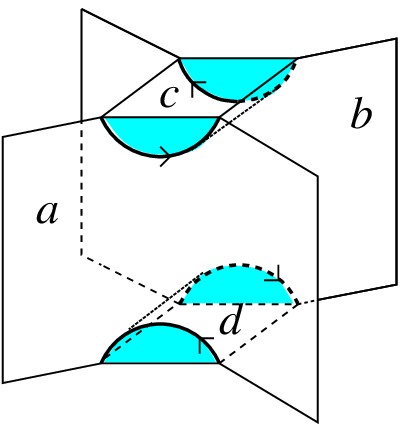}
\rlap{\hspace{4ex}\text{(SqR$_1$)}}
$$

\end{prop}

\begin{proof}
We first explain the idea of the proof. Naively one could try to consider 
all closures of the foams in a relation and compare their KL-evaluations. However, in practice 
we are not able to compute the KL-evaluations of arbitrary closed foams 
with singular vertices. Therefore we use a different strategy. We consider any foam in the 
proposition as a foam from $\emptyset$ to its boundary, rather than as a foam between one part 
of its boundary to another part. If $u$ is such a foam whose 
boundary is a closed web $\Gamma$, then by the properties explained in Section~\ref{sec:KL} 
the KL-formula associates to $u$ an element in $H^*(\Gamma)$, which is the homology of the 
complex associated to $\Gamma$ in~\cite{KR}. By Definition~\ref{defn:foam} and by the 
glueing properties of the KL-formula, as explained in Section~\ref{sec:KL}, the induced 
linear map from $\langle\ \vert_{KL}\colon \foam(\emptyset,\Gamma)\to H^*(\Gamma)$ is injective. 
The KL-formula also defines an inner product on $\foam(\emptyset,\Gamma)$ by 
$$(u,v)\mapsto \langle u\hat{v}\rangle_{KL}.$$
By $\hat{v}$ we mean the foam in $\foam(\Gamma,\emptyset)$ obtained by rotating 
$v$ along the axis which corresponds to the $y$-axis (i.e. the horizontal axis parallel to the 
computer screen) in the original picture in this 
proposition. By the 
results in~\cite{KR} we know 
the dimension of $H^*(\Gamma)$. Suppose it is equal to $m$ 
and that we can find two sets of elements $u_i$ and $u_i^*$ in $\foam(\emptyset,\Gamma)$, 
$i=1,2,\ldots,m$, such that 
$$\langle u_i\widehat{u_j^*}\rangle_{KL}=\delta_{i,j},$$ 
where $\delta_{i,j}$ is the Kronecker delta. 
Then $\{u_i\}$ and $\{u_i^*\}$ are mutually dual bases of $\foam(\emptyset,\Gamma)$ and 
$\langle\ \vert_{KL}$ is an isomorphism. Therefore, two elements 
$f,g\in\foam(\emptyset,\Gamma)$ are equal if and only if 
$$\langle f\widehat{u_i}\rangle_{KL}=
\langle g\widehat{u_i}\rangle_{KL},$$ 
for all $i=1,2,\ldots,m$ (alternatively one can use the $u_i^*$ 
of course). In practice this only helps if the l.h.s. and the r.h.s. of 
these $m$ equations can be computed, e.g. if $f,g$ and the $u_i$ are all foams with singular graphs without vertices. 
Fortunately that is the case for all the relations in this proposition.     
 
Let us now prove (DR$_1$) in detail. Note that the boundary of any of the 
foams in (DR$_1$), 
denoted $\Gamma$, is homeomorphic to the web 
$$\figins{-17}{0.5}{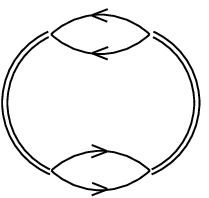}.$$ 
Recall that the dimension of 
$H^*(\Gamma)$ is equal to $2N(N-1)$ (see~\cite{KR}). For $0\leq i,j\leq 1$ 
and $0\leq m\leq k\leq N-2$, let $u_{i,j;(k,m)}$ 
denote the following foam 
$$\figins{-23}{0.8}{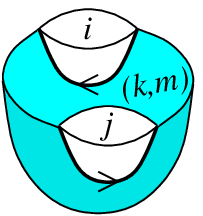}.$$
Let 
$$u_{i,j;(k,m)}^*=u_{1-j,1-i;(N-2-m,N-2-k)}.$$ 
From Equation (\ref{eq:bubble-ij}) and the sphere relation ($S_2$) 
it easily follows that $\langle u_{i,j;(k,m)}\widehat{u_{r,s;(t,v)}^*}
\rangle_{KL}=
\delta_{i,r}\delta_{j,s}\delta_{k,t}\delta_{m,v}$, where $\delta$ denotes the 
Kronecker delta. Note that there are exactly $2N(N-1)$ quadruples 
$(i,j;(k,l))$ 
which satisfy the conditions. Therefore the $u_{i,j;(k,m)}$ define a basis of 
$H^*(\Gamma)$ and 
the $u_{i,j;(k,m)}^*$ define its dual basis. In order to prove (DR$_1$) all we 
need to do next is check that 
$$\langle(\mbox{l.h.s. of (DR$_1$)}) \widehat{u_{i,j;(k,m)}}
\rangle_{KL}=\langle(\mbox{r.h.s. of (DR$_1$)}) \widehat{u_{i,j;(k,m)}}
\rangle_{KL},$$ for all $i,j$ and $(k,m)$. This again follows easily from 
equation (\ref{eq:bubble-ij}) and the sphere relation (S$_2$).     

The other digon removal relations are proved in the same way. We do not repeat the  
whole argument again for each digon removal relation, 
but will only give the relevant mutually dual bases. 
For (DR$_2$), note that $\Gamma$ is equal to 
$$\figins{-17}{0.5}{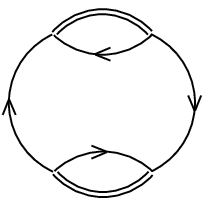}.$$ 
Let $u_{i,k,m}$ denote the foam
$$\figins{-23}{0.8}{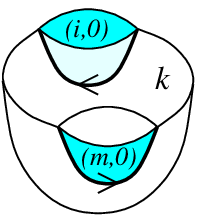}$$
for 
$0\leq i,m\leq N-2$ and $0\leq k\leq N-1$. The dual basis is defined by 
$$u_{i,k,m}^*=-u_{N-2-m,N-1-k,N-2-i},$$ 
for the same range of indices. Note that 
there are $N(N-1)^2$ possible indices, which corresponds exactly to the dimension of 
$H^*(\Gamma)$.

For (DR$_{3_1}$), the web $\Gamma$ is equal to 
$$\figins{-17}{0.5}{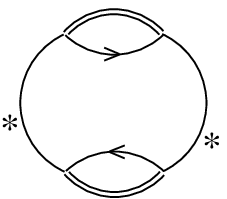}.$$
let $u_{i;(k,m);(p,q,r)}$ denote the foam 
$$\figins{-23}{0.8}{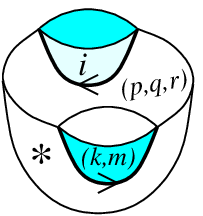}$$
for 
$0\leq i\leq 2$, $0\leq m\leq k\leq 1$ and $0\leq r\leq q\leq p\leq N-3$. The dual 
basis is given by 
$$u_{i;(k,m);(p,q,r)}^*=u_{2-(k+m);(1-\lfloor \frac{i}{2}\rfloor,
1-\lceil \frac{i}{2}\rceil);(N-3-r,N-3-q,N-3-p)},$$ 
for $0\leq t\leq s\leq 1$, $0\leq m\leq k\leq 1$ and $0\leq r\leq q\leq p\leq N-3$. 
Note that there are $3^2\binom{N}{3}$ possible indices, which corresponds exactly to 
the dimension of $H^*(\Gamma)$. 

For (DR$_{3_2}$), take $\Gamma$ to be 
$$\figins{-17}{0.5}{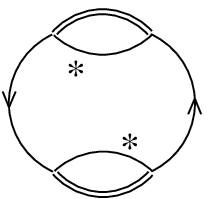}.$$
Let 
$u_{i;(k,m);(s,t)}$ denote the foam 
$$\figins{-23}{0.8}{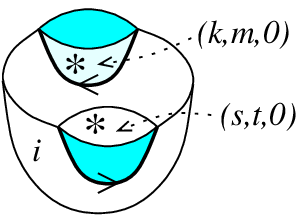}$$
for 
$0\leq i\leq N-1$, $0\leq m\leq k\leq N-3$ and $0\leq t\leq s\leq N-3$. The 
dual basis is given by  
$$u_{i;(k,m);(s,t)}^*=u_{N-1-i;(N-3-t,N-3-s);(N-3-m,N-3-k)},$$
for the same range of indices. Note that there are $N\binom{N-1}{2}^2$ indices, which 
corresponds exactly to the dimension of $H^*(\Gamma)$.

For (DR$_{3_3}$), take $\Gamma$ to be 
$$\figins{-17}{0.5}{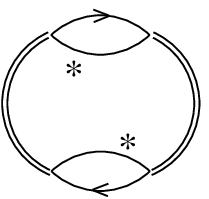}.$$ 
Let $u_{i,j;(k,m)}$ denote the foam 
$$\figins{-23}{0.8}{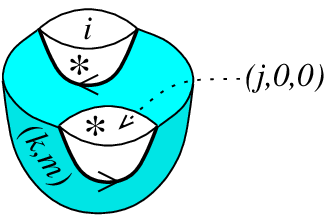}$$
for 
$0\leq i,j\leq N-3$ and $0\leq m\leq k\leq N-2$. Define 
$$u_{i,j;(k,m)}^*=u_{N-3-j,N-3-i;(N-2-m,N-2-k)},$$ 
for the same range of indices. Note that there are $(N-2)^2\binom{N}{2}$ indices, which 
corresponds exactly to the dimension of $H^*(\Gamma)$.

For (SqR$_1$), the relevant web $\Gamma$ is equal to 
$$\figins{-20}{0.6}{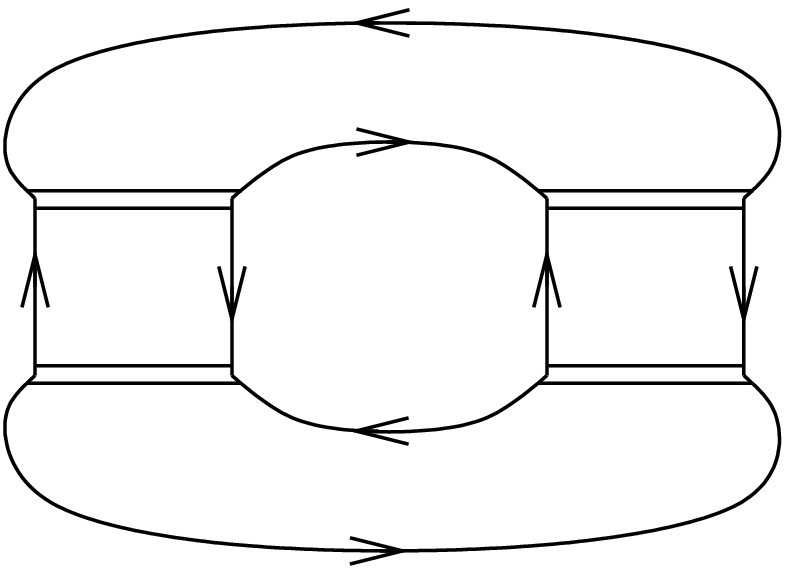}\ .$$ 
By the results in 
\cite{KR} we know that the dimension of $H^*(\Gamma)$ is equal to $N^2+2N(N-2)+N^2(N-2)^2$. 
The proof of this relation is very similar, except that it is slightly harder to 
find the mutually dual bases in $H^*(\Gamma)$. The problem is that the 
two terms on the right-hand side of (SqR$_1$) are topologically distinct. Therefore 
we get four different types of basis elements, which are glueings of the upper or lower 
half of the first term and the upper or lower half of the second term. For 
$0\leq i,j\leq N-1$, let $u_{i,j}$ 
denote the foam 
$$\figins{0}{0.8}{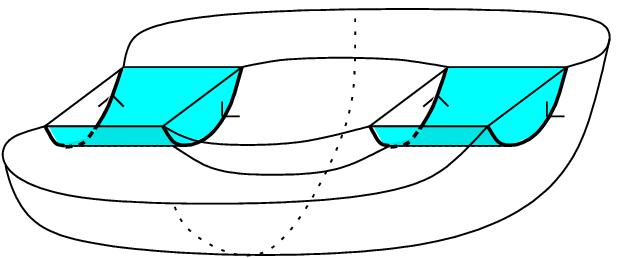}$$
with the top simple facet labelled by $i$ and the bottom one by $j$.
Take 
$$u_{i,j}^*=u_{N-1-j,N-1-i}.$$ 
Note that 
$$\langle u_{i,j}\widehat{u_{k,m}^*}\rangle_{KL}=
\delta_{i,k}\delta_{j,m}$$ by the (FC) relation in 
Corollary~\ref{cor:tubes} and the Sphere 
Relation (S$_1$). 

For $0\leq i\leq N-1$ and $0\leq k\leq N-3$, let $v'_{i,k}$ denote the foam 
$$\figins{0}{0.8}{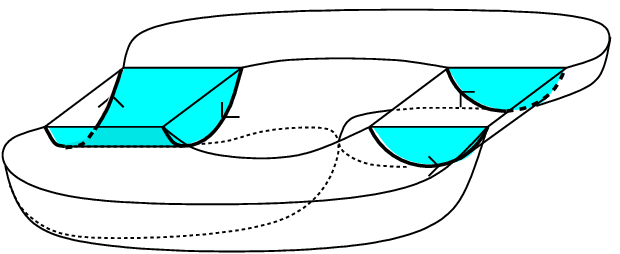}$$
with the simple square on the r.h.s. labelled by $i$ and the other simple facet by $k$. Note 
that the latter is only one facet indeed, because it has a saddle-point in the middle where the dotted lines meet.  
For the same range of indices, we define $w'_{i,k}$ by 
$$\figins{0}{0.8}{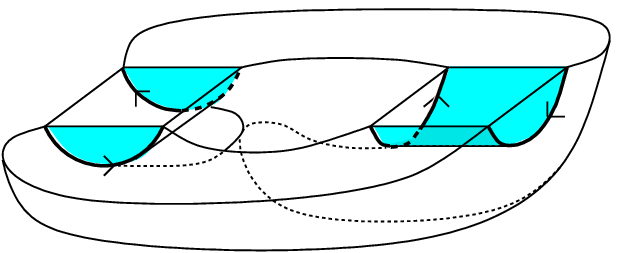}$$
with the simple square on the l.h.s. labelled by $i$ and the other simple facet by $k$.
The basis elements are now defined by 
$$v_{i,k}=\sum_{a+b+c=N-3-k}v'_{c,a+b+i}$$
and   
$$w_{i,j}=\sum_{a+b+c=N-3-j}w'_{c,a+b+i}.$$
The respective duals are defined by  
$$v_{i,k}^*=w'_{k,N-1-i}\quad\mbox{and}\quad w_{i,k}^*=v'_{k,N-1-i}.$$ 

We show that 
$$\langle v_{i,j}\widehat{v_{k,m}^*}\rangle_{KL}=\delta_{i,k}\delta_{j,m}=
\langle v_{i,j}\widehat{v_{k,m}^*}\rangle_{KL}$$
holds. First apply the 
(FC) relation of Corollary~\ref{cor:tubes}. 
Then apply (RD$_1$) of the same corollary twice and finally use the sphere relation (S$_1$).  

For $0\leq i,j\leq N-1$ and $0\leq k,m\leq N-3$, let $s'_{i,j,k,m}$ denote the foam 
$$\figins{0}{0.8}{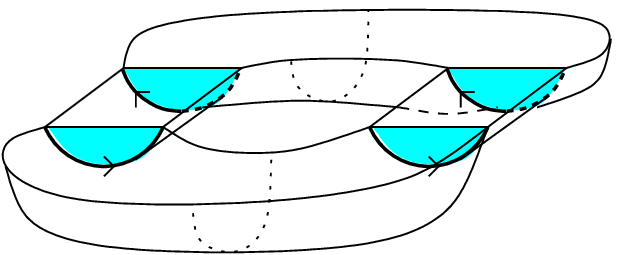}$$
with the simple squares labelled by $k$ and $m$, from left to right respectively, and 
the other two simple facets by $i$ and $j$, from front to back respectively. 
The basis elements are defined by 
$$s_{i,j,k,m}=\sum_{{a+b+c=N-3-k}\atop {d+e+f=N-3-m}}
s'_{c,f,i+a+d,j+b+e}.$$ For the same range of indices, the dual elements of this shape 
are given by 
$$s_{i,j,k,m}^*=s'_{m,k,N-1-i,N-1-j}.$$

From (RD$_1$) of Corollary~\ref{cor:tubes}, applied twice, and the sphere relation (S$_1$) 
it follows that  
$$\langle s_{a,b,c,d}\widehat{s_{i,j,k,m}^*}\rangle_{KL}=\delta_{a,i}\delta_{b,j}
\delta_{c,k}\delta_{d,m}$$ holds. 

It is also easy to see that the inner product $\langle\;\rangle_{KL}$ of a basis element and a dual 
basis element of distinct shapes, i.e. indicated by different letters above, gives zero. 
For example, consider 
$$\langle u_{i,j} \widehat{v_{k,m}^*}\rangle_{KL}$$ for any valid choice 
of $i,j,k,m$. At the place where the two different shapes are glued, 
$$u_{i,j}\widehat{v_{k,m}^*}$$ 
contains a simple saddle with a simple-double bubble. 
By Equation (\ref{eq:bubble-ikj}) that bubble kills 
$$\langle u_{i,j} \widehat{v_{k,m}^*}\rangle_{KL},$$ 
because $m\leq N-3$. The same argument holds for 
the other cases. This shows that $\{u,v,w,s\}$ and $\{u^*,v^*,w^*,s^*\}$ 
form dual bases of $H^*(\Gamma)$, because the number of possible indices 
equals $N^2+2N(N-2)+N^2(N-2)^2$.

In order to prove (SqR$_1$) one now has to compute the inner product of the 
l.h.s. and the r.h.s. with any basis element of $H^*(\Gamma)$ and show that they are equal. We 
leave this to the reader, since the arguments one has to use are the same as we used above. 
\end{proof}

\begin{cor}{(The \emph{second square removal} relation)}
$$
\figins{-34}{0.8}{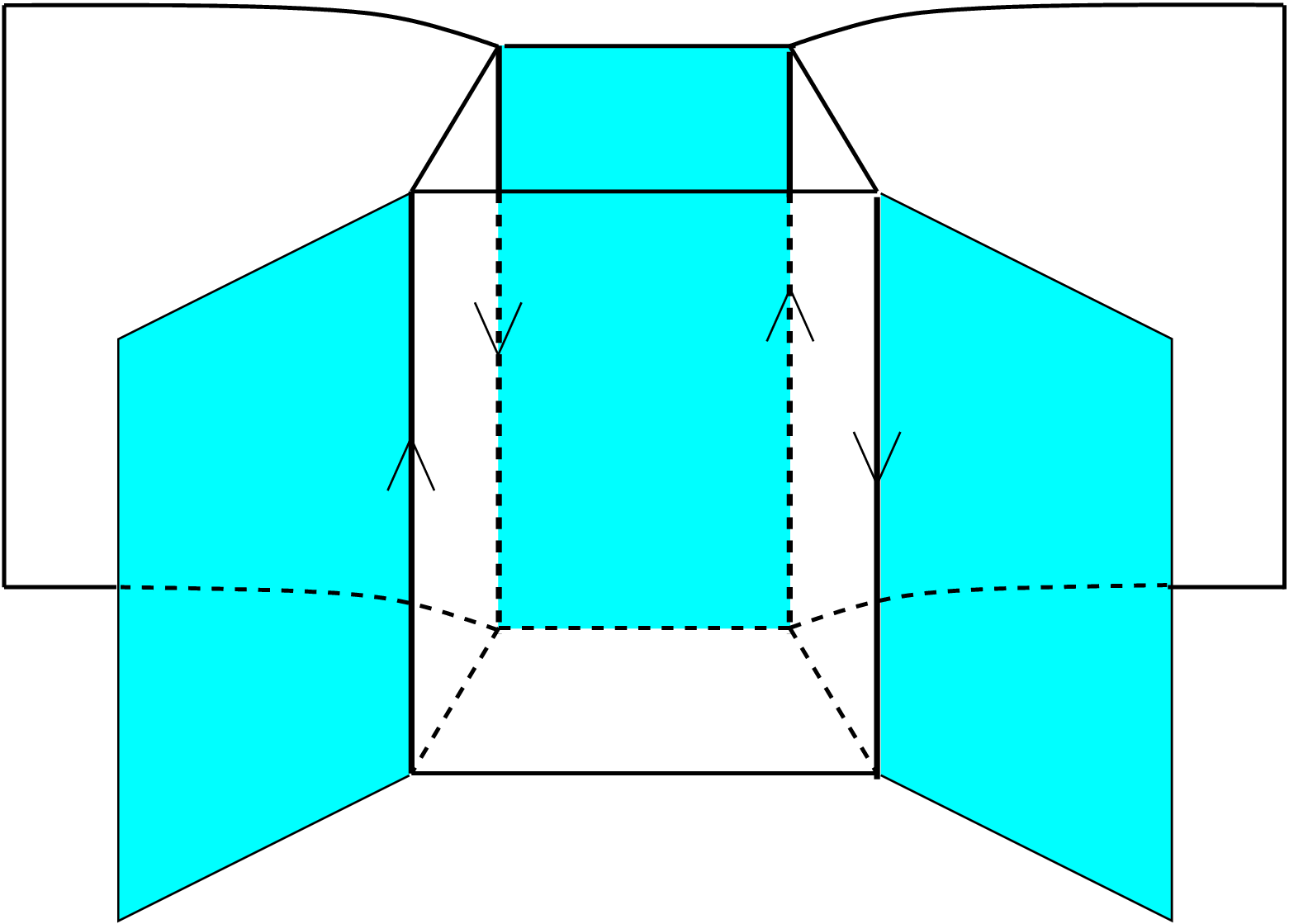}=
-\;\figins{-34}{0.8}{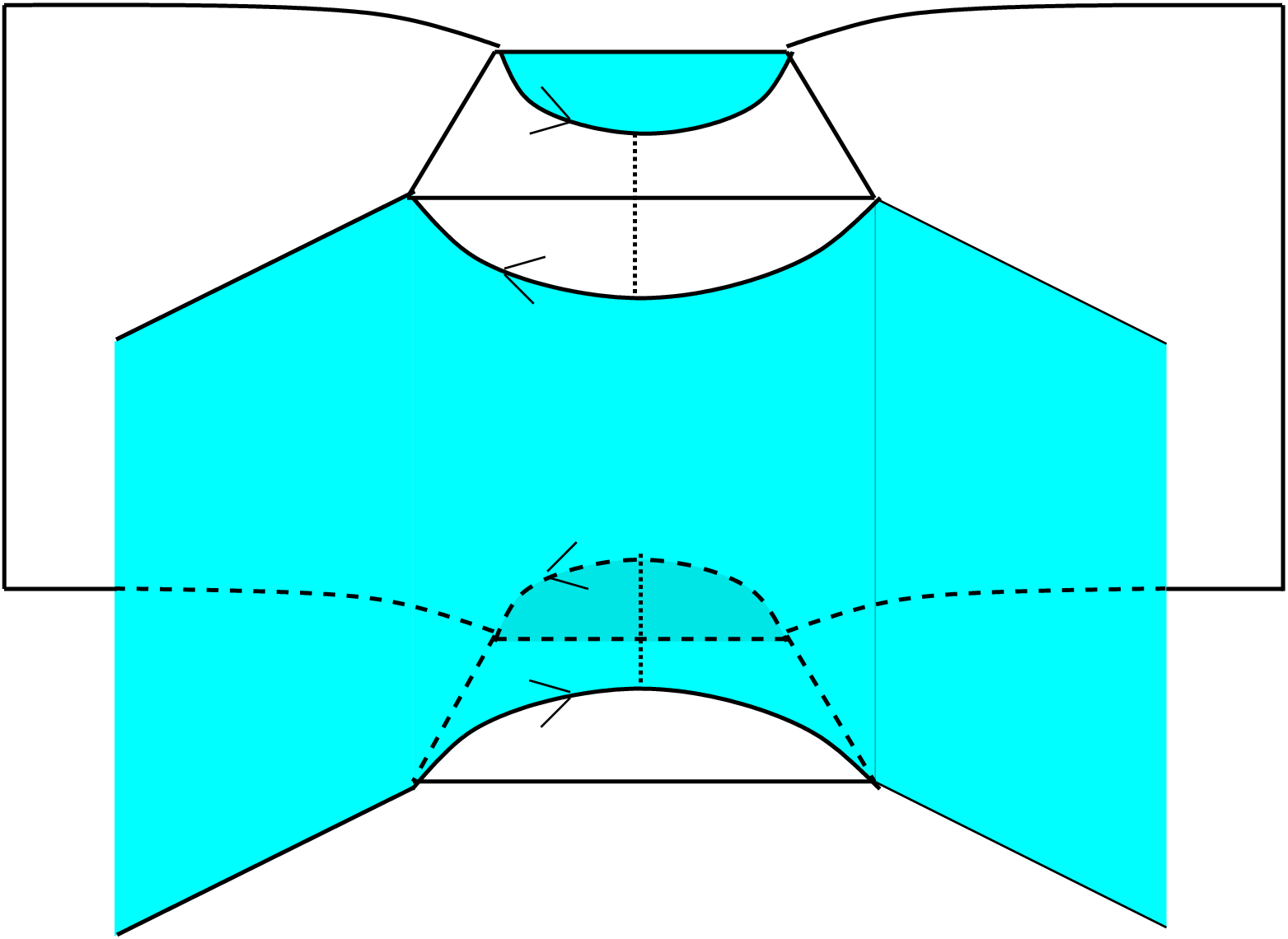}-
\figins{-34}{0.8}{sq_rem_b2}
\rlap{\hspace{6ex}\text{(SqR$_2$)}}
$$
\end{cor}
\begin{proof} 
Apply the Relation (SqR$_1$) to the simple-double square tube perpendicular to 
the triple facet of the second term on the r.h.s. of (SqR$_2$). The first term 
on the r.h.s. of (SqR$_1$) yields minus the first term on the r.h.s. of 
(SqR$_2$) after applying 
the relations (DR$_{3_2}$), (MP) and the Bubble Relation (\ref{eq:bubble-pqri}).  The second term 
on the r.h.s. of (SqR$_1$), i.e. the whole sum, yields the l.h.s. of (SqR$_2$) after applying 
the relations (DR$_{3_3}$), (MP) and the Bubble Relation (\ref{eq:bubble-pqri}). Note that 
the signs come out right because in both cases we get two bubbles with opposite orientations.  
\end{proof}

\section{Invariance under the Reidemeister moves}
\label{sec:invariance}

Let $\mathbf{Kom}(\foam)$ and $\mathbf{Kom}_{/h}(\foam)$ denote the category of complexes in $\foam$ 
and the same category modulo homotopies respectively.
As in~\cite{bar-natancob} and~\cite{mackaay-vaz} we can take all different flattenings of $D$ to obtain 
an object in $\mathbf{Kom}(\foam)$ which we call $\brak{D}$. The construction is well known by now 
and is indicated in Figure~\ref{fig:geom-cplx}.
\begin{figure}[h]
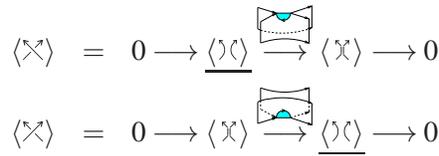

\begin{eqnarray*}
\brak{\undercrossing} &=& 0 \lra \underline{\brak{\orsmoothing}} 
     \stackrel{\figins{0}{0.2}{ssaddle}}{\longrightarrow} \brak{\DoubleEdge} \lra 0 \\
\brak{\overcrossing} &=& 0 \lra \brak{\DoubleEdge}   
     \stackrel{\figins{0}{0.2}{ssaddle_ud}}{\longrightarrow} \underline{\brak{\orsmoothing}} \lra 0  
\end{eqnarray*}
\caption{Complex associated to a crossing. Underlined terms correspond to homological degree zero}
\label{fig:geom-cplx}
\end{figure}


\begin{thm}\label{thm:inv}
The bracket $\brak{\;}$ is invariant in $\mathbf{Kom}_{/h}(\foam)$ under the Reidemeister moves.
\end{thm}

\begin{proof}
\n{\em Reidemeister I}: 
Consider diagrams $D$ and $D'$ that differ in a circular region, as in the 
figure below.
$$D=\figins{-13}{0.4}{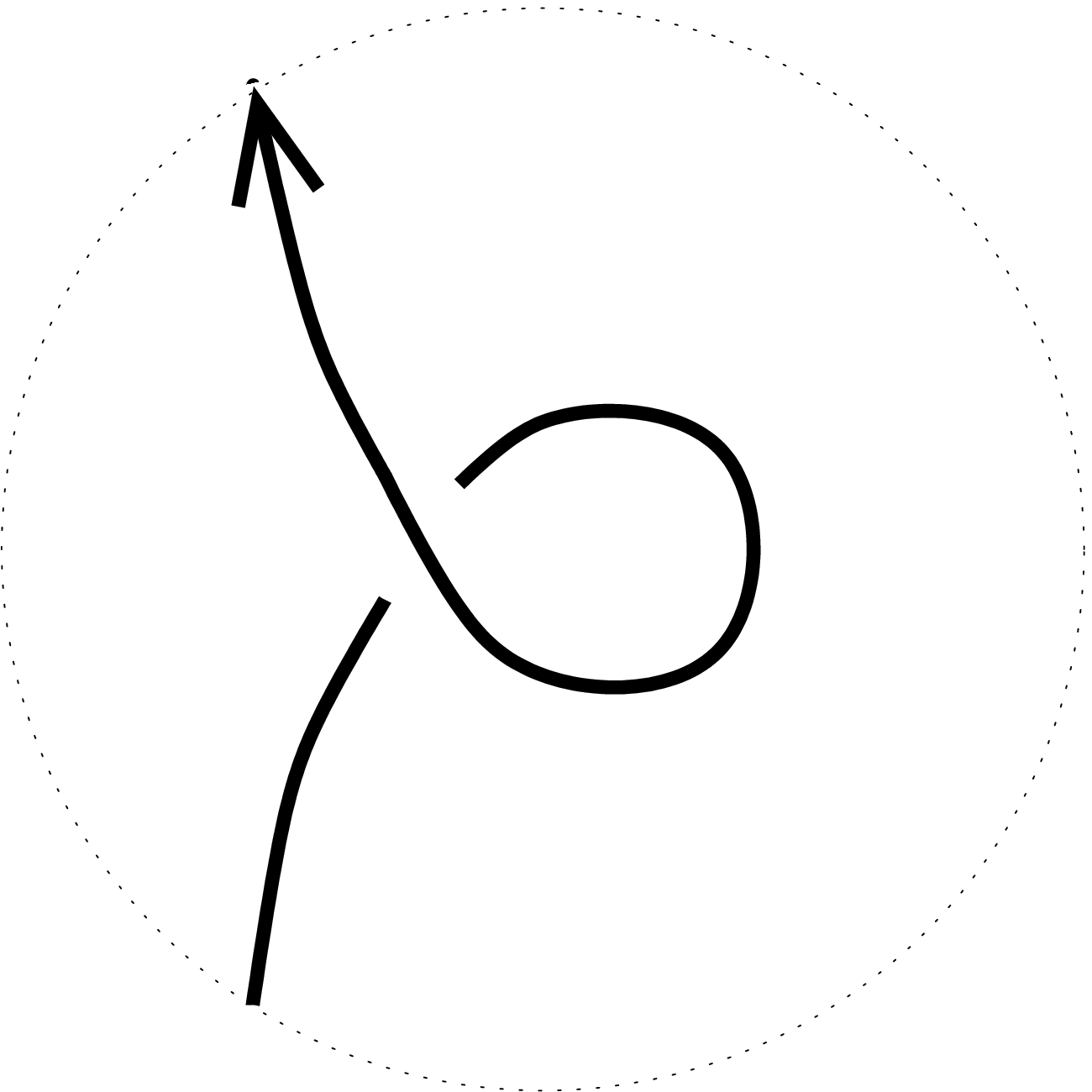}\qquad
D'=\figins{-13}{0.4}{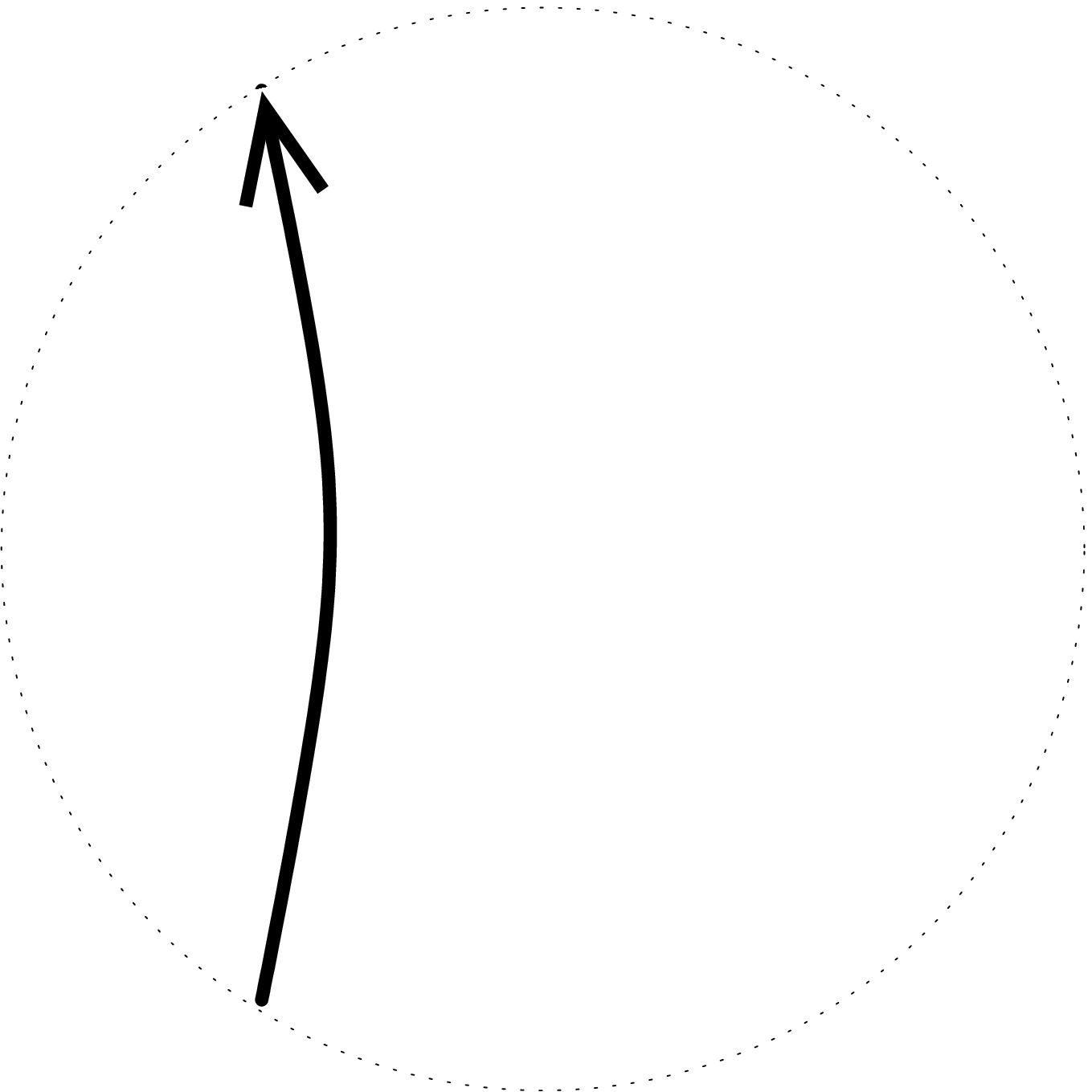}\,$$
We give the homotopy between complexes $\brak{D}$ and $\brak{D'}$ in 
Figure~\ref{fig:RI_Invariance}~\footnote{We thank Christian 
Blanchet for spotting a mistake in a previous version of this diagram.}.
\begin{figure}[h!]
$\xymatrix@R=32mm{
  \brak{D}:
\\
  \brak{D'}:
}
\xymatrix@C=35mm@R=28mm{
 \figins{0}{0.3}{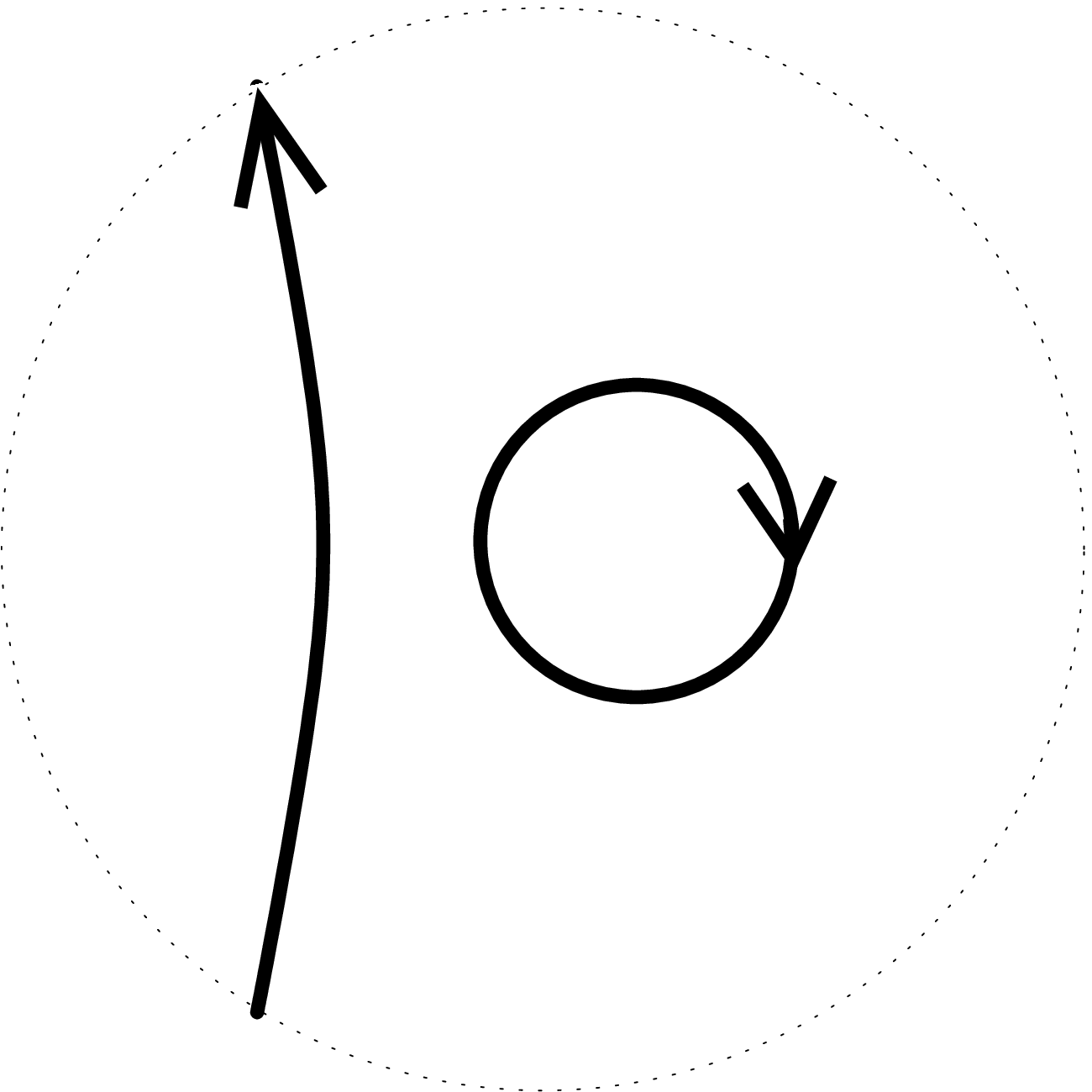}
  \ar@<4pt>[d]^{
        g^0 \;=\; \figins{-22}{0.6}{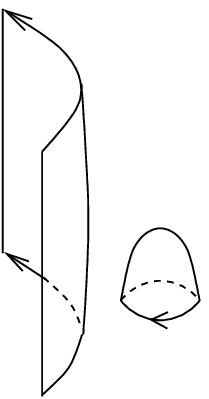} }
  \ar@<4pt>[r]^{
        d \;=\; \figins{-24}{0.6}{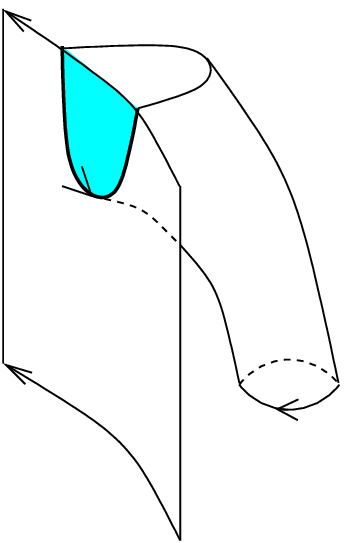} } &
  \figins{0}{0.3}{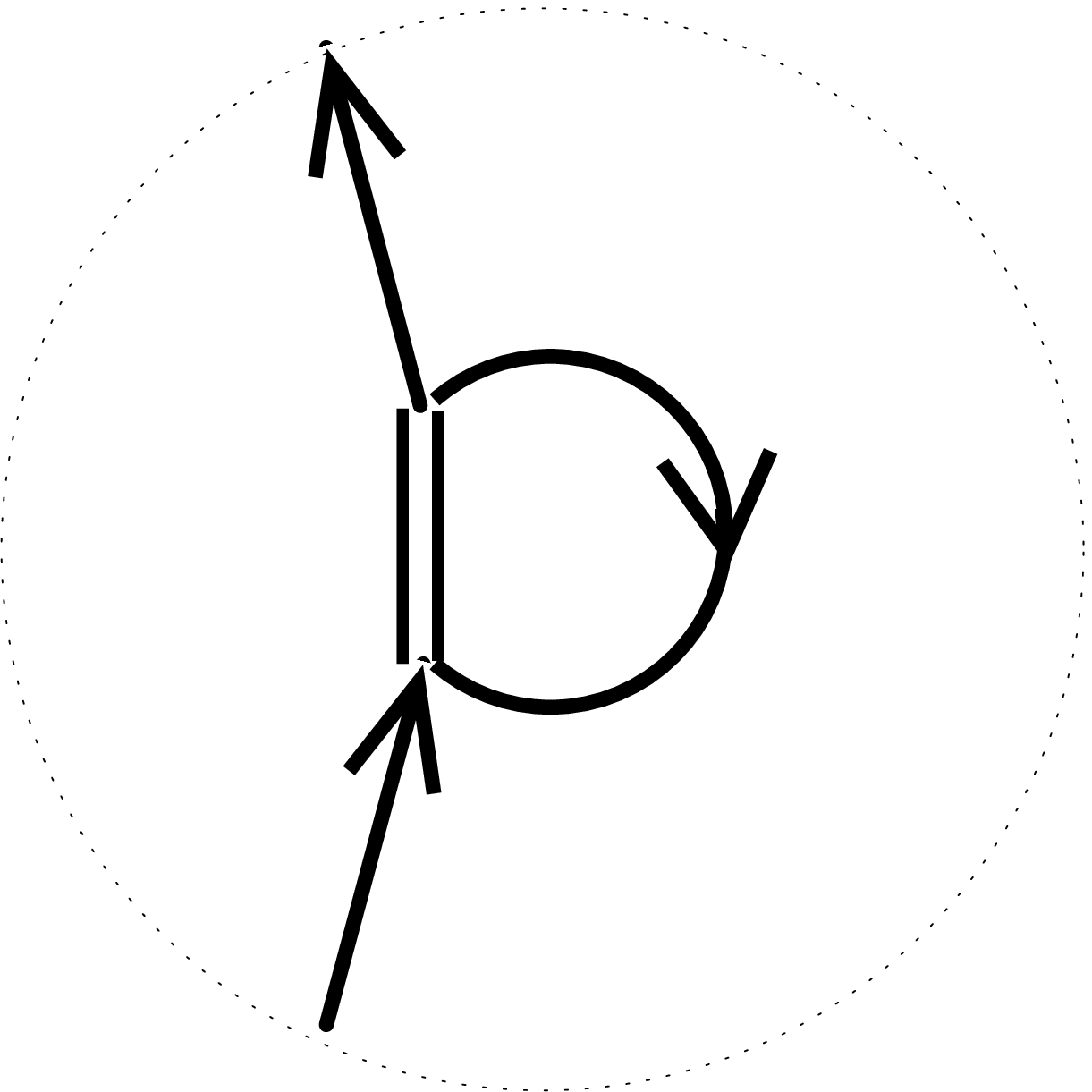} \ar@<2pt>[l]^{
     \quad h \;= \sum\limits_{a+b+c=N-2}\ \figins{-24}{0.6}{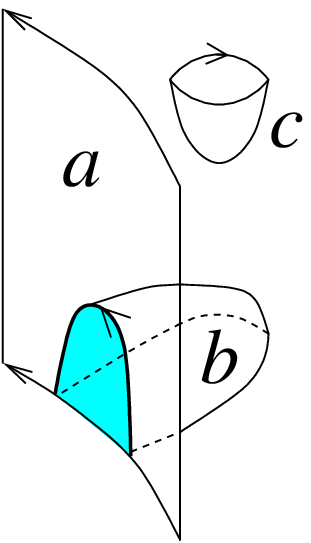} } \\
 \figins{0}{0.3}{ReidI-1}\ar[r]^0 
  \ar@<4pt>[u]^{
        f^0 \;= \sum\limits_{i=0}^{N-1}\ \figins{-24}{0.6}{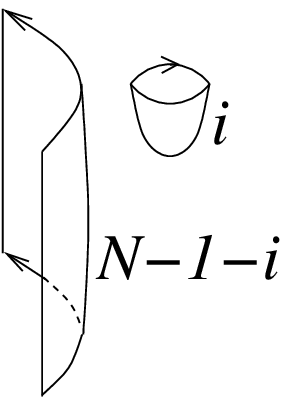} } &
  0 \ar@{<->}[u]_0 
}$
\caption{Invariance under $Reidemeister \ I$} 
\label{fig:RI_Invariance}
\end{figure} 
\noindent  By the Sphere Relation (S$_1$), we get $g^0f^0=Id_{\brak{D'}^0}$.
To see that $df^0=0$ holds, one can use dot mutation to get a new labelling of the same foam 
with the double facet labelled by $\pi_{N-1,0}$, which kills the foam by the dot conversion 
relations. The equality $dh=Id_{\brak{D}^1}$ follows from (\emph{DR$_2$}). To show that 
$f^0g^0+hd=Id_{\brak{D}^0}$, apply (RD$_1$) to $hd$ and then cancel all terms which appear 
twice with opposite signs. What is left is the sum of $N$ terms which is equal to 
$Id_{\brak{D}^0}$ by (CN$_1$). Therefore $\brak{D'}$  is homotopy-equivalent to $\brak{D}$.

\bigskip
\n{\em Reidemeister IIa}:
Consider diagrams $D$ and $D'$ that differ in a circular region, as
in the figure below.
$$
D=\figins{-13}{0.4}{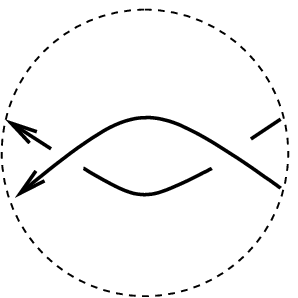}\qquad
D'=\figins{-13}{0.4}{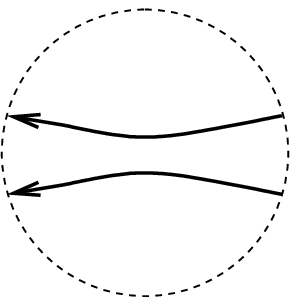}$$
We only sketch the arguments that the diagram in 
Figure~\ref{fig:RIIaInvariance} defines a homotopy 
equivalence between the complexes $\brak{D}$ and 
$\brak{D'}$: 
\begin{figure}[hb!]
\figwins{0}{4.9}{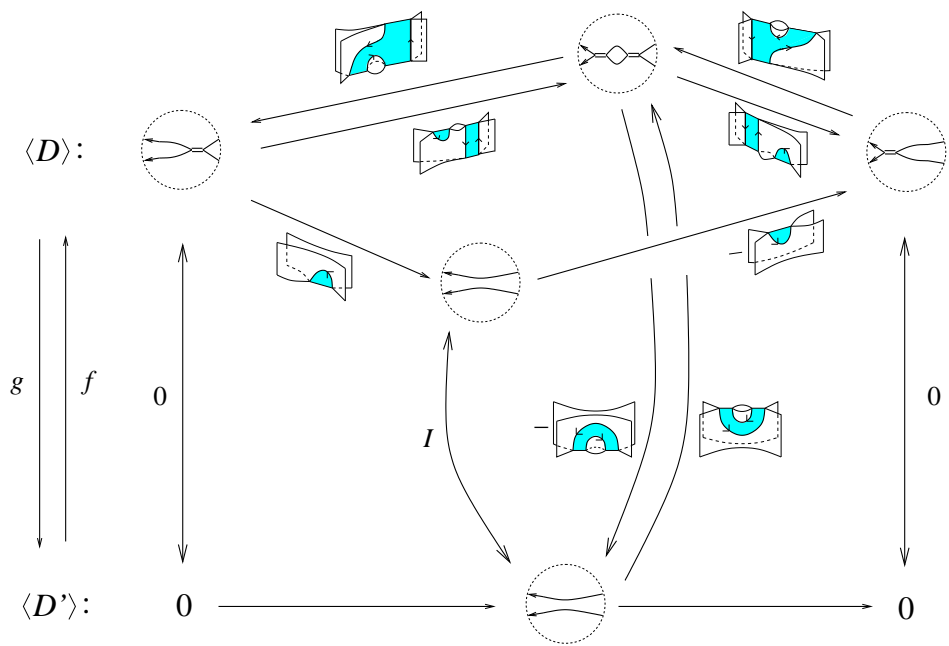}
\caption{Invariance under $Reidemeister \ IIa$} \label{fig:RIIaInvariance}
\end{figure}

\begin{itemize}
\item $g$ and $f$ are morphisms of complexes (use only isotopies); 
\item $g^1f^1=Id_{\brak{D'}^1}$  (uses equation (\ref{eq:bubble-ij})); 
\item $f^0g^0+hd=Id_{\brak{D}^0}$ and $f^2g^2+dh=Id_{\brak{D}^2}$  (use isotopies);
\item $f^1g^1+dh+hd=Id_{\brak{D}^1}$  (use (\emph{DR}$_1$)).
\end{itemize}

\bigskip
\n{\em Reidemeister IIb}:
Consider diagrams $D$ and $D'$ that differ only in a circular region, as 
in the figure below.
$$
D=\figins{-13}{0.4}{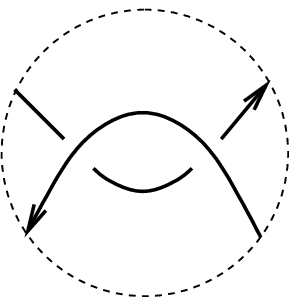}\qquad
D'=\figins{-13}{0.4}{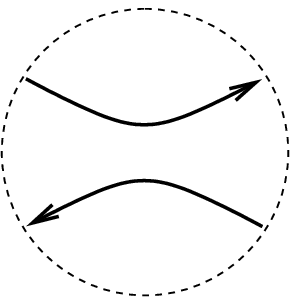}$$
Again, we sketch the arguments that the diagram in 
Figure~\ref{fig:RIIbInvariance} defines a 
homotopy equivalence between the complexes $\brak{D}$ and $\brak{D'}$:
\begin{figure}[h!]
\figwins{0}{5.4}{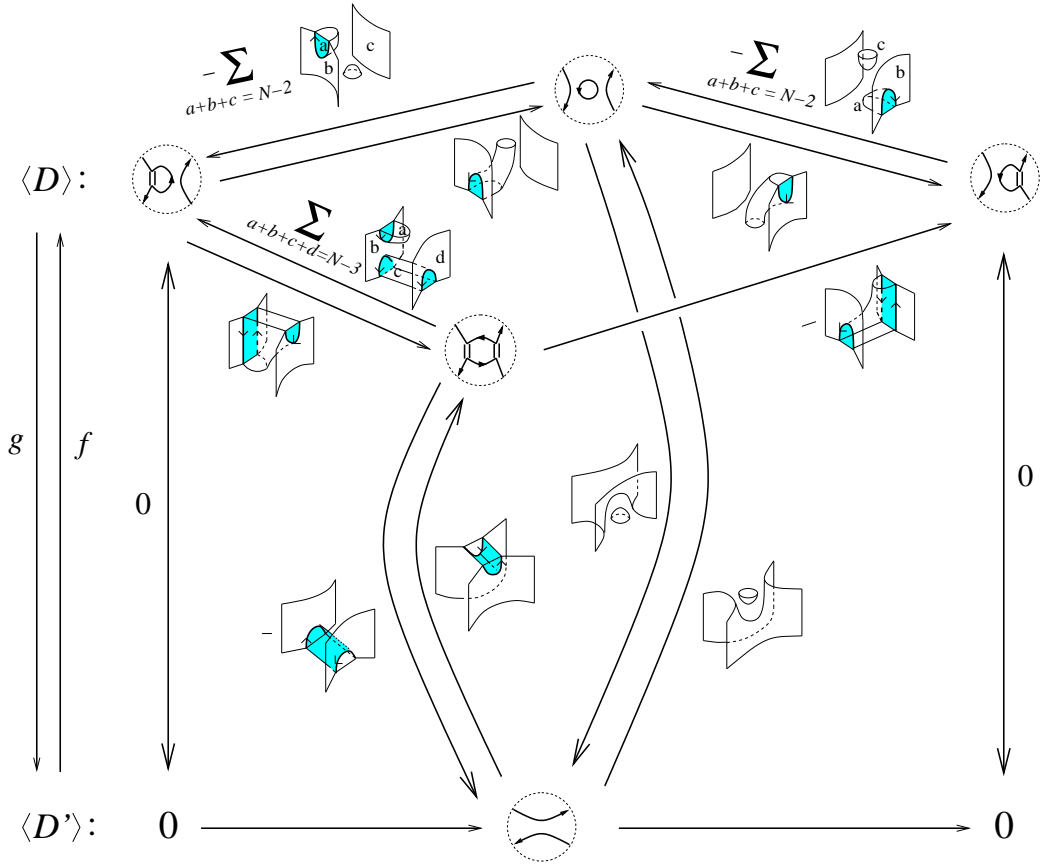}
\caption{Invariance under $Reidemeister \ IIb$} \label{fig:RIIbInvariance}
\end{figure}
\begin{itemize}
\item $g$ and $f$ are morphisms of complexes (use isotopies and DR$_2$); 
\item $g^1f^1=Id_{\brak{D'}^1}$  (use (\emph{FC}) and (\emph{S}$_1$)); 
\item $f^0g^0+hd=Id_{\brak{D}^0}$ and $f^2g^2+dh=Id_{\brak{D}^2}$ (use (\emph{RD}$_1$) and 
(\emph{DR}$_2$));
\item $f^1g^1+dh+hd=Id_{\brak{D}^1}$ (use (\emph{DR}$_2$), (\emph{RD}$_1$), 
(\emph{3C}) and (\emph{SqR}$_1$)).
\end{itemize}

\bigskip
\n{\em Reidemeister III}:
Consider diagrams $D$ and $D'$ that differ only in a circular region, as
in the figure below.
$$
D =\figins{-13}{0.4}{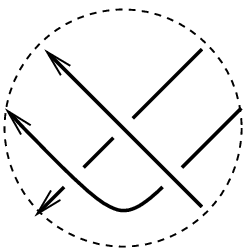}\qquad
D'=\figins{-13}{0.4}{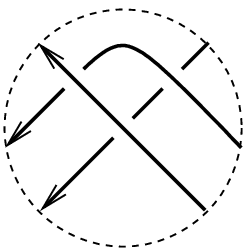}$$
\begin{figure}[ht!]
\figwins{0}{4.7}{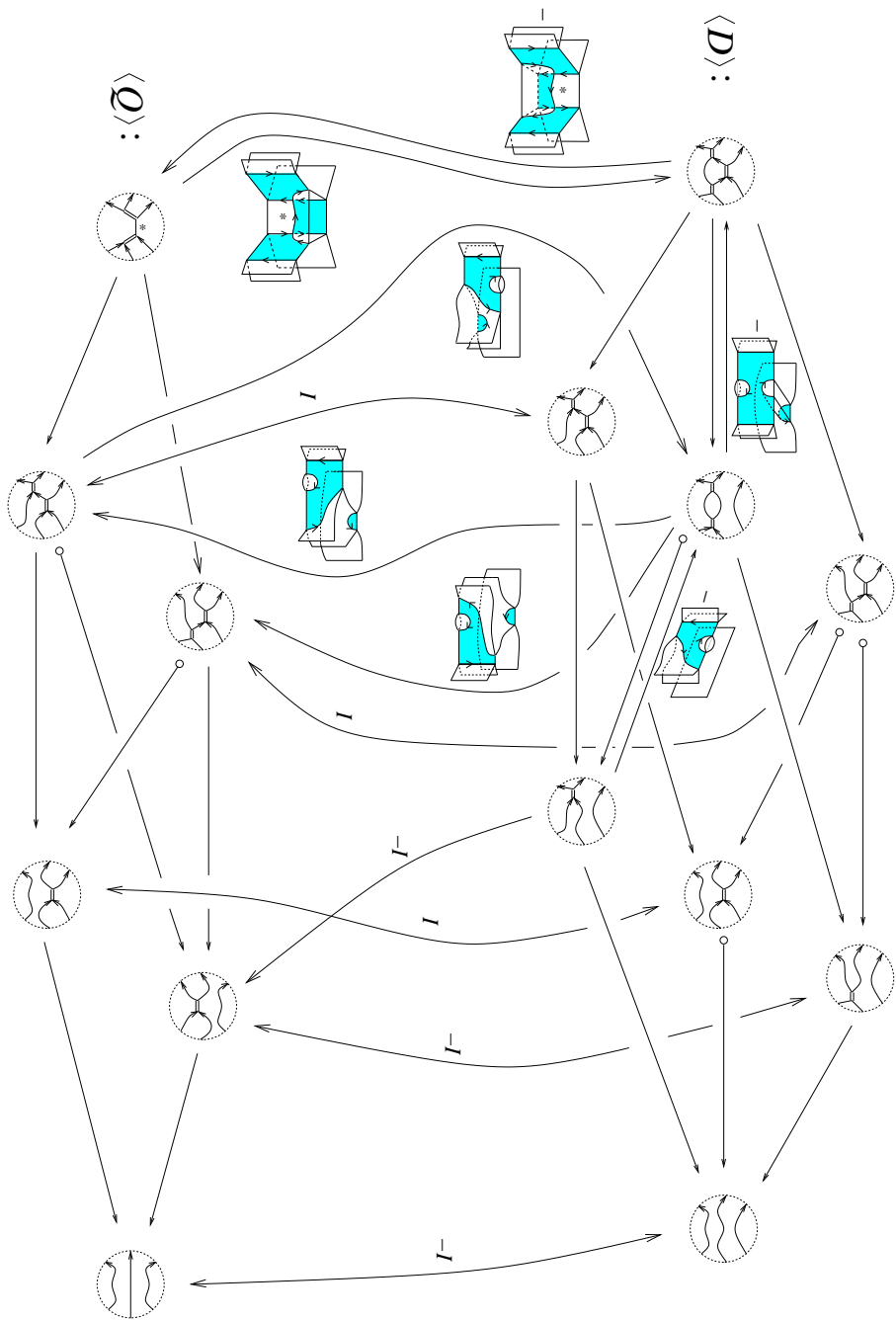} 
\caption{Invariance under $Reidemeister \ III$. A circle attached to the tail of an arrow indicates that the corresponding morphism has a minus sign.}
\label{fig:RIIIInvariance}
\end{figure}
In order
to prove that $\brak{D'}$ is homotopy equivalent to $\brak{D}$ 
we show that the latter is homotopy equivalent to a third complex denoted 
$\brak{Q}$ in Figure~\ref{fig:RIIIInvariance}. The 
differential in $\brak{Q}$ in homological 
degree $0$ is defined by 
$$\figins{-10}{0.6}{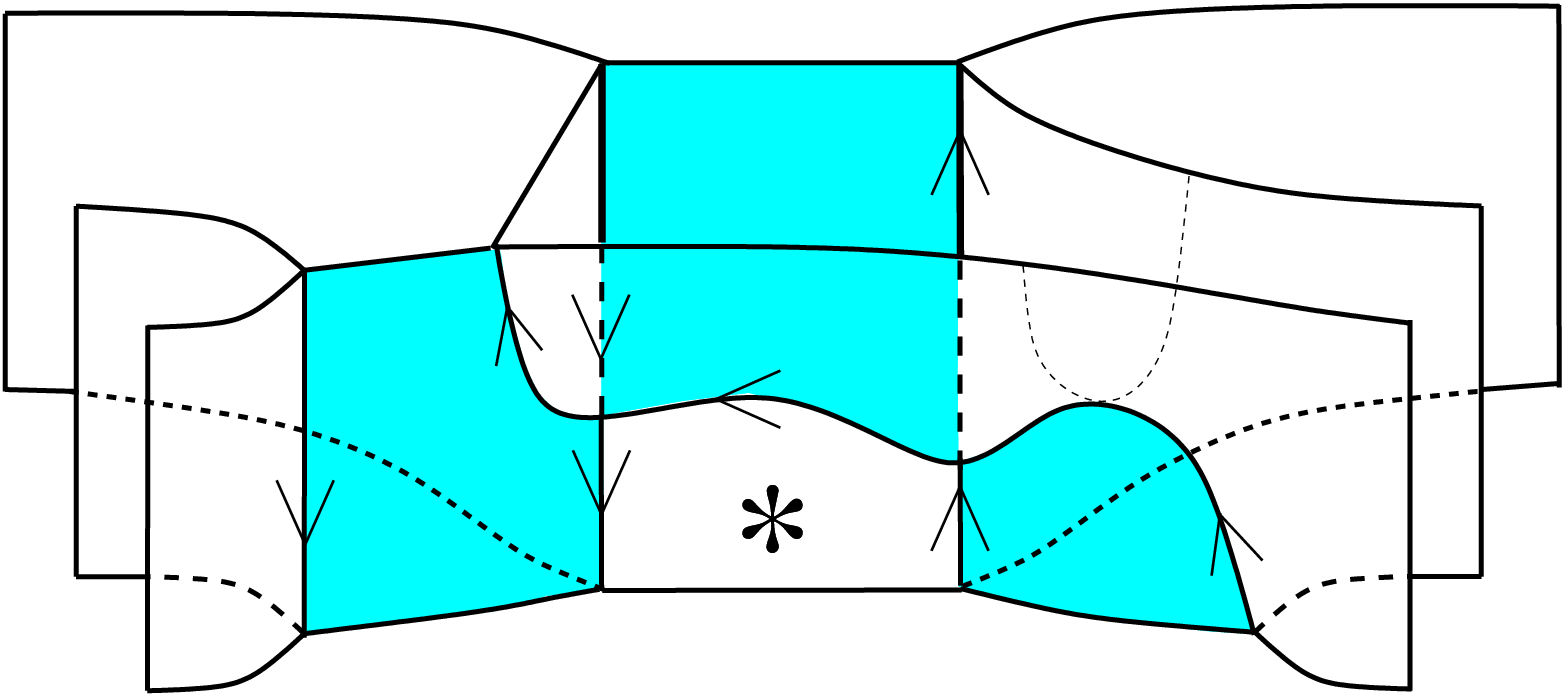}$$
for one summand and a similar foam for the other summand.   
By applying a symmetry relative to a horizontal axis crossing each diagram 
in $\brak{D}$ and $\brak{Q}$ we obtain a homotopy equivalence between 
$\brak{D'}$ and $\brak{Q'}$. It is easy to see that $\brak{Q}$ and 
$\brak{Q'}$ are isomorphic. In homological degree $0$ the isomorphism 
is given by the obvious foam with two singular vertices. In the other 
degrees the isomorphism is given by the identity (in degrees $1$ and $2$ 
one has to swap the two terms of course). The fact that this defines an 
isomorphism follows immediately from the (MP) relation. We conclude 
that $\brak{D}$ and $\brak{D'}$ are homotopy equivalent.    
\end{proof}

From Theorem~\ref{thm:inv} we see that we can use any diagram $D$ of $L$ to obtain the 
invariant in $\mathbf{Kom}_{/h}(\foam)$ which justifies the notation $\brak{L}$ for $\brak{D}$.

\section{Functoriality}
\label{sec:functoriality}
The proof of the functoriality of our homology follows the same line of 
reasoning as in \cite{bar-natancob} and \cite{mackaay-vaz}.
As in those papers, it is clear that the construction and the results 
of the previous sections 
can be 
extended to the category of tangles, following a similar approach using open webs 
and foams with corners. A foam with corners should be considered as living inside a cylinder, as 
in \cite{bar-natancob}, such that the intersection with the cylinder is just a disjoint set of 
vertical edges.  

The degree formula can be extended to the category of open webs and foams with corners by
\begin{defn}
Let $u$ be a foam with $d_\bdot$ dots of type $\bdot$, $d_\wdot$ dots of type 
$\wdot$ and $d_\bwdot$ dots of type $\bwdot$. Let $b_i$ be the number of vertical edges 
of type $i$ of the boundary of $u$. The $q$-grading of $u$ is given 
by
\begin{equation}
q(u)= -\sum_{i=1}^3 i(N-i) q_i(u) - 2(N-2)q_{\sk}(u) + 
\dfrac{1}{2}\sum_{i=1}^3 i(N-i)b_i+ 2d_\bdot + 4d_\wdot +6d_\bwdot.
\end{equation}
\end{defn}

Note that the Kapustin-Li formula also induces a grading on foams with corners, because 
for any foam $u$ between two (open) webs $\Gamma_1$ and $\Gamma_2$, it gives an element 
in the graded vector space $\Ext(M_1,M_2)$, 
where $M_i$ is the matrix factorization associated to $\Gamma_i$ in \cite{KR}, for $i=1,2$. 
Recall that the $\Ext$ groups have a $\bZ/2\bZ\times\bZ$-grading. For foams there is 
no $\bZ/2\bZ$-grading, but the $\bZ$-grading survives.  

\begin{lem} For any foam $u$, the Kapustin-Li grading of $u$ is equal to $q(u)$.
\end{lem}
\begin{proof}
Both gradings are additive under horizontal and vertical glueing and 
are preserved by the 
relations in $\foam$. Also the degrees of the dots are the same in both 
gradings. 
Therefore it is enough to establish the equality between the gradings 
for the foams which generate $\foam$. For any foam without a singular graph the gradings are 
obviously equal, so let us concentrate on the singular cups and caps, the singular saddle point 
cobordisms and the cobordisms with one singular vertex in 
Figure~\ref{fig:elemfoams}. To compute the degree of the singular cups and caps, for both gradings, one can use the digon removal 
relations. For example, let us consider the singular cup 
$$\figins{-10}{0.4}{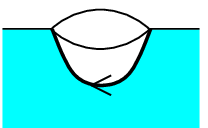}.$$ 
Any grading that preserves relation (DR$_1$) has to attribute the value of 
$-1$ to that foam, because the foam on the l.h.s. of (DR$_1$) has degree $0$, being an identity, 
and the dot on the r.h.s. has degree $2$. Similarly one can compute the degrees of the other 
singular cups and caps. To compute the degree of the singular saddle-point cobordisms, one can 
use the removing disc relations (RD$_1$) and (RD$_2$). For example, the saddle-point cobordism in 
Figure~\ref{fig:elemfoams} 
has to have degree $1$. Finally, using the (MP) relation 
one shows 
that both foams on the r.h.s. in Figure~\ref{fig:elemfoams} 
have to have degree $0$.          
\end{proof}

\begin{cor}
For any closed foam $u$ we have that $\langle u\rangle_{KL}$ is zero if $q(u)\neq 0$.
\end{cor}

As in \cite{mackaay-vaz} we have the following lemma, which 
is the analogue of Lemma 8.6 in \cite{bar-natancob}:

\begin{lem}
For a crossingless tangle diagram $T$ we have that
$\mbox{Hom}_{\foam}(T,T)$ is zero in negative degrees and $\bQ$ in degree zero.
\end{lem}

\begin{proof}
Let $T$ be a crossingless tangle diagram and $u\in\mbox{Hom}_{\foam}(T,T)$.   
Recall that $u$ can be considered to be in a cylinder with vertical edges 
intersecting the latter. The boundary of $u$ consists of a disjoint union 
of circles (topologically speaking). By dragging these 
circles slightly into the interior of $u$ one gets a disjoint union of circles in 
the interior of $u$. Apply relation (CN$_1$) to each of these circles. We 
get a linear combination of terms in $\mbox{Hom}_{\foam}(T,T)$ each of which 
is the disjoint union of the identity on $T$, possibly decorated with dots, 
and a closed foam, which can be evaluated by $\langle\;\rangle_{KL}$. Note 
that the identity of $T$ with any number of dots has always 
non-negative degree. Therefore, if $u$ has negative degree, the closed 
foams above have negative degree as well and evaluate to zero. This shows 
the first claim in the lemma. If $u$ has degree $0$, the only terms 
which survive after evaluating the closed foams have degree $0$ as well and 
are therefore a multiple of the identity on $T$. This proves the second 
claim in the lemma. 
\end{proof}

The proofs of Lemmas~8.7-8.9 in~\cite{bar-natancob} are ``identical''. 
The proofs of Theorem 4 and Theorem 5 follow the same reasoning but 
have to be adapted as in \cite{mackaay-vaz}. One has to use the 
homotopies of our Section~\ref{sec:invariance} instead of the homotopies 
used in \cite{bar-natancob}. 
Without giving further details, we state the 
main result. Let $\mathbf{Kom}_{/\bQ^*h}(\foam)$ denote the category 
$\mathbf{Kom}_{/h}(\foam)$ 
modded out by $\bQ^*$, the invertible rational numbers. Then

\begin{prop}
\label{prop:func}
$\brak{\;}$ defines a functor $\mathbf{Link}\ra 
\mathbf{Kom}_{/\bQ^*h}(\foam )$.
\end{prop}

\section{The $\ \sln$-link homology}
\label{sec:taut-functor}

\begin{defn}
Let $\Gamma$, $\Gamma'$ be closed webs and $f\in\mbox{Hom}_{\foam }(\Gamma,\Gamma')$. Define a functor 
$\mathcal{F}$ between the categories $\foam$ and the category $\V$ 
of $\bZ$-graded rational vector spaces and ${\bZ}$-graded linear maps as
\begin{enumerate}
\item
$\cF(\Gamma)=\mbox{Hom}_{\foam}(\emptyset,\Gamma),$
\item $\cF(f)$ is the $\bQ$-linear map 
$\cF(f):\mbox{Hom}_{\foam }(\emptyset,\Gamma)\ra\mbox{Hom}_{\foam }(\emptyset,\Gamma')$ given by composition.
\end{enumerate} 
\end{defn}
\vspace{1ex}

Note that $\cF$ is a tensor functor and that the degree of $\cF(f)$ equals $q(f)$. Note also that 
$\cF(\unknot)\cong H^*\left(\cp{N-1}\right)\{-N+1\}$ and 
$\cF(\figins{-2.5}{0.16}{dble-circ})\cong H^*\left(\G_{2,N}\right)\{-2N+4\}$.

The following are a categorified version of the relations in Figure~\ref{fig:moy}.

\begin{lem}[MOY decomposition]\label{lem:moy-F}
We have the following decompositions under the functor $\cF$:
\begin{enumerate}
\item \label{eq:decomp-dr1}
$\cF\left(
\figins{-8}{0.3}{digon-up}\right)\cong \cF\left(
\figins{-8}{0.3}{dbedge-up}\right)\left\{-1\right\}\bigoplus \ \cF\left(
\figins{-8}{0.3}{dbedge-up}\right)\left\{1\right\}$.
\vspace{1em}
\item \label{eq:decomp-dr2}
$\cF\left(
\figins{-8}{0.3}{dbedge-dig}\right)\cong\bigoplus\limits_{i=0}^{N-2}\cF\left(
\figins{-8}{0.3}{edge-up}\right)\left\{2-N+2i\right\}$.
\vspace{1em}
\item \label{eq:decomp-sqr1}
$\cF\left(
\figins{-8}{0.3}{square}\right)\cong\cF\left(
\figins{-8}{0.3}{twoedges-lr}\right)\bigoplus\left(\bigoplus\limits_{i=0}^{N-3} \ \cF\left(
\figins{-8}{0.3}{twoedges-ud}\right)\left\{3-N+2i\right\}\right)$.
\vspace{1em}
\item \label{eq:decomp-sqr2}
$\cF\left(\figins{-20}{0.65}{moy5-1}\right) \bigoplus
\cF\left(\figins{-20}{0.65}{moy5-21}\right)\cong 
\cF\left(\figins{-20}{0.65}{moy5-2}\right)\bigoplus
\cF\left(\figins{-20}{0.65}{moy5-11}\right)
$.
\end{enumerate}
\end{lem}

\begin{proof}
\emph{(\ref{eq:decomp-dr1})}: Define grading-preserving maps 
$$\varphi_0:\cF\left(
\figins{-8}{0.3}{digon-up}\right)\left\{1\right\} 
\ra \cF\left(\figins{-8}{0.3}{dbedge-up}\right)
\qquad
\varphi_1:\cF\left(
\figins{-8}{0.3}{digon-up}\right)\left\{-1\right\}  
\ra \cF\left(\figins{-8}{0.3}{dbedge-up}\right)
$$

$$\psi_0:\cF\left(
\figins{-8}{0.3}{dbedge-up}\right) \ra
\cF\left(
\figins{-8}{0.3}{digon-up}\right)\left\{1\right\} 
\qquad
\psi_1: \cF\left(
\figins{-8}{0.3}{dbedge-up}\right) \ra
\cF\left(
\figins{-8}{0.3}{digon-up}\right)\left\{-1\right\} 
$$
as
$$
\varphi_0=  \cF\left(\figins{-10}{0.3}{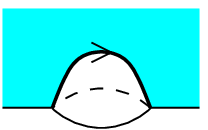}\right),\ \ \ 
\varphi_1=\cF\left(\figins{-10}{0.3}{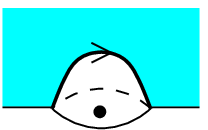} \right),\ \ \
\psi_0=\cF\left(\figins{-6}{0.3}{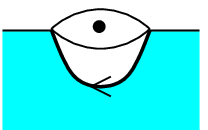}\right),\ \ \
\psi_1=-\cF\left( \figins{-6}{0.3}{scup}\right).
$$
The bubble identities imply that $\varphi_i\psi_j=\delta_{i,j}$ (for $i,j=0,1$) and from the (\emph{DR$_1$}) relation it follows that $\psi_0\varphi_0+\psi_1\varphi_1$ is the identity map in $\cF\left(\figins{-5}{0.2}{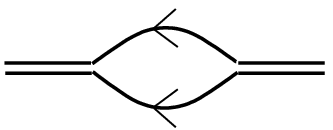}\right)$.

\n\emph{(\ref{eq:decomp-dr2})}: Define grading-preserving maps 
$$\varphi_i:\cF\left(
\figins{-8}{0.3}{dbedge-dig}\right)\left\{N-2-2i\right\}
\ra \cF\left(\figins{-8}{0.3}{edge-up}\right),
\qquad 
\psi_i:\cF\left(
\figins{-8}{0.3}{edge-up}\right)
\ra\cF\left(
\figins{-8}{0.3}{dbedge-dig}\right)\left\{N-2-2i\right\},
$$
for $0\leq i\leq N-2$, as
$$
\varphi_i=  \cF\left(\figins{-24}{0.65}{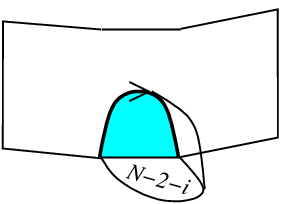}\right),\qquad
\psi_i=\sum\limits_{j=0}^{i}\cF\left(\figins{-22}{0.65}{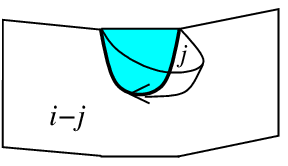}\right).
$$

We have $\varphi_i\psi_k=\delta_{i,k}$ and $\sum\limits_{i=0}^{N-2}\psi_i\varphi_i=Id\left(\cF\left(\figins{-4}{0.2}{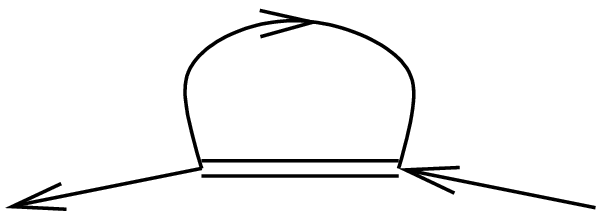} \right)\right)$. The first assertion is straightforward and can be checked using the (\emph{RD}) and (\emph{S$_1$}) relations and the second is immediate from the (\emph{DR$_2$}) relation, which can be written as  

$$
\figins{-32}{0.8}{digonfid_b} = \sum_{i=0}^{N-2}\sum_{j=0}^{i}\
\figins{-32}{0.8}{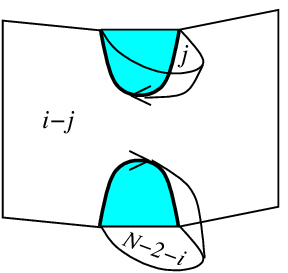}
\rlap{\hspace{15ex}\text{(DR$_2$)}}
$$

\bigskip

\n\emph{(\ref{eq:decomp-sqr1})}: Define grading-preserving maps 
$$\varphi_i:\cF\left(
\figins{-8}{0.3}{square}\right)\left\{N-3+2i\right\}
\ra \cF\left(\figins{-8}{0.3}{twoedges-ud}\right),
\quad 
\psi_i:\cF\left(
\figins{-8}{0.3}{twoedges-ud}\right)
\ra\cF\left(
\figins{-8}{0.3}{square}\right)\left\{N-3+2i\right\},
$$
for $0\leq i\leq N-3$, and
$$\rho:\cF\left(
\figins{-8}{0.3}{square}\right)
\ra \cF\left(\figins{-8}{0.3}{twoedges-lr}\right),
\qquad 
\tau:\cF\left(
\figins{-8}{0.3}{twoedges-lr}\right)
\ra\cF\left(
\figins{-8}{0.3}{square}\right),
$$
as
$$
\varphi_i=  \cF\left(\figins{-22}{0.7}{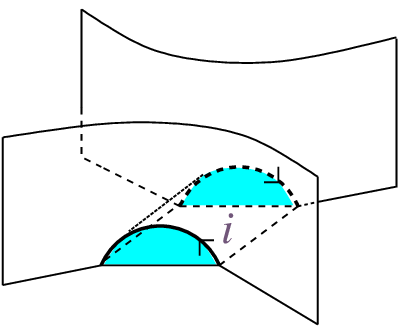}\right), 
\qquad
\psi_i=\sum\limits_{a+b+c=N-3-i}\cF\left(\figins{-22}{0.7}{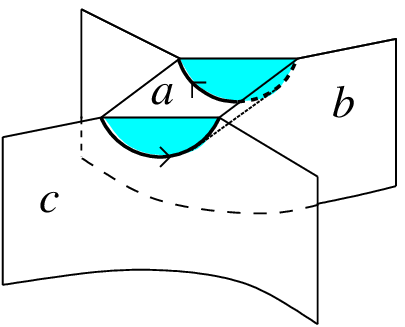}\right),
$$
$$
\rho = \cF\left(\figins{-22}{0.7}{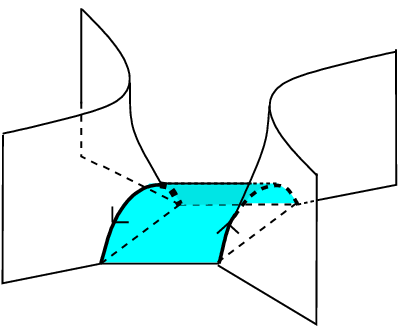}\right),
\qquad
\tau = -\cF\left(\figins{-22}{0.7}{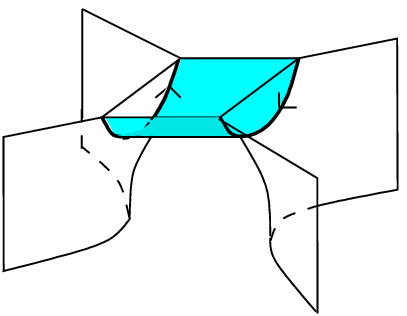}\right).
$$

Checking that $\varphi_i\psi_k=\delta_{i,k}$ for $0\leq i,k\leq N-3 $, 
$\varphi_i\tau=0$ and $\rho\psi_i=0$, for $0\leq i\leq N-3$, and 
$\rho\tau=-1$ is left to the reader. From the (\emph{SqR$_1$}) relation it 
follows that 
$\tau\rho+\sum\limits_{i=0}^{N-3}\psi_i\varphi_i=Id\left(\cF\left(\figins{-4.5}{0.2}{square} 
\right)\right)$.

\n\emph{Direct Sum Decomposition~(\ref{eq:decomp-sqr2})}: We prove direct decomposition~(\ref{eq:decomp-sqr2}) showing that

$$\xymatrix@R=1mm{
\cF\left(\figins{-20}{0.65}{moy5-1}\right)\cong
\cF\left(\figins{-20}{0.65}{moy5-11}\right)\bigoplus
\cF\left(\figins{-20}{0.65}{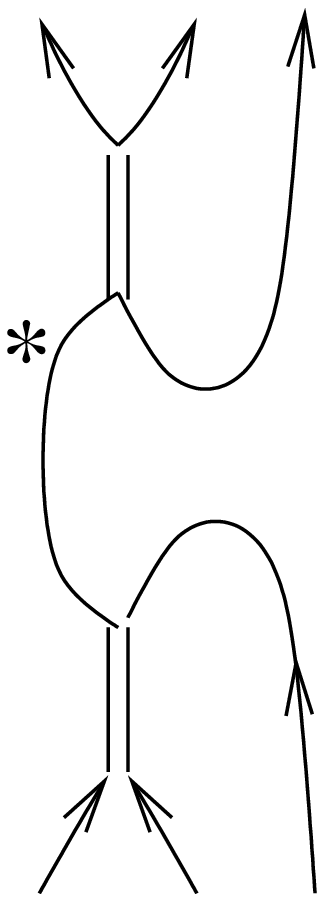}\right) &
\cF\left(\figins{-20}{0.65}{moy5-2}\right)\cong
\cF\left(\figins{-20}{0.65}{moy5-21}\right)\bigoplus
\cF\left(\figins{-20}{0.65}{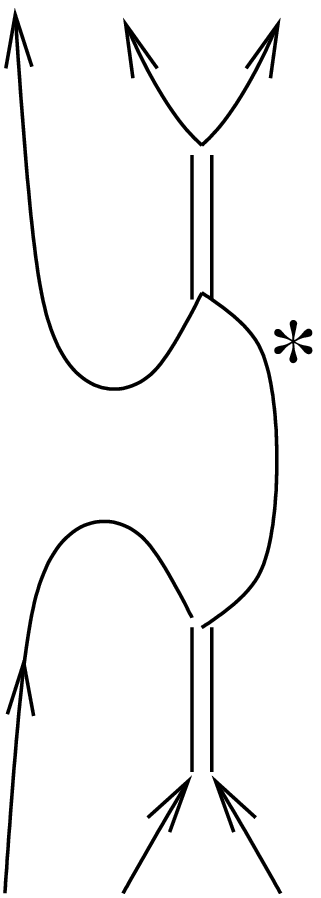}\right). \\
a) & b) 
}$$
Note that this suffices because the last term on the r.h.s. of a) is 
isomorphic to the last term on the r.h.s. of b) by the (MP) relation. 

To prove $a)$ we define grading-preserving maps 
$$\varphi_0:\cF\left(
\figins{-15}{0.5}{moy5-1}\right)\ra 
\cF\left(
\figins{-15}{0.5}{moy5-11}\right),
\qquad 
\varphi_1:\cF\left(
\figins{-15}{0.5}{moy5-1}\right)
\ra\cF\left(
\figins{-15}{0.5}{moy5-22}\right),
$$
$$\psi_0:\cF\left(
\figins{-15}{0.5}{moy5-11}\right)\ra 
\cF\left(
\figins{-15}{0.5}{moy5-1}\right),
\qquad 
\psi_1:\cF\left(
\figins{-15}{0.5}{moy5-22}\right)
\ra\cF\left(
\figins{-15}{0.5}{moy5-1}\right),
$$
by
$$
\varphi_0 = -\cF\left(\figins{-19}{0.6}{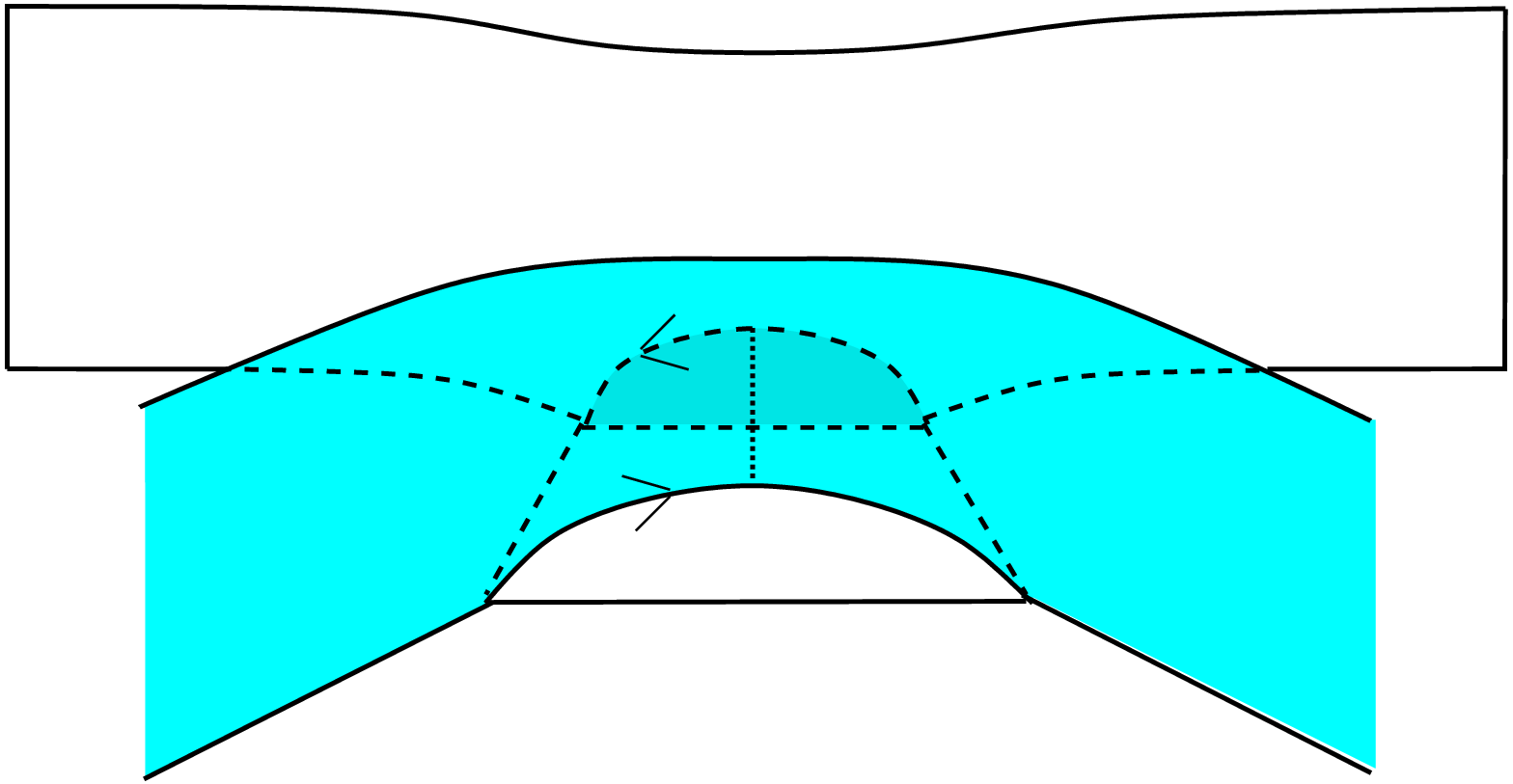}\right),\qquad
\varphi_1 = -\cF\left(\figins{-19}{0.6}{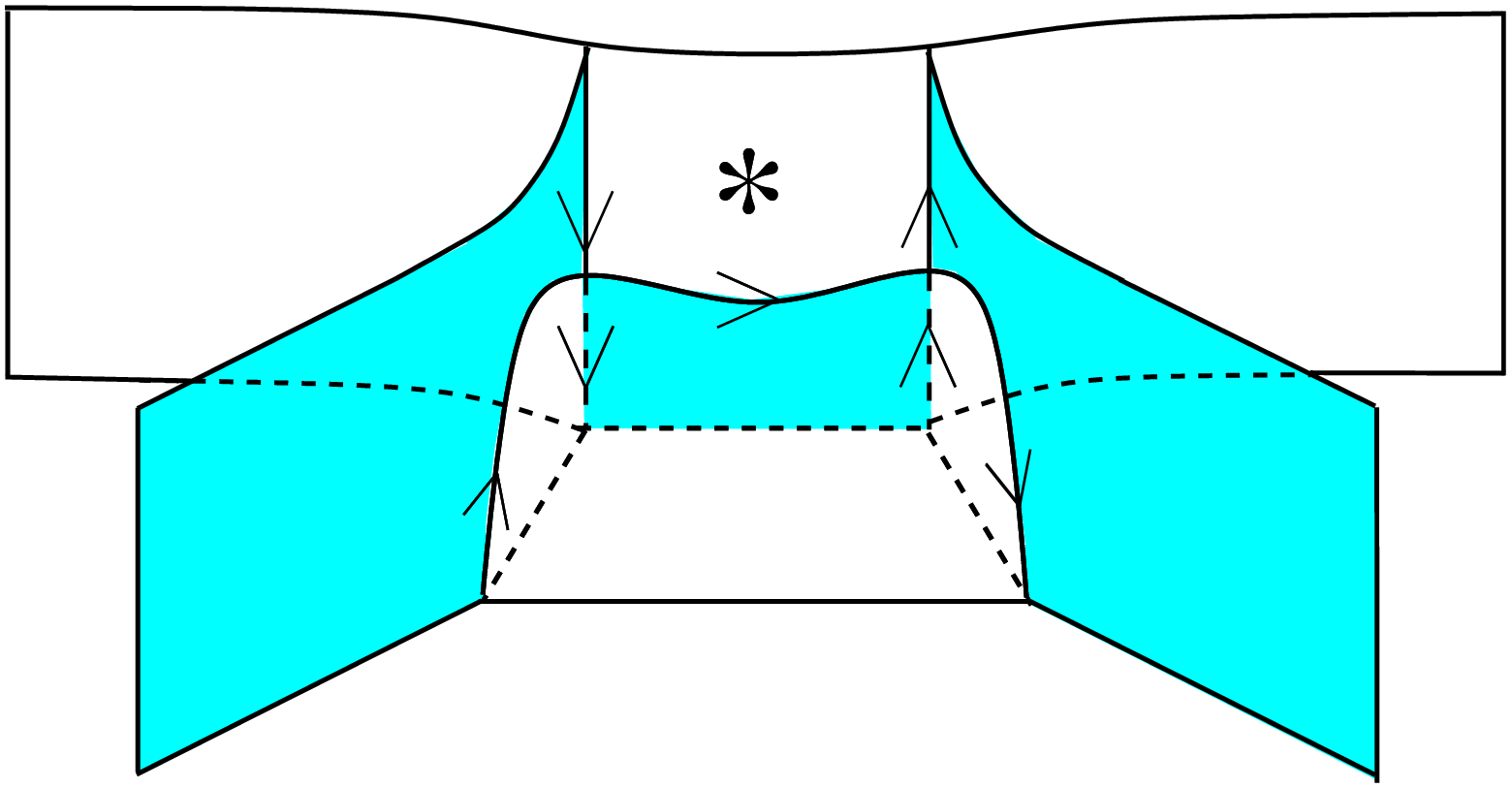}\right), 
$$
$$
\psi_0 = \cF\left(\figins{-19}{0.6}{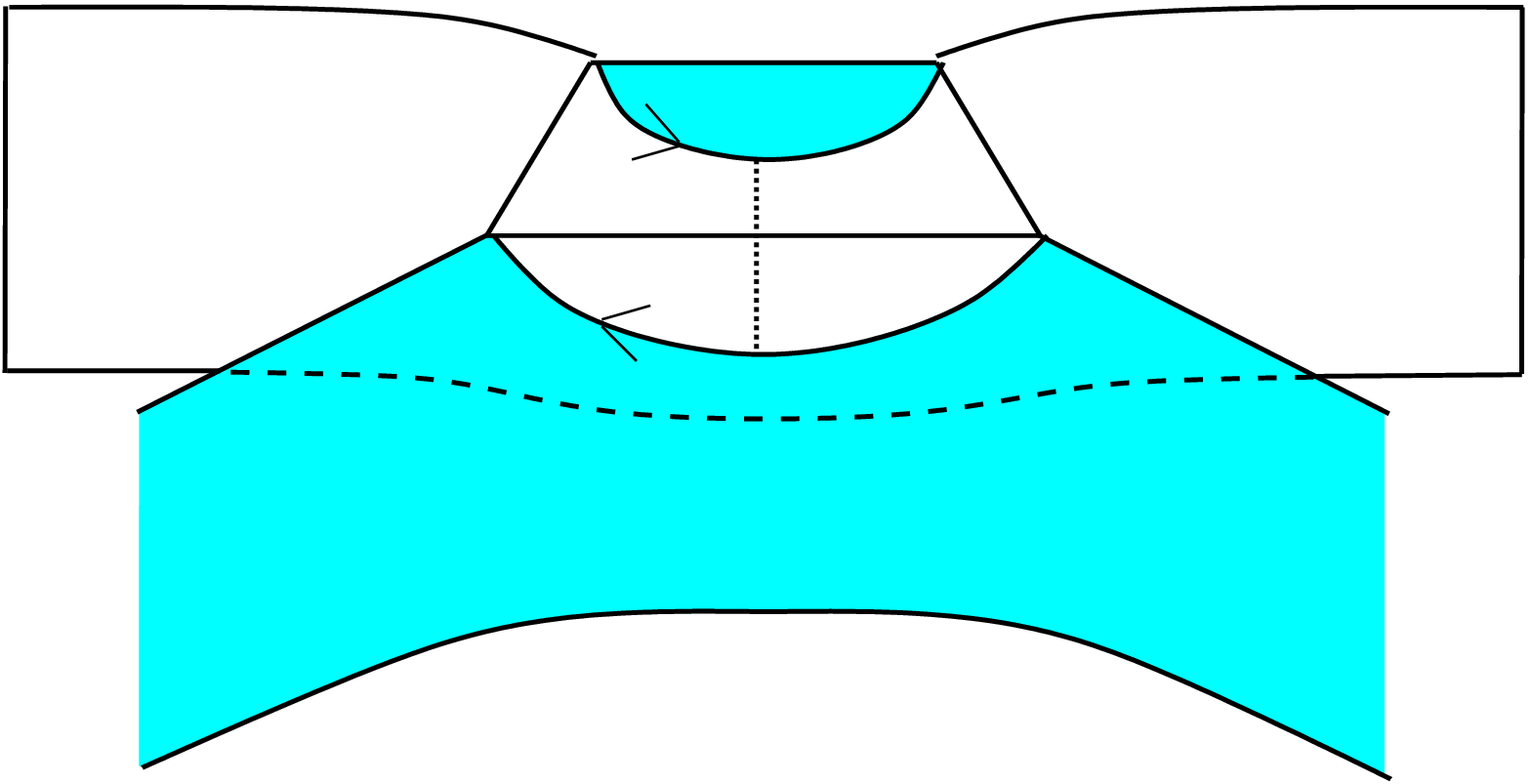}\right),\qquad
\psi_1 = \cF\left(\figins{-19}{0.6}{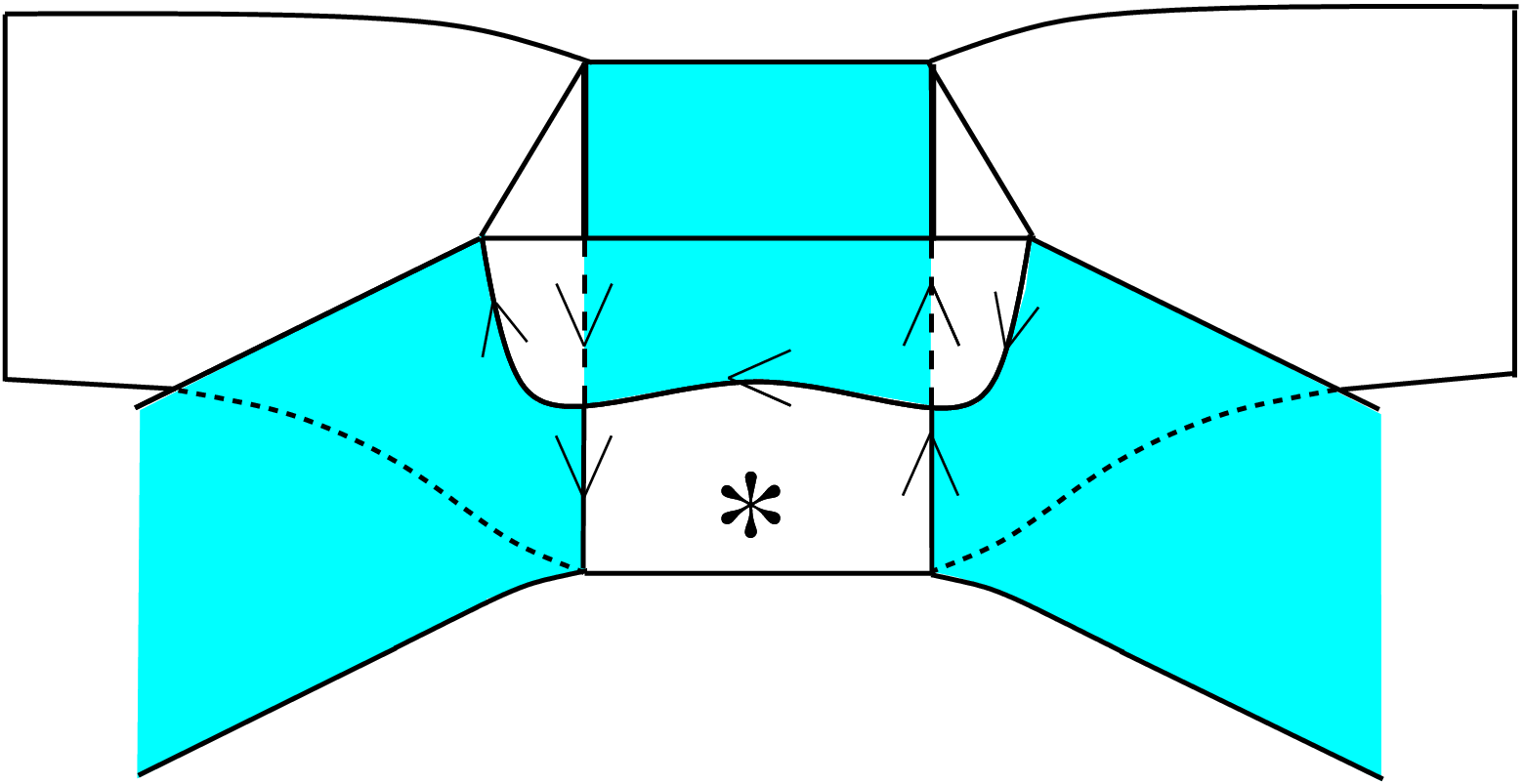}\right).
$$

We have that $\varphi_i\psi_j=\delta_{i,j}$ for $i,j = 0,1 $ (we leave the details to the reader). From the (\emph{SqR$_2$}) relation it follows that $$\psi_0\varphi_0+\psi_1\varphi_1=
Id\left(\cF\left(\figins{-15}{0.5}{moy5-1} \right)\right).$$ Applying a symmetry to all diagrams 
in decomposition $a)$ gives us decomposition $b)$.
\end{proof}

In order to relate our construction to the $\sln$ polynomial we need to introduce shifts. We denote by $\{n\}$ an upward shift in the $q$-grading by $n$ and by $[m]$ an upward shift in the homological grading by $m$. 
\begin{defn}\label{def:gradedcomplx}
Let $\brak{L}_i$ denote the $i$-th homological degree of the complex $\brak{L}$. We define the 
$i$-th homological degree of the complex $\cF(L)$ to be  
$$\cF_i(L)=\cF\brak{L}_i[-n_-]\{ (N-1)n_+ - Nn_-  + i \},$$ 
where $n_+$ and $n_-$ denote the number of positive and negative crossings in the diagram used to 
calculate $\brak{L}$.
\end{defn}

We now have a homology functor $\mathbf{Link}\ra\V$ which 
we still call $\cF$. Definition~\ref{def:gradedcomplx}, 
Theorem~\ref{thm:inv} and Lemma~\ref{lem:moy-F} imply that
\begin{thm}
For a link $L$ the graded Euler characteristic of 
$H^*\left(\mathcal{F}(L)\right)$ equals $P_N(L)$, the $\sln$ polynomial of $L$. 
\end{thm}

The MOY-relations are also the last bit that we need in 
order to show the following theorem. 
\begin{thm} For any link $L$, the bigraded complex $\mathcal{F}(L)$ is 
isomorphic to the Khovanov-Rozansky complex $KR(L)$ in \cite{KR}.  
\end{thm}
\begin{proof} The map $\langle\ \vert_{KL}$ 
defines a grading preserving linear injection 
$\mathcal{F}(\Gamma)\to KR(\Gamma)$, for 
any web $\Gamma$. Lemma~\ref{lem:moy-F} implies that the 
graded dimensions of $\mathcal{F}(\Gamma)$ and 
$KR(\Gamma)$ are equal, so $\langle\ \vert_{KL}$ is 
a grading preserving linear isomorphism, for any web $\Gamma$.   

To prove the theorem we would have to show that 
$\langle\ \vert_{KL}$ commutes 
with the differentials. We call  
$$
\figins{0}{0.7}{ssaddle} \qquad
$$
the {\em zip} and 
$$
\figins{0}{0.7}{ssaddle_ud} \qquad
$$
the {\em unzip}. Note that both the zip and the unzip 
have $q$-degree $1$. Let $\Gamma_1$ be the source web of the zip and 
$\Gamma_2$ its target web, and let $\Gamma$ be the theta web, which is 
the total boundary of the zip where the vertical edges have been pinched. 
The $q$-dimension of 
$\Ext(\Gamma_1,\Gamma_2)$ is equal to 
$$q^{2N-2}\qdim(\Gamma)=q+q^2(\ldots),$$
where $(\ldots)$ is a polynomial in $q$. Therefore the differentials in the 
two complexes 
commute up to a scalar. By the removing disc relation (RD$_1$) we see that if the ``zips'' 
commute up to $\lambda$, then the ``unzips'' commute up to $\lambda^{-1}$. If $\lambda\ne 1$, 
we have to modify our map between the two complexes slightly, in order to get an honest morphism of complexes. We 
use Khovanov's idea of ``twist equivalence'' in 
\cite{khovanovfrob}. For a given 
link consider the hypercube of resolutions. If an arrow in the hypercube 
corresponds to a zip, multiply $\langle(\mbox{target})\vert_{KL}$ by $\lambda$, where 
target means the target of the arrow. 
If it corresponds to an unzip, multiply $\langle(\mbox{target})\vert_{KL}$ by $\lambda^{-1}$.
This is well-defined, because all squares in the hypercube (anti-)commute. By definition this 
new map commutes with the differentials and therefore proves that the two complexes are 
isomorphic.   
\end{proof}

We conjecture that the above isomorphism actually extends to link cobordisms, 
giving a projective natural isomorphism between the two projective link 
homology functors. Proving this would require a careful comparison between the 
two functors for all the elementary link cobordisms.


\vspace*{1cm}

\noindent {\bf Acknowledgements} The authors thank Mikhail Khovanov for 
helpful comments and Lev Rozansky for patiently explaining the 
Kapustin-Li formula in detail to us. 

The authors were supported by the 
Funda\c {c}\~{a}o para a Ci\^{e}ncia e a Tecnologia (ISR/IST plurianual funding) through the
programme ``Programa Operacional Ci\^{e}ncia, Tecnologia, Inova\-\c
{c}\~{a}o'' (POCTI) and the POS Conhecimento programme, cofinanced by the European Community 
fund FEDER.


\end{document}